\documentclass[10pt,a4paper]{amsart}
\usepackage{dhucs}

\usepackage{amsmath,amsbsy}
\usepackage{amssymb}
\usepackage{amsfonts}
\usepackage{ mathrsfs }
\usepackage{stmaryrd}
\usepackage{stmaryrd}
\usepackage[all]{xy}
\usepackage{kotex}
\usepackage[pdftex]{color, graphicx}
\usepackage{mathdots}
\usepackage{tikz}
\usepackage{hyperref}
\usepackage{marginnote}
\usepackage{mathabx}

\usepackage{xcolor}

\newtheorem{proposition}{Proposition}[section]
\newtheorem{theorem}[proposition]{Theorem}
\newtheorem{lemma}[proposition]{Lemma}
\newtheorem{corollary}[proposition]{Corollary}
\theoremstyle{definition}

\newtheorem{definition}[proposition]{Definition}

\newtheorem{notation}[proposition]{Notation}

\theoremstyle{remark}
\newtheorem{remark}[proposition]{Remark}

\newcommand{\R}{\mathbb{R}}

\newcommand{\Z}{\mathbb{Z}}

\newcommand{\Hom}{\operatorname{Hom}}

\newcommand{\SL}{\mathrm{SL}}
\newcommand{\PSL}{\mathrm{PSL}}
\newcommand{\Tr}{\operatorname{Tr}}
\newcommand{\GL}{\mathrm{GL}}

\renewcommand{\phi}{\varphi}

\newcommand{\Ad}{\operatorname{Ad}}

\newcommand{\Rep}{\mathcal{X}}

\renewcommand{\epsilon}{\varepsilon}
\renewcommand{\Tr}{\operatorname{Tr}}
\newcommand{\der}{\mathrm{d}}

\newcommand{\per}[1]{\operatorname{per}(#1)}

\begin{document}
\title[Symplectic coordinates]{Symplectic coordinates on the deformation spaces of convex projective structures on 2-orbifolds}
\author{Suhyoung Choi}\thanks{Suhyonug Choi is supported in part by NRF grant 2019R1A2C108454412}
\address{Department of Mathematical Sciences, KAIST,
	Daejeon 34141, Republic of Korea}
\email{schoi@math.kaist.ac.kr}

\author{Hongtaek Jung}\thanks{Hongtaek Jung is supported by IBS-R003-D1.}
\address{Center for Geometry and Physics, Institute for Basic Science (IBS),
	Pohang 37673, Republic of Korea}
\email{htjung@ibs.re.kr}

\maketitle
\begin{abstract}
Let $\mathcal{O}$ be a closed orientable 2-orbifold of negative Euler characteristic.  Huebschmann constructed the Atiyah-Bott-Goldman type symplectic form $\omega$ on the deformation space $\mathcal{C}(\mathcal{O})$ of convex projective structures on $\mathcal{O}$. We show that the deformation space $\mathcal{C}(\mathcal{O})$ of convex projective structures on $\mathcal{O}$ admits a global Darboux coordinates system with respect to $\omega$. To this end, we show that $\mathcal{C}(\mathcal{O})$ can be decomposed into smaller symplectic spaces. In the course of the proof, we also study the deformation space $\mathcal{C}(\mathcal{O})$ for  an orbifold $\mathcal{O}$ with boundary and construct the symplectic form on the deformation space of convex projective structures on $\mathcal{O}$ with fixed boundary holonomy. 
\end{abstract}

\section{Introduction}
\subsection{Motivation}
Let $\mathcal{O}$ be a closed orientable 2-orbifold of negative Euler characteristic.  In this paper we study the deformation space $\mathcal{C}(\mathcal{O})$ of convex projective structures on $\mathcal{O}$ and its symplectic nature. 

It is a classical result that the Teichmüller space $\mathcal{T}(S)$ of a given orientable closed hyperbolic surface $S$ with genus $g$ admits a global Darboux coordinates system \cite{wolpert1985}.  In this case the Fenchel-Nielsen coordinates system is the one where the Weil-Petersson symplectic form $\omega_{WP}$ can be written in the standard form:
\[
\omega_{WP} = \sum_{i=1} ^{3g-3} \der \ell_i \wedge \der \theta_i. 
\]

The Teichmüller space $\mathcal{T}(S)$ for a closed surface $S$ can be  generalized in many directions. One way to do this starts from regarding $\mathcal{T}(S)$ as the set of conjugacy classes of discrete faithful representations of $\pi_1(S)$ into $\operatorname{PSL}_2(\R)$. Since there is a unique (up to conjugation) irreducible representation $\operatorname{PSL}_2(\R)\to \operatorname{PSL}_n(\R)$ for each $n$, we can canonically embed $\mathcal{T}(S)$ into $\Hom(\pi_1(S), \operatorname{PSL}_n(\R))/\operatorname{PSL}_n(\R)$. We call the image of this embedding the Fuchsian locus and call each element in the Fuchsian locus a Fuchsian character. The Hitchin component $\operatorname{Hit}_n(S)$ is then defined by a connected component of $\Hom(\pi_1(S), \operatorname{PSL}_n(\R))/\operatorname{PSL}_n(\R)$ that contains a Fuchsian character. This ``higher Teichmüller space'' has been studied by many people including Hitchin \cite{hitchin1992}, Goldman \cite{goldman1990, goldman1984}, Choi-Goldman \cite{choi1993}, Labourie \cite{labourie2006}, and Guichard-Wienhard \cite{guichard2012}. 

Due to Goldman \cite{goldman1984}, we know that the Hitchin components also carry the symplectic form so called Atiyah-Bott-Goldman form $\omega_{ABG}$ which is a natural generalization of $\omega_{WP}$. Their symplectic nature has been extensively studied  by Sun-Zhang \cite{sz2017}, Sun-Zhang-Wienhard \cite{swz2017} and they proved that $(\operatorname{Hit}_n(S), \omega_{ABG})$ admits a global Darboux coordinates system. Using a different method, the authors \cite{choi2019} also obtained the same result for  $\operatorname{Hit}_3(S)$. 

Another direction is to generalize a surface $S$ to an orbifold $\mathcal{O}$. Thurston's lecture note \cite{thurston1979}, for instance, contains some results about the Teichmüller spaces of orbifolds.  

In this paper we proceed in both directions and study symplectic nature of the ``Hitchin components'' of orbifolds via convex projective geometry. Our approach is geometric  and works mostly for $\operatorname{PSL}_3(\R)$. Recently Alessandrini-Lee-Schaffhauser \cite{alessandrini2018} gave a generalization of Hitchin components of orbifolds by using equivariant Higgs bundle theory.

\subsection{Statement of results}
Let $\mathcal{O}$ be a compact oriented  2-orbifold of negative Euler characteristic, and  $\Gamma:=\pi_1^{\operatorname{orb}}(\mathcal{O})$ its orbifold fundamental group.   Denote by $\mathcal{C}(\mathcal{O})$  the deformation space of convex projective structures on $\mathcal{O}$ with principal geodesic boundary. 

We can relate $\mathcal{C}(\mathcal{O})$ to an algebraic object. For this, we define the set of so-called good representations
\[
\Hom_{\operatorname{s}}(\Gamma, \operatorname{PSL}_{n}(\R)):=\{\rho\in \Hom(\Gamma, \operatorname{PSL}_{n}(\R))\,|\, Z(\rho)=\{1\} \text{ and }\rho\text{ is irreducible}\},
\]
where $Z(\rho)$ denotes the centralizer of $\rho$. Our target object is  the smooth part of the character variety $\Rep_{n} (\Gamma)$, which is defined by 
\[
\Rep_{n} (\Gamma) := \Hom_{\operatorname{s}}(\Gamma, \operatorname{PSL}_{n}(\R))/\operatorname{PSL}_n(\R).
\]
It is known that $\Rep_{n}(\Gamma)$ is a smooth manifold and $\Rep_3(\Gamma)$ contains $\mathcal{C}(\mathcal{O})$  as an open subspace.  

Suppose that  $\mathcal{O}$ has $b>0$ boundary components. For each oriented boundary component $\zeta_i$, choose an element $z_i\in \pi_1^{\operatorname{orb}}(\mathcal{O})$ that is freely homotopic to $\zeta_i$. We call such $z_i$ a \emph{primitive peripheral element} for $\zeta_i$.  An element in $\PSL_n(\R)$ is \emph{hyperbolic} if it lifts to $\SL_n(\R)$ a diagonalizable matrix with $n$ distinct positive eigenvalues. We also choose a set of $\PSL_n(\R)$-conjugacy classes $\mathscr{B}=\{B_1, \cdots, B_b\}$ of hyperbolic elements and define $\mathcal{C}^{\mathscr{B}}(\mathcal{O})$ to be a subspace of $\mathcal{C}(\mathcal{O})$ whose $i$th oriented boundary holonomy is in $B_i$.  The algebraic counterparts for $\mathcal{C}^{\mathscr{B}}(\mathcal{O})$ are the following spaces:
\[
\Hom_{\operatorname{s}} ^{\mathscr{B}}(\Gamma, \PSL_{n}(\R)):=  \{\rho\in \Hom_{\operatorname{s}} (\Gamma,  \PSL_{n}(\R))\,|\,\rho(z_i)\in B_i\}
\]
and
\[
\Rep_{n} ^{\mathscr{B}}(\Gamma):=\Hom_{\operatorname{s}} ^{\mathscr{B}}(\Gamma, \PSL_{n}(\R))/\PSL_{n}(\R). 
\]
It is also known that $\Rep_{n} ^{\mathscr{B}}(\Gamma)$ is an embedded submanifold of $\Rep_{n}(\Gamma)$ and $\mathcal{C}^{\mathscr{B}}(\mathcal{O})$ is an open  subspace of  $\Rep_3 ^{\mathscr{B}}(\Gamma)$. 

In our first main theorem, we construct a symplectic form on $\Rep_{n} ^{\mathscr{B}} (\Gamma)$. In fact we construct an equivariantly closed 2-form on a neighborhood of $\Hom_{\operatorname{s}} ^{\mathscr{B}} (\Gamma,\PSL_n(\R))$ which is reduced to a symplectic form on $\Rep_n ^{\mathscr{B}} (\Gamma)$. 
\begin{theorem}\label{sympintro}
Let $\mathcal{O}$ be a compact oriented 2-orbifold of negative Euler characteristic. Let $G=\PSL_n(\R)$. Choose  a set $\mathscr{B}$ of conjugacy classes of hyperbolic elements in $G$.  Then there is a neighborhood of $\Hom_{\operatorname{s}} ^{\mathscr{B}} (\pi_1 ^{\operatorname{orb}}(\mathcal{O}),G)$ and an equivariantly closed 2-form $\omega$ such that the restriction of $\omega$ to $\Hom_{\operatorname{s}} ^{\mathscr{B}} (\pi_1 ^{\operatorname{orb}}(\mathcal{O}),G)$ descends to a symplectic form $\omega ^{\mathcal{O}}$ on $\Rep_n ^{\mathscr{B}} (\pi_1 ^{\operatorname{orb}}(\mathcal{O}))$. In particular for $n=3$, $\omega^\mathcal{O}$ is a symplectic form on  $\mathcal{C}^{\mathscr{B}}(\mathcal{O})$.
\end{theorem}

Goldman \cite{goldman1984} studied the symplectic structure on $\Rep_n (\pi_1(S))$ for a closed hyperbolic surface $S$. After his work, Huebschmann \cite{huebschmann1995b}, Karshon \cite{karshon1992} also gave their own constructions of the symplectic form on $\Rep_n (\pi_1(S))$.  Huebschmann \cite{huebschmann1995} extended his work to $\Rep_n (\pi_1^{\operatorname{orb}} (\mathcal{O}))$ where $\mathcal{O}$ is a closed orientable hyperbolic orbifold. Guruprasad-Huebschmann-Jeffrey-Weinstein \cite{guruprasad1997} generalized the construction to $\Rep_n ^{\mathscr{B}}(\pi_1(S))$  for a compact orientable hyperbolic surface $S$ with boundary. Our work is mainly based on   \cite{goldman1984}, \cite{huebschmann1995} and  \cite{guruprasad1997}  and covers all previous cases. 

We first construct  a 2-form $\omega_{PD}$ using the Poincaré duality for the group pair $(\pi_1 ^{\operatorname{orb}} (\mathcal{O}),\mathcal{S})$ where $\mathcal{S}=\{\langle z_1 \rangle , \cdots, \langle z_b\rangle\}$ is the set of cyclic subgroups of $\pi_1 ^{\operatorname{orb}} (\mathcal{O})$ each of which is generated by the primitive peripheral element $z_i\in \pi_1 ^{\operatorname{orb}}(\mathcal{O})$. We show that $(\pi_1 ^{\operatorname{orb}} (\mathcal{O}), \mathcal{S})$ is a $PD^2_{\R}$-pair, meaning that there is a fundamental class $[\mathcal{O}, \partial \mathcal{O}]\in H_2(\pi_1 ^{\operatorname{orb}} (\mathcal{O}), \mathcal{S}; \R)$ such that the cap product $H^2(\pi_1 ^{\operatorname{orb}} (\mathcal{O}), \mathcal{S} ;\R)\to H_0(\pi_1 ^{\operatorname{orb}} (\mathcal{O});\R)=\R$ is an isomorphism. The 2-form $\omega_{PD}$ so obtained is the pull-back of the well-known Atiyah-Bott-Goldman symplectic form on $\Rep^{\mathscr{B}'}_n(\pi_1(X'_{\mathcal{O}}))$ for some compact surface $X'_{\mathcal{O}}$. 

To find the equivariantly closed 2-form on a neighborhood of $\Hom_{\operatorname{s}} ^{\mathscr{B}} (\pi_1 ^{\operatorname{orb}}(\mathcal{O}),G)$,  we use the equivariant de Rham complex and Cartan-Maurer calculus. Then we show that after taking the reduction, this 2-form coincides up to sign with the symplectic form $\omega_{PD}$ on $\Rep^{\mathscr{B}}_n(\pi_1^{\operatorname{orb}}(\mathcal{O}))$.

A $1$-dimensional suborbifold of $\mathcal{O}$ is said to be {\em full} if it is based on an arc ending at cone points of order $2$ in $\mathcal{O}$. Choose pairwise disjoint essential simple closed curves or full 1-suborbifolds  $\{\xi_1, \cdots, \xi_m\}$ such that the completion of each connected component $\mathcal{O}_i$ of $\mathcal{O}\setminus \bigcup_{i=1} ^m  \xi_i$ has also negative Euler characteristic. Say $\xi_1, \cdots, \xi_{m_1}$ are simple closed curves and $\xi_{m_1+1},\cdots, \xi_{m_1+m_2}$ are full 1-suborbifolds, so that $m=m_1+m_2$. By using our previous work \cite{choi2019}, we show that there is an Hamiltonian $\R^{M}$-action, $M=2m_1+m_2$, on $\mathcal{C}^{\mathscr{B}}(\mathcal{O})$ with a moment map $\mu$ such that the Marsden-Weinstein quotient  $\mu^{-1}(y)/\R^{M}$ exists for each $y$ in the image of $\mu$. 

For suitably chosen $\mathscr{B}_1,\cdots,\mathscr{B}_l$, there is a map
\[
\mathcal{SP}'_y: \mu^{-1}(y) / \R^{M} \to  \mathcal{C}^{\mathscr{B}_1}(\mathcal{O}_1) \times \cdots \times \mathcal{C}^{\mathscr{B}_l}(\mathcal{O}_l)
\]
that is induced from restrictions to each subgroup $\pi_1 ^{\operatorname{orb}}(\mathcal{O}_i)\subset \pi_1 ^{\operatorname{orb}}(\mathcal{O})$. Theorem \ref{sympintro} implies that the right hand side admits the symplectic form $\omega^{\mathcal{O}_1}\oplus \cdots \oplus \omega^{\mathcal{O}_l}$. Then we have the following symplectic decomposition theorem:

\begin{theorem}\label{globaldecomp}
Let $\mathcal{O}$ be a compact oriented 2-orbifold  of negative Euler characteristic. Let $\{\xi_1, \cdots, \xi_m\}$ be a set of pairwise disjoint essential simple closed curves or full 1-suborbifolds such that the completion of  each connected component $\mathcal{O}_i$ of $\mathcal{O}\setminus \bigcup_{i=1} ^m  \xi_i$ has also negative Euler characteristic. Then the map
\[
\mathcal{SP}'_y: \mu^{-1}(y) / \R^{M} \to  \mathcal{C}^{\mathscr{B}_1}(\mathcal{O}_1) \times \cdots \times \mathcal{C}^{\mathscr{B}_l}(\mathcal{O}_l)
\]
defined above is a symplectomorphism.
\end{theorem}

The essential part of Theorem \ref{globaldecomp} is that the splitting map also preserves some symplectic information. This is a consequence of the following more general local decomposition theorem, whose proof is given in section \ref{localsecompsec}.

\begin{theorem}\label{localdecomp}
Let $\mathcal{O}$ be a compact oriented 2-orbifold of negative Euler characteristic.  Let $\{\xi_1, \cdots, \xi_m\}$ be pairwise disjoint essential simple closed curves or full 1-suborbifolds such that the completions of the connected components $\mathcal{O}_1, \cdots, \mathcal{O}_l$ of $\mathcal{O}\setminus \bigcup_{i=1} ^m  \xi_i$ have also negative Euler characteristic.  For each $i$, let $\iota_i: \pi_1 ^{\operatorname{orb}}(\mathcal{O}_i) \to \pi_1 ^{\operatorname{orb}}(\mathcal{O})$ be the inclusion.  Let $[\rho]\in \Rep_n ^{\mathscr{B}}(\pi_1^{\operatorname{orb}}(\mathcal{O}))$ be such that, with a nice choice of $\mathscr{B}_i$, $[\rho\circ \iota_i ]\in \Rep_n ^{\mathscr{B}_i}(\mathcal{O}_i)$ for each $i$.  Then
\[
\omega^{\mathcal{O}}([u],[v]) = \sum_{i=1} ^{l} \omega^{\mathcal{O}_i}(\iota_i^* [u], \iota_i ^* [v])
\]
for all $[u],[v]\in H^1_{\operatorname{par}}(\pi_1^{\operatorname{orb}}(\mathcal{O}), \mathcal{S};\mathfrak{g})$ where $\mathcal{S}= \{\langle z_1\rangle, \cdots, \langle z_b \rangle, \langle \xi_1 \rangle, \cdots, \langle \xi_m \rangle \}$.\end{theorem}

Strategy for proving Theorem \ref{localdecomp}  is the same as that of our previous paper \cite{choi2019}. The essential difference is that the splitting can happen along a full 1-suborbifold. We will exhibit how the fundamental cycle of our orbifold is decomposed when we do splitting along a full 1-suborbifold. 

By following the framework of \cite{choi2019}, we can show our main theorem:

\begin{theorem}\label{mainthm}
Let $\mathcal{O}$ be a closed oriented 2-orbifold of negative Euler characteristic. Then the deformation space $\mathcal{C}(\mathcal{O})$ of convex projective structures admits a global Darboux coordinates system with respect to $\omega^{\mathcal{O}}$.
\end{theorem}
One may conceive to generalize Theorem \ref{mainthm} to $\PSL_n(\R)$-Hitchin components of orbifolds in the sense of \cite{alessandrini2018}. Indeed, some ingredients involved in Theorem \ref{mainthm} work in the generality: Theorem \ref{sympintro} and Theorem \ref{localdecomp} hold verbatim for general $\PSL_n(\R)$-Hitchin components and Theorem \ref{globaldecomp} seems to work with a minor modification on the statement. The major obstacle is, however, the lack of global symplectic coordinates on  $\PSL_n(\R)$-Hitchin components of elementary orbifolds. Choi-Goldman coordinates used in this paper are based on convex projective geometry and it seems hard to promote them to the general setting. Once we resolve this setback, we may apply our relevant tools to prove Theorem \ref{mainthm} for $\PSL_n(\R)$-Hitchin components of orbifolds.

\subsection{Organization of the paper} Section \ref{prem1} is about basic concepts related to orbifolds, convex projective structures, deformation spaces and character varieties. We also review geometric splitting operations which play an important role throughout this paper. 

In section \ref{group}, we review  the theory of group cohomology and parabolic group cohomology. To compute the parabolic group cohomology we introduce various projective resolutions. By using these resolutions and their relations, we show that  orbifold fundamental groups have a finite length projective resolution.  This fact will be intensively used in the next section. 

In section \ref{symp}, we construct the Atiyah-Bott-Goldman type symplectic form on the deformation space of convex projective structures on a given compact oriented 2-orbifold $\mathcal{O}$ of negative Euler characteristic that has only cone singularities as well as an equivariant closed 2-form  on a neighborhood of $\Hom_{\operatorname{s}} ^{\mathscr{B}} (\pi_1 ^{\operatorname{orb}}(\mathcal{O}),G)$.  Using the Poincar\'{e} duality we construct the 2-form $\omega_{PD}$ on $\Rep_n ^{\mathscr{B}}(\pi_1 ^{\operatorname{orb}}(\mathcal{O}))$ and show that this 2-form is the pull-back of some symplectic form. Then we utilizes the equivariant de Rham complex to find an  equivariantly closed 2-form $\omega_H$ on a neighborhood of $\Hom_{\operatorname{s}} ^{\mathscr{B}} (\pi_1 ^{\operatorname{orb}}(\mathcal{O}),G)$. We argue that this 2-form restricts on  $\Hom_{\operatorname{s}} ^{\mathscr{B}} (\pi_1 ^{\operatorname{orb}}(\mathcal{O}),G)$  to   $-\omega_{PD}$ after the reduction. 

Section \ref{main} is devoted to prove our main theorem. The proof goes almost parallel to that of our previous paper \cite{choi2019} for closed surface. We prove the local and global decomposition theorems.  By using this decomposition theorem we can find a fiber bundle structure whose fibers are Lagrangian.  Then apply our version of action-angle principle to get the main result.

\bigskip
\emph{Acknowledgements. } The authors would like to appreciate J. Huebschmann for helpful communication. The second author also wishes to give special thanks to Kyeongro Kim.  

\section{Orbifolds and their convex projective structures}\label{prem1}
In this section, we introduce our two main objects of interest; orbifolds and convex projective structures. 

We give a concise introduction to orbifolds in section \ref{orbifold}. We also summarize important facts on 2-orbifolds, which include the Euler characteristic, orbifold fundamental groups, classification of singularities. 

In section \ref{ccgeo}, we introduce convex projective structures and their deformation space. We also introduce the space of representations and character varieties. A  theorem of Choi-Goldman, Theorem \ref{ET}, provides the link between the deformation spaces and the character varieties. 

In the last section \ref{splittingpasting} we introduce splitting operations that split a given convex projective orbifold into smaller convex projective orbifolds. One should note that this operations are not only topological but also geometric. The splitting operations play an essential role throughout this paper. 

Most material in this section is based on Thurston \cite{thurston1979},  Choi \cite{choi2012}, Goldman \cite{goldman1990},  and Choi-Goldman \cite{choi2005}. 

\subsection{Orbifolds}\label{orbifold}
An orbifold is a generalization of a manifold that allows some singular points.  We adopt the chart and atlas style approach to orbifolds. Another more abstract definition using étale groupoid can be found, for instance, in \cite{bridson1999} and \cite{ALR}. Here, we only study 2-orbifolds.

Let $X$ be a paracompact Hausdorff topological space. An \emph{orbifold chart}  $(U,\Gamma, \phi)$ at $x \in X$ consists of an open neighborhood $U$ of $x$, a finite group $\Gamma$ acting effectively on an open subspace $\widetilde{U}$ of $\R^2$ as diffeomorphisms and a $\Gamma$-invariant map $\phi:\widetilde{U} \to U$ that induces the homeomorphism $\phi': \widetilde{U}/\Gamma \to U$. In this case, we say that $(\widetilde{U}, \Gamma)$ is a \emph{model pair} of $x$. Two orbifold charts $(U_1, \Gamma_1, \phi_1)$ and $(U_2,\Gamma_2, \phi_2)$ are said to be \emph{compatible} if each point $x\in U_1 \cap U_2$ has a chart $(V,\Gamma, \phi)$ with, for each $i=1,2$, an injective homomorphism $f_i: \Gamma \to \Gamma_i$ and a $f_i$-equivariant embedding $\psi_{i}: \widetilde{V} \to \widetilde{U}_i$ such that the following diagram
\[
 \xymatrix{
 \widetilde{V} \ar[rr] \ar[d]_{\psi_i}& & \widetilde{V}/\Gamma \ar[d] \ar[r] ^{\phi'} _{\cong}& V\ar[d]^{\text{inclusion}}\\
 \widetilde{U}_i \ar[r] & \widetilde{U}_i/ f_i(\Gamma) \ar[r] & \widetilde{U}_i/\Gamma_i \ar[r]^{\phi'_i}_{\cong} & U_i
 }
 \]
commutes for each $i=1,2$. We call the maps $f_{i}$, $i=1,2$, \emph{transition lifts}. An \emph{orbifold atlas} on $X$ is a set of compatible orbifold charts that covers $X$. An $2$\emph{-orbifold} is a paracompact Hausdorff topological space equipped with a maximal orbifold atlas.  Unless otherwise stated, we assume that the underlying topological space is connected. 

We can also define an \emph{orbifold with boundary} by allowing model pairs of the form $(\widetilde{U}, \Gamma)$ where $\widetilde{U}$ is an open subset of $\R^2 _+=\{(x_1,x_2)\in \R^2 \,|\,x_2\ge 0\}$. In this case we define the boundary $\partial \mathcal{O}$ of an orbifold $\mathcal{O}$ with boundary by the set of $x\in \mathcal{O}$ that has an orbifold chart $(U, \Gamma,\phi)$ such that  $(\widetilde{U},\Gamma)=(\mathbb{R}^2_+,\{1\})$ and $\phi(0) = x$.

\begin{remark}
For a 2-orbifold $\mathcal{O}$, the underlying space $X_{\mathcal{O}}$ is always a surface. In general, $\partial \mathcal{O}$ is not the same as $\partial X_{\mathcal{O}}$.\end{remark}

Suppose that, for a maximal orbifold atlas, each chart $\widetilde{U}$ is given an orientation and all transition lifts preserves are orientation preserving. Then the 2-orbifold is called \emph{oriented}. The collections of orientations on the models are called the \emph{orientation}.

Given two orbifolds $\mathcal{O}_1$ and $\mathcal{O}_2$, a continuous map $f:\mathcal{O}_1 \to \mathcal{O}_2$ is called an \emph{orbifold map} if, for each $x\in \mathcal{O}_1$, there is an orbifold chart $(U_1,\Gamma_1,\phi_1)$ at $x$ and $(U_2,\Gamma_2,\phi_2)$ at $f(x)$  such that $f|_{U_1}$ has an $h$-equivariant lift $\widetilde{f}:\widetilde{U}_1 \to \widetilde{U}_2$ with respect to a homomorphism $h:\Gamma_1\to \Gamma_2$.

Suppose that we are given an orbifold $\mathcal{O}$. An orbifold  $\mathcal{O}'$ with an orbifold map  $f: \mathcal{O}'\to \mathcal{O}$ is called a \emph{covering} of $\mathcal{O}$ if each point $x\in \mathcal{O}$ has an orbifold chart $(U, \Gamma, \phi)$ such that for each component $U_i$ of $f^{-1}(U)$, there is a subgroup $\Gamma_i$ of $\Gamma$ and an orbifold diffeomorphism $\psi_i: \widetilde{U}/\Gamma_i \to U_i$ such that $\phi=f\circ \psi_i \circ q_i $ where $q_i :\widetilde{U}\to \widetilde{U}/\Gamma_i$ is the quotient map.  A covering $\widetilde{\mathcal{O}}\to \mathcal{O}$ of $\mathcal{O}$ that has the universal lifting property is called a \emph{universal covering} orbifold of $\mathcal{O}$. It is known that every orbifold has a universal covering orbifold $\widetilde{\mathcal{O}}$. We call the group of deck transformations of the universal cover $\widetilde{\mathcal{O}}\to \mathcal{O}$ the \emph{orbifold fundamental group} and denote it by $\pi_1 ^{\operatorname{orb}}(\mathcal{O})$. 

Finally, we present some facts. Let $\mathcal{O}$ be an orientable 2-orbifold.
\begin{itemize}
\item Each singular point $x$ of $\mathcal{O}$ is
a \emph{cone point of order} $r$; that is, $x$ has a model pair $(\R^2, \Z/r)$ where $\Z/r$ acts on $\R^2$ as a group of rotations.
\item The underlying space $X_{\mathcal{O}}$ of the 2-orbifold $\mathcal{O}$ is always a topological manifold possibly with boundary. 
\item The 2-orbifold $\mathcal{O}$ is \emph{compact} if the underlying space $X_{\mathcal{O}}$ is compact. In this case, the set of cone points is necessarily discrete.
\item The \emph{Euler characteristic} of $\mathcal{O}$ with cone points of order $r_1, \cdots, r_{c}$ equals
\[
\chi(\mathcal{O}) = \chi(X_{\mathcal{O}})- \sum_{i=1} ^{c} \left( 1-\frac{1}{r_i}\right)
\] 
where $X_{\mathcal{O}}$ is the underlying topological space of $\mathcal{O}$. 
\item The \emph{genus} of $\mathcal{O}$ is the genus of its underlying space. 
\item We say that $\mathcal{O}$ is \emph{hyperbolic} if $\mathcal{O}$ can be realized as a quotient orbifold $\mathbb{H}^2/\Gamma$ for some discrete subgroup $\Gamma \subset \PSL_2(\R)$. Compact orientable $\mathcal{O}$ is hyperbolic  if and only $\chi(\mathcal{O})<0$. 
\end{itemize}

\subsection{Convex projective structures}\label{ccgeo}

We can equip a given orbifold with some additional structure.  Our main interest is the real projective structures on 2-orbifolds and their deformation space. This deformation space is closely related to the Hitchin component of the character variety.  Throughout this section $\mathcal{O}$ denotes a compact 2-orbifold. 

The real projective plane $\mathbb{RP}^2$ is the quotient space $(\R^3\setminus \{0\})/\sim$ where the equivalence relation is given by $x\sim y$ if and only if $x=\lambda y$ for some nonzero real number $\lambda$.  Hence each point of $\mathbb{RP}^2$ represents a 1-dimensional subspace in $\R^3$. A \emph{geodesic line} in $\mathbb{RP}^2$ represents a 2-dimensional subspace in $\R^3$. 

Recall that the projective linear group $\operatorname{PSL}_3(\R)$ acts on $\mathbb{RP}^2$. This action induces a simply transitive action on \[
(\mathbb{RP}^2)^{(4)} = \{(x_1,x_2,x_3,x_4)\in (\mathbb{RP}^2)^4\,|\, x_1, x_2,x_3,x_4 \text{ are in general position}\}.
\]

An element $g\in \operatorname{SL}_3(\R)$ is \emph{hyperbolic} if $g$ is conjugate to a matrix of the form 
\[
\begin{pmatrix}
\lambda_1 & 0 & 0 \\
0& \lambda_2 & 0 \\
0&0& \lambda_3
\end{pmatrix}
\]
with $\lambda_1>\lambda_2> \lambda_3>0$. We denote by $\mathbf{Hyp}^+$ the set of hyperbolic elements in $\operatorname{SL}_3(\R)$. 

A hyperbolic element $g$ has exactly three distinct fixed points in $\mathbb{RP}^2$.  The attracting fixed point is the one corresponding to the largest eigenvalue $\lambda_1$. The repelling fixed point is the attracting fixed point of $g^{-1}$. The remaining fixed point is called the saddle. A \emph{principal line} of a hyperbolic element $g$ is the geodesic line passing through attracting and repelling fixed points of $g$. A \emph{principal segment} of $g\in \mathbf{Hyp}^+$ is an open geodesic line segment joining the attracting fixed point and the repelling fixed point of $g$.   A closed curve $\xi$ on $\mathcal{O}$ is a {\em principal geodesic if for its lift $\widetilde{\xi}: \R \rightarrow \widetilde{\mathcal{O}}$, $\mathbf{dev}\circ \widetilde{\xi}$ develops to a principal segment of a hyperbolic element $\mathbf{h}(\hat{\xi})$ where $\hat{\xi}$ is the deck transformation corresponding to $\widetilde{\xi}$.}

Now we define what we mean by a real projective structure on $\mathcal{O}$. 
\begin{definition}\label{devpair}
A \emph{real projective structure} on a given $2$-orbifold $\mathcal{O}$ consists of 
\begin{itemize}
\item a homomorphism $\operatorname{\mathbf{h}}: \pi_1 ^{\operatorname{orb}}(\mathcal{O}) \to \operatorname{PSL}_3(\R)$ called the \emph{holonomy} and
\item a $\operatorname{\mathbf{h}}$-equivariant orbifold map $\operatorname{\mathbf{dev}}: \widetilde{\mathcal{O}} \to \mathbb{RP}^2$ called the \emph{developing map}. 
\end{itemize}
Suppose $\partial \mathcal{O}\ne\emptyset$.  A real projective structure  $(\operatorname{\mathbf{h}},\operatorname{\mathbf{dev}})$ on $\mathcal{O}$ satisfies the \emph{principal geodesic boundary condition} if each boundary component is a principal geodesic.

The pair $(\operatorname{\mathbf{h}}, \operatorname{\mathbf{dev}})$ is called the \emph{developing pair}. Two developing pairs  $(\operatorname{\mathbf{h}}, \operatorname{\mathbf{dev}})$ and $(\operatorname{\mathbf{h}}', \operatorname{\mathbf{dev}}')$ are \emph{equivalent} if there is $g\in \operatorname{PGL}_3(\R)$ such that $\operatorname{\mathbf{dev}}' = g \operatorname{\mathbf{dev}}$ and $\operatorname{\mathbf{h}}'(\cdot) = g\operatorname{\mathbf{h}}(\cdot)g^{-1}$. 
\end{definition}

One special property that we are interested in is the convexity of projective structures which is defined as follow

\begin{definition}
A real projective structure $(\operatorname{\mathbf{h}}, \operatorname{\mathbf{dev}})$ on a 2-orbifold $\mathcal{O}$ is \emph{convex} if $\operatorname{\mathbf{dev}}(\widetilde{\mathcal{O}})$ is convex domain in $\mathbb{RP}^2$, i.e., there is an affine patch of $\mathbb{RP}^2$ containing $\operatorname{\mathbf{dev}}(\widetilde{\mathcal{O}})$ and $\operatorname{\mathbf{dev}}(\widetilde{\mathcal{O}})$ is a convex subset of it. 
\end{definition}

Hence by a \emph{convex projective structure} on $\mathcal{O}$ we mean a developing pair $(\operatorname{\mathbf{h}}, \operatorname{\mathbf{dev}})$ such that $\operatorname{\mathbf{dev}}(\widetilde{\mathcal{O}})$ is convex. A \emph{convex real projective structure with principal geodesic boundary} on $\mathcal{O}$ is one that satisfies the principal geodesic boundary condition.

\begin{lemma} Let $\mathcal{O}$ be a compact orientable convex projective 2-orbifold of negative Euler characteristic.  Then  every closed essential simple closed curve $\xi$ is isotopic to a unique closed principal geodesic in $\mathcal{O}$. Moreover, $\operatorname{\mathbf{h}}(\xi)$ is a hyperbolic element. 
\end{lemma}

Now we can define the deformation space $\mathbb{RP}^2(\mathcal{O})$ of projective structures on a given compact 2-orbifold $\mathcal{O}$. We say two projective structures $(\operatorname{\mathbf{h}},\operatorname{\mathbf{dev}})$ and $(\operatorname{\mathbf{h}}',\operatorname{\mathbf{dev}}')$ are isotopic if there is a diffeomorphism $f:\widetilde{\mathcal{O}}\to \widetilde{\mathcal{O}}$ commuting with $\pi_1 ^{\operatorname{orb}}(\mathcal{O})$ such that $\operatorname{\mathbf{dev}}' = \operatorname{\mathbf{dev}} \circ f$ and $\operatorname{\mathbf{h}}' = \operatorname{\mathbf{h}}$. This isotopy equivalence and the $G$-action equivalence in Definition \ref{devpair} generate an equivalence relation. 
The quotient space is the \emph{deformation space} $\mathbb{RP}^2(\mathcal{O})$ \emph{of projective structures} on $\mathcal{O}$. 
(See Section 6.2 of \cite{choi2012} where the $G$-action and the isotopy action are 
acting left and right respectively and hence they commute.
See also Goldman \cite{goldman1987})

We also define the deformation space $\mathcal{C}(\mathcal{O})$ of convex projective structures on $\mathcal{O}$ as a subspace of $\mathbb{RP}^2(\mathcal{O})$ that consists of equivalent classes of convex projective structures  with principal geodesic boundary. 

To relate $\mathcal{C}(\mathcal{O})$ with an algebraic object, we define so called the character variety of $\Gamma:=\pi_1^{\operatorname{orb}}(\mathcal{O})$. The space of representations $\Hom(\Gamma, \operatorname{PSL}_n(\R))$ can be regarded as an affine algebraic set of $\operatorname{PSL}_n(\R)^N$ for some $N$. In general  $\Hom(\Gamma, \operatorname{PSL}_n(\R))$ may have a lot of singularities. In Corollary \ref{smoothpoint}, we will see that a sufficient condition for $\rho\in \Hom(\Gamma, \operatorname{PSL}_n(\R))$ being smooth is $H^2(\Gamma; \mathfrak{g}_\rho)=0$. Here $\mathfrak{g}_\rho$ is the Lie algebra of $\PSL_n(\R)$ regarded as a $\R\Gamma$-module with the $\Gamma$-action given by $\operatorname{Ad}\circ \rho$.  

Consider 
\[
\Hom_{\operatorname{s}} (\Gamma, \operatorname{PSL}_n(\R)):=\{\rho\in \Hom(\Gamma, \operatorname{PSL}_n(\R))\,|\, Z(\rho)=\{1\} \text{ and }\rho\text{ is irreducible}\},
\]
where 
\[
Z(\rho) := \{ g\in \PSL_n(\R)\,|\, g\rho(\gamma) = \rho(\gamma) g \text{ for all }\gamma\in \Gamma\}
\]
is the centralizer of $\rho$. This subspace $\Hom_{\operatorname{s}} (\Gamma, \operatorname{PSL}_n(\R))\subset \Hom(\Gamma, \operatorname{PSL}_n(\R))$  is nonempty and  open. The following lemma suggests that  $\Hom_{\operatorname{s}} (\Gamma, \operatorname{PSL}_n(\R))$ is smooth. 
\begin{lemma}\label{H2=0} 
$H^2(\Gamma; \mathfrak{g}_\rho)=0$ for all $\rho \in \Hom_{\operatorname{s}} (\Gamma, \operatorname{PSL}_n(\R))$.  
\end{lemma}
\begin{proof}
For the sake of convenience, we write simply $\mathfrak{g}$ in place of $\mathfrak{g}_\rho$. 

Take a finite index torsion-free normal subgroup $\Gamma'$ of $\Gamma$. Then $\Gamma'$ is the fundamental group of a surface of negative Euler characteristic, say $S$. 

Because $\Gamma'$ is a normal subgroup of finite index, we have the transfer map $\operatorname{tr}: H^* (\Gamma'; \mathfrak{g})\to H^*(\Gamma;\mathfrak{g})$ \cite[Section 6.7.16]{weibel1994}. This map is induced from the average map $\operatorname{tr}:H^0(\Gamma';\mathfrak{g}) \to H^0(\Gamma;\mathfrak{g})$ given by $X \mapsto \frac{1}{[\Gamma:\Gamma']}\sum_{\eta\in \Gamma/\Gamma'}\eta X$. It is known that 
\[
H^*(\Gamma; \mathfrak{g} ) \overset{\operatorname{res}}{\to} H^* (\Gamma'; \mathfrak{g}) \overset{\operatorname{tr}}{\to} H^*(\Gamma; \mathfrak{g})
\]
is the multiplication by the index $[\Gamma:\Gamma']$ and therefore is an isomorphism in our setting. Here, $\operatorname{res}:H^*(\Gamma; \mathfrak{g}) \to H^*(\Gamma';\mathfrak{g})$ is the restriction map induced from the inclusion $\Gamma'\to \Gamma$. Regarding this fact, we refer the readers to Weibel  \cite[Lemma 6.7.17]{weibel1994}. In particular $\operatorname{res}$ is injective and $\operatorname{tr}$ is surjective.

If $S$ has nonempty boundary, then $\Gamma'$ is a free group and we have $H^2(\Gamma';\mathfrak{g})=0$. Since  $\operatorname{res}: H^2(\Gamma;\mathfrak{g})\to H^2(\Gamma';\mathfrak{g})=0$ is injective, $H^2(\Gamma;\mathfrak{g})=0$ as well. 

Suppose now that $S$ is closed. Since $\mathfrak{g}$ admits an $\Ad$-invariant nondegenerate bilinear form, we have the isomorphism
\[
D: H^2(\Gamma';\mathfrak{g}) \to H^0( \Gamma';\mathfrak{g})^*\cong H^0(\Gamma';\mathfrak{g})
\]
arisen from the Poincar\'{e} duality. Recall that the restriction map $\operatorname{res}:H^2(\Gamma;\mathfrak{g})\to H^2(\Gamma';\mathfrak{g})$ is the same as the edge map $H^2(\Gamma;\mathfrak{g})\to H^0(\Gamma/\Gamma';H^2(\Gamma';\mathfrak{g}))$ in the Lyndon-Hochschild-Serre  spectral sequence. This shows that the image of the restriction map in $H^2(\Gamma';\mathfrak{g})$  is the set of fixed points of $H^2(\Gamma';\mathfrak{g})$ under the action of the finite group $\Gamma/\Gamma'$. It follows that  
\[
H^2(\Gamma;\mathfrak{g})\overset{\operatorname{res}}{\to} H^2(\Gamma'\mathfrak{g}) \overset{D}{\to} H^0(\Gamma/\Gamma' ; H^0(\Gamma';\mathfrak{g}))\subset H^0(\Gamma';\mathfrak{g})
\]
is injective. The Lyndon-Hochschild-Serre spectral sequence yields
\[
H^0(\Gamma/\Gamma' ; H^0(\Gamma';\mathfrak{g}))\cong H^0(\Gamma;\mathfrak{g}).
\]
Because $H^0(\Gamma;\mathfrak{g})=0$ by the assumption one can conclude that $H^2(\Gamma;\mathfrak{g})=0$.  
\end{proof}

\begin{remark}
Lemma \ref{H2=0} is also proven in Porti's paper \cite[Proposition 3.4]{porti2020} with somewhat different techniques which appeared after this manuscript. 
\end{remark}

Now we take the quotient of $\Hom_{\operatorname{s}}(\Gamma,\PSL_n(\R))$ by the conjugation action of $\PSL_n(\R)$. This action of $\operatorname{PSL}_n(\R)$ on  $\Hom_{\operatorname{s}} (\Gamma, \operatorname{PSL}_n(\R))$ is free as $Z(\rho)=\{1\}$ for all $\rho\in \Hom_{\operatorname{s}} (\Gamma, \operatorname{PSL}_n(\R))$.  Moreover, due to Kim \cite[Lemma 1]{kim2001},  the action on $\Hom_{\operatorname{s}} (\Gamma, \operatorname{PSL}_n(\R))$ is  proper.  This allows us to define the smooth manifold
\[
\Rep_n (\Gamma) := \Hom_{\operatorname{s}} (\Gamma, \operatorname{PSL}_n(\R))/\operatorname{PSL}_n(\R).
\]

The space $\Hom_{\operatorname{s}}(\pi_1 ^{\operatorname{orb}}(\mathcal{O}),\PSL_3(\R))$ contains special elements. Let $\mathcal{O}$ be a closed orientable hyperbolic 2-orbifold.  A \emph{Fuchsian representation} is a representation that is of the form $\pi_1 ^{\operatorname{orb}}(\mathcal{O}) \overset{\rho_F}{\to} \operatorname{PO}(2,1)\hookrightarrow \operatorname{PSL}_3(\R)$ where $\rho_F:\pi_1 ^{\operatorname{orb}}(\mathcal{O}) \to \operatorname{PO}(2,1)$ is a discrete faithful representation  that corresponds to an element of the Teichmüller space. We denote by $\mathcal{C}_T(\pi_1 ^{\operatorname{orb}}(\mathcal{O}))$ the connected component of $\Hom_{\operatorname{s}} (\pi_1 ^{\operatorname{orb}}(\mathcal{O}), \operatorname{PSL}_3(\R))$ that contains the Fuchsian representations. 

Now assume that $\partial \mathcal{O} \ne \emptyset$. We define Fuchsian representations as we did for the closed orbifolds case.  Then $\mathcal{C}_T(\pi_1 ^{\operatorname{orb}}(\mathcal{O}))$ is defined to be the set of representations in $\Hom_{\operatorname{s}}(\pi_1 ^{\operatorname{orb}}(\mathcal{O}), \operatorname{PSL}_3(\R))$ that can be continuously deformed to a Fuchsian through a path $\rho_t$ such that $\rho_t(\zeta)$ is hyperbolic for all $t$ and for all boundary component $\zeta$. 

\begin{theorem}[Choi-Goldman \cite{choi2005}] \label{ET} Let $\mathcal{O}$ be a compact orientable 2-orbifold of negative Euler characteristic. Then the holonomy map 
\[
\operatorname{\mathbf{Hol}}: \mathcal{C}(\mathcal{O}) \to \mathcal{C}_T(\pi_1 ^{\operatorname{orb}}(\mathcal{O}))/\PSL_3(\R)
\]
mapping an equivalent class $[(\operatorname{\mathbf{h}}, \operatorname{\mathbf{dev}})]$ of developing pair of convex real projective structure with principal geodesic boundary to the conjugacy class $[\operatorname{\mathbf{h}}]$ of its  holonomy representation is a homeomorphism onto its image. 
\end{theorem}

\begin{remark}
One can say that $\mathcal{C}_T(\pi_1 ^{\operatorname{orb}}(\mathcal{O}))/\PSL_3(\R)$ is the Hitchin component for an orbifold $\mathcal{O}$. Another approach to this object is recently made by Alessandrini-Lee-Schaffhauser \cite{alessandrini2018}.
\end{remark}

In light of Theorem \ref{ET}, we identify $\mathcal{C}(\mathcal{O})$ with its image in $\mathcal{C}_T(\pi_1 ^{\operatorname{orb}}(\mathcal{O}))/\PSL_3(\R)$ from now on. By pulling back the smooth structure of $\mathcal{C}_T(\pi_1 ^{\operatorname{orb}}(\mathcal{O}))/\PSL_3(\R)$ by the map $\operatorname{\mathbf{Hol}}$, we can equip $\mathcal{C}(\mathcal{O})$ with the smooth structure. 

Suppose that $\mathcal{O}$ has $b$ boundary components. Recall that an elements in $\pi_1^{\operatorname{orb}}(\mathcal{O})$ is called \emph{primitive peripheral} if it is freely homotopic to a oriented boundary component of $\mathcal{O}$. Choose primitive peripheral elements $z_1,\cdots, z_b\in \pi_1 ^{\operatorname{orb}}(\mathcal{O})$ representing each boundary component. Let $\Gamma:= \pi_1 ^{\operatorname{orb}}(\mathcal{O})$. Let $\mathscr{B}=\{B_1, \cdots, B_b\}$ be a set of conjugacy classes of hyperbolic elements in $\operatorname{PSL}_n(\R)$, i.e., which can be lifted to diagonalizable matrices in $\SL_n(\R)$ with $n$ distinct positive eigenvalues. For later use, we define 
\[
\Hom_{\operatorname{s}} ^{\mathscr{B}}(\Gamma, \PSL_n(\R)) := \{ \rho\in \Hom_{\operatorname{s}} (\Gamma, \PSL_n(\R)\,|\, \rho(z_i) \in B_i \text{ for }i=1,2,\cdots, b\}
\]
and
\[
\Rep_n ^{\mathscr{B}}(\Gamma) :=\Hom_{\operatorname{s}} ^{\mathscr{B}}(\Gamma, \PSL_n(\R))/\operatorname{PSL}_n(\R). 
\]
We also define $\mathcal{C}^{\mathscr{B}}(\mathcal{O})$ in the same manner. Then $\mathcal{C}^{\mathscr{B}}(\mathcal{O})$ is a submanifolds of $\Rep_3 ^\mathscr{B}(\Gamma)$.  

We end this section by mentioning the following lemma. 

\begin{lemma}\label{irreducible}
Let $\mathcal{O}$ be a compact orientable 2-orbifold of negative Euler characteristic. The holonomy $\rho$ of $\mathcal{C}(\mathcal{O})$  is strongly irreducible and $Z(\rho) =\{1\}$. 
\end{lemma}
\begin{proof}  
We note that every finite index subgroup $\Gamma$ of $\pi_1 ^{\operatorname{orb}}(\mathcal{O})$ admits a further finite index normal subgroup which is isomorphic to a surface group say $\Gamma'$. Since  $\rho|_{\Gamma'}$ is in the $\PSL_3(\R)$-Hitchin component,  it is irreducible. Therefore $\rho|_{\Gamma}$ itself must be irreducible. 

The second assertion follows from a modified version of Schur's lemma: If $g\in \GL_n(\R)$ is a automorphism on an irreducible $\pi_1^{\operatorname{orb}}(\mathcal{O})$-module that has at least one eigenvector then $g$ is a scalar. Note that, when $n=3$, we have $\SL_3(\R) = \PSL_3(\R)$.
\end{proof}

\begin{remark}Schur's Lemma used in the proof of Lemma \ref{irreducible} reveals that if $n$ is odd, the Lie group $G=\PSL_n(\R)$ has the property CI in the sense of Sikora \cite{sikora2012}. Namely, for any irreducible subgroup $H\le G$, the centralizer $Z(H)$ is the same as the center of $G$. Consequently, we may drop the condition $Z(\rho)=\{1\}$ in the definition of $\Hom_{\operatorname{s}}(\Gamma, \PSL_n(\R))$ when $n$ is odd.\end{remark}

\subsection{Splitting convex projective 2-orbifolds}\label{splittingpasting}
Let $\mathcal{O}$ be a convex projective compact oriented 2-orbifold with genus $g$, $b$ boundary components and $c$ cone points of orders $r_1, \cdots, r_{c}$. Let $[\mathbf{h}]$ be a (conjugacy class) of the holonomy representation for the convex projective structure on $\mathcal{O}$.

Choose a canonical presentation 
\begin{equation*}
\Gamma:= \pi_1 ^{\text{orb}}(\mathcal{O}) = \langle x_1, y_1, \cdots, x_g, y_g, z_1, \cdots, z_b, s_1, \cdots, s_{c}\,|\,\mathbf{r}=  \mathbf{r}_1 = \cdots =\mathbf{r}_{c}=1\rangle
\end{equation*}
where
\[
\mathbf{r}=\prod_{i=1} ^g [x_i,y_i] \prod_{j=1} ^b z_j \prod_{k=1} ^{c} s_k
\]
and where $\mathbf{r}_i= s_i ^{r_i}$, $i=1,2,\cdots, c$. We also choose conjugacy classes $\mathscr{B}=\{B_1, \cdots, B_b\}$ and consider $\mathcal{C}^{\mathscr{B}}(\mathcal{O})$ the deformation space of convex projective structures on $\mathcal{O}$ with fixed boundary holonomies, $\operatorname{\mathbf{h}}(z_i)\in B_i$, $i=1,2,\cdots, b$.

\subsubsection{Splitting along a full 1-suborbifold}
Suppose that $\mathcal{O}$ has $c_b\ge 2$ numbers of order two cone points. Let $s_1, \cdots, s_{c_b}$ be the corresponding order two generators of $\pi_1 ^{\operatorname{orb}}(\mathcal{O})$.  A \emph{full 1-suborbifold} is an embedded 1-suborbifold which is a segment with both endpoints being mirror points.  Topologically a full 1-suborbifold is the image under the projection $\widetilde{\mathcal{O}}\to\mathcal{O}$ of a segment that contains exactly two fixed points of order 2 elements of $\pi_1^{\operatorname{orb}}(\mathcal{O})$ at its endpoints. Given a convex projective structure $(\mathbf{h},\mathbf{dev})$ on $\mathcal{O}$ with principal geodesic boundary, a full 1-suborbifold $\xi$ in $\mathcal{O}$ is \emph{principal} if it has a lift $\widetilde{\xi}$ in $\widetilde{\mathcal{O}}$ such that $\mathbf{dev}(\widetilde{\xi})$ is a principal segment for some hyperbolic element in $\mathbf{h}(\pi_1 ^{\operatorname{orb}}(\mathcal{O}))$.

\begin{lemma}[see also Section 3.3 of Choi-Goldman \cite{choi2005}]\label{full1sub}
Let $\mathcal{O}$ be a compact orientable 2-orbifold with negative Euler characteristic. Let $(\operatorname{\mathbf{h}},\operatorname{\mathbf{dev}})$ be a convex projective structure on $\mathcal{O}$ with principal geodesic boundary. Let $s_1$ and $s_2$ be distinct order two elements of $\pi_1 ^{\operatorname{orb}}(\mathcal{O})$. Then, $\operatorname{\mathbf{h}}(s_1s_2)$ is hyperbolic and is conjugate to a matrix of the form 
\[
\begin{pmatrix}
\lambda & 0 & 0 \\
0& 1 & 0 \\
0&0& \lambda^{-1}
\end{pmatrix}
\]
for some $\lambda>1$.
\end{lemma}
\begin{proof}
Let $\Omega=\mathbf{dev}(\widetilde{\mathcal{O}})$ be a convex domain in $\mathbb{RP}^2$ associated to $(\operatorname{\mathbf{h}},\operatorname{\mathbf{dev}})$. The set of fixed points of an order two rotation consists of a projective geodesic and an isolated point away from the geodesic. Let $l_i$ be the fixed geodesic of $\operatorname{\mathbf{h}}(s_i)$ and $p_i$ be the isolated fixed point of $\operatorname{\mathbf{h}}(s_i)$ respectively. Observe that $l_1$ and $l_2$ are disjoint from the closure of $\Omega$, meanwhile $p_1$ and $p_2$ are in $\Omega$. Moreover, we know that $l_1\ne l_2$ and $p_1\ne p_2$, otherwise, $\operatorname{\mathbf{h}}(s_1)$ and $\operatorname{\mathbf{h}}(s_2)$ become identical. The two geodesics $l_1$ and $l_2$ must intersect in $\mathbb{RP}^2$ at a single point say $q_0$.  This $q_0$ is a fixed point of $\operatorname{\mathbf{h}}(s_1 s_2)$ with eigenvalue 1. Now draw a geodesic line $l$ joining $p_1$ and $p_2$. Because $\Omega$ is convex, this $l$ intersects $\partial\Omega$ at exactly two points, say $q_1$ and $q_2$. Since $l$ is invariant under $\operatorname{\mathbf{h}}(s_1s_2)$, $q_1$ and $q_2$ are fixed points of $\operatorname{\mathbf{h}}(s_1 s_2)$. This shows that $\operatorname{\mathbf{h}}(s_1 s_2)$ has three distinct eigenvectors, hence, the lemma follows.
\end{proof}

Given order 2 cone points $p_1$ and $p_2$, a geodesic full 1-suborbifold joining $p_1$ and $p_2$ is principal because a geodesic segment joining the isolated fixed points of $\mathbf{h}(s_1)$ and $\mathbf{h}(s_2)$ lies on a principal segment of $\mathbf{h}(s_1s_2)$ as we can see from the proof of Lemma \ref{full1sub}.

Suppose that $\xi$ is any full 1-suborbifold between order two cone points $p_1$ and $p_2$. Due to  Lemma 4.1 of Choi-Goldman \cite{choi2005} we can isotope $\xi$ to  a principal full 1-suborbifold provided $\mathcal{O}\setminus \xi$ has negative Euler characteristic. We use the same letter $\xi$ to denote this principal full 1-suborbifold. Then the convex projective structure with principal boundary condition on $\mathcal{O}$  induces the canonical convex projective structure with principal boundary condition on the completion of $\mathcal{O}\setminus \xi$. We thus get the \emph{splitting map} $\mathcal{SP}: \mathcal{C}(\mathcal{O})\to \mathcal{C}(\mathcal{O}\setminus \xi)$ that equips $\mathcal{O}\setminus \xi$ with this canonical convex projective structure. 

This map can be best understood in terms of its effect on the holonomy representation. To describe this, we observe that $\pi_1 ^{\operatorname{orb}}(\mathcal{O}\setminus \xi)$ is a subgroup of $\pi_1 ^{\operatorname{orb}}(\mathcal{O})$ and $\pi_1 ^{\operatorname{orb}}(\mathcal{O}\setminus \xi)$ admits a presentation 
\begin{multline*}
\langle x_1, y_1, \cdots, x_g, y_g, z_1, \cdots,z_b, z_{b+1}, s_3,s_4, \cdots, s_{c}\,|\,\prod_{i=1} ^g [x_i,y_i] \prod_{i=1} ^{b+1} z_i \prod_{i=3} ^{c} s_i , \\s_3^{r_3}=\cdots = s_{c} ^{r_{c}}=1\rangle
\end{multline*}
where the inclusion $\iota:\pi_1 ^{\operatorname{orb}}(\mathcal{O}\setminus \xi) \to \pi_1^{\operatorname{orb}}(\mathcal{O})$ is induced from the set map on the generating set
\begin{gather*}
x_i \mapsto x_i,\quad  y_i \mapsto y_i, \quad s_i \mapsto s_i \\
z_i \mapsto z_i ,\quad i=1,2,\cdots, b\\
z_{b+1} \mapsto s_1s_2. 
\end{gather*}
This inclusion induces the smooth map $\iota^*:\Rep_3 (\pi_1 ^{\operatorname{orb}}(\mathcal{O}))\to  \Rep_3 (\pi_1 ^{\operatorname{orb}}(\mathcal{O}\setminus \xi))$ defined by $[\rho]\mapsto [\rho\circ \iota]$. Then the following relation holds
\[
\mathbf{Hol}\circ \mathcal{SP} =\iota^*\circ \mathbf{Hol}.
\]

\subsubsection{Splitting along a essential simple closed curves}

Let $X_{\mathcal{O}}$ be an underlying space of $\mathcal{O}$. Let $\Sigma_{\mathcal{O}}$ be the set of singularities and let  $X_{\mathcal{O}}':= X_{\mathcal{O}} \setminus \Sigma_{\mathcal{O}}$. By an \emph{essential simple closed curve} in $\mathcal{O}$ we mean a simple closed curve in $X'_{\mathcal{O}}$ that is not homotopic to a point, a boundary component or a puncture. Let $\xi$ be an essential simple closed curve in $\mathcal{O}$. If each connected component of $\mathcal{O}\setminus \xi$ has negative Euler characteristic, then by  Lemma 4.1 of Choi-Goldman \cite{choi2005},  $\xi$ can be isotoped to a principal closed geodesic, which is denoted again by $\xi$. Because $\operatorname{\mathbf{h}}(\xi)$ is in $\mathbf{Hyp}^+$, the convex projective structure with principal boundary condition on $\mathcal{O}$ induces the canonical convex projective structure with principal boundary condition on $\mathcal{O}\setminus \xi$.

The effect of splitting along $\xi$ varies by connectedness of $\mathcal{O}\setminus \xi$. If $\mathcal{O}\setminus\xi$ is disconnected and has two connected components $\mathcal{O}_1$ and $\mathcal{O}_2$,  we have the splitting map $\mathcal{SP}: \mathcal{C}(\mathcal{O}) \to \mathcal{C}(\mathcal{O}_1) \times \mathcal{C}(\mathcal{O}_2)$ that assigns the induced convex projective structure on each $\mathcal{O}_i$. At the same time, we can decompose $\pi_1^{\operatorname{orb}}(\mathcal{O})$ into the amalgamated product of $\pi_1 ^{\operatorname{orb}}(\mathcal{O})=\pi_1 ^{\operatorname{orb}}(\mathcal{O}_1) \star_{\xi} \pi_1 ^{\operatorname{orb}}(\mathcal{O}_2)$. Let $\iota_i$ denote the inclusions $\pi_1 ^{\operatorname{orb}}(\mathcal{O}_i) \to \pi_1 ^{\operatorname{orb}}(\mathcal{O})$ that induce the maps $\iota_i ^* : \Rep_3 (\pi_1 ^{\operatorname{orb}}(\mathcal{O}))\to \Rep_3 (\pi_1 ^{\operatorname{orb}}(\mathcal{O}_i))$ as above. Then we have 
\[
(\mathbf{Hol}_1 \times \mathbf{Hol}_2)\circ \mathcal{SP} =(\iota_1^* \times \iota_2 ^*)\circ \mathbf{Hol}.
\] 

If $\xi$ is non-separating so that $\mathcal{O}\setminus \xi =: \mathcal{O}_0$ remains connected, we have the splitting map $\mathcal{SP}:\mathcal{C}(\mathcal{O}) \to \mathcal{C}(\mathcal{O}_0)$ and the decomposition of $\pi_1 ^{\operatorname{orb}}(\mathcal{O})$ into the HNN-extension $\pi_1^{\operatorname{orb}}(\mathcal{O}_0)\star_{\xi, \xi^{\perp}}$. Denote by $\iota:\pi_1 ^{\operatorname{orb}}(\mathcal{O}_0) \to \pi_1 ^{\operatorname{orb}}(\mathcal{O})$ the  inclusion with the induced map $\iota^*:\Rep_3 (\pi_1 ^{\operatorname{orb}}(\mathcal{O}))\to \Rep_3 (\pi_1 ^{\operatorname{orb}}(\mathcal{O}_0))$. Then we have the relation $\mathbf{Hol}\circ \mathcal{SP} =\iota^* \circ \mathbf{Hol}$.

\subsubsection{Splitting along a family}
Now we treat a general situation. Let $\{\xi_1, \cdots, \xi_{m}\}$ be a set of pairwise disjoint essential simple closed curves or full 1-suborbifolds. Suppose that each connected component of $\mathcal{O} \setminus \bigcup _{i=1} ^m \xi_i$ has also negative Euler characteristic. As we did above, we can isotope $\xi_1, \cdots, \xi_m$ to pairwise disjoint principal simple closed geodesics or principal full 1-suborbifolds. The resulting principal geodesics and principal full 1-suborbifolds will be denoted by the same letters $\xi_1,\cdots,\xi_m$. Let $\mathcal{O}_1,\cdots, \mathcal{O}_l$ be the completions of connected components of $\mathcal{O} \setminus \bigcup _{i=1} ^{m} \xi_i$. The convex projective structure with principal boundary on $\mathcal{O}$ induces the convex projective structure with principal boundary on each $\mathcal{O}_i$. By  construction, we know that for each $i$,  the boundary holonomies of $\mathcal{O}_i$ are hyperbolic. Therefore this splitting operation induces the splitting map
\[
\mathcal{SP}: \mathcal{C}(\mathcal{O}) \to \mathcal{C}(\mathcal{O}_1) \times \cdots \times \mathcal{C}(\mathcal{O}_l). 
\]

As we described above, the holonomy $\operatorname{\mathbf{h}}_i$ of the induced convex projective structure on each $\mathcal{O}_i$ is the restriction of $\operatorname{\mathbf{h}}$ to the subgroup $\pi_1 ^{\operatorname{orb}}(\mathcal{O}_i) \subset \pi_1 ^{\operatorname{orb}}(\mathcal{O})$. In other words if we denote the inclusion $\pi_1 ^{\operatorname{orb}}(\mathcal{O}_i) \to \pi_1 ^{\operatorname{orb}}(\mathcal{O})$ by $\iota_i$, we have the commutative diagram
\[
\xymatrix{\mathcal{C}(\mathcal{O})\ar[d]^{\mathbf{Hol}} \ar[rr]^{\mathcal{SP}} & & \mathcal{C}(\mathcal{O}_1) \times \cdots \times \mathcal{C}(\mathcal{O}_l) \ar[d]^{\mathbf{Hol}_1 \times \cdots \times \mathbf{Hol}_l}\\
\Rep_3(\pi_1 ^{\operatorname{orb}}(\mathcal{O})) \ar[rr]_--{\iota^*} && \Rep_3(\pi_1 ^{\operatorname{orb}}(\mathcal{O}_1))\times \cdots \times \Rep_3(\pi_1 ^{\operatorname{orb}}(\mathcal{O}_l))  }
\]
where $\iota^*$ is given by $\iota^*([\mathbf{h}])=([\mathbf{h}\circ\iota_1],[\mathbf{h}\circ\iota_2],\cdots,[\mathbf{h}\circ\iota_l])$.

We will see that the map $\mathcal{SP}$ is a submersion but is not an injection. Pasting and folding are the processes that produce a section of $\mathcal{SP}$.

\section{Preliminaries on group cohomology}\label{group}
In this section we define group pairs and their parabolic cohomology group to study the local structure of character varieties through the cohomology of the orbifold fundamental group. 

We introduce four types of projective resolutions. By using these resolutions, we prove Proposition \ref{FP} which states that the orbifold fundamental group of compact  orientable orbifolds with negative Euler characteristic has a finite length projective resolution. In Proposition \ref{paraboliccocycle} we give a characterization of the parabolic cohomology in terms of the bar resolution.  They are preliminary results for the construction of the symplectic form.

Throughout this section we will use $G$ to denote the Lie group $\operatorname{PSL}_n(\R)$ with its Lie algebra $\mathfrak{g}$ if we do not need to specify $n$.  A representation $\rho: \Gamma \to G$ induces a linear representation $\Ad \circ \rho: \Gamma \to \operatorname{GL}(\mathfrak{g})$ and this turns $\mathfrak{g}$ into a (left) $\R\Gamma$-module. We denote this $\R\Gamma$-module by $\mathfrak{g}_\rho$ or just $\mathfrak{g}$ if the representation $\rho$ is clear from the context. 

\subsection{Group pairs and the parabolic group cohomology}
The parabolic group cohomology was first considered by Weil \cite{weil1964}. Here we recall the construction of the parabolic cohomology for the purpose of  describing the tangent space of relative character variety $\Rep_n ^{\mathscr{B}} (\Gamma)$.  We adopt the description for the parabolic group cohomology from  \cite{guruprasad1997}. 

Let $\Gamma$ be a finitely generated group. By a  \emph{group pair}, we mean a pair $(\Gamma, \mathcal{S})=(\Gamma, \{\Gamma_1, \cdots, \Gamma_b\})$ of a group $\Gamma$ and a collection $\mathcal{S}$ of its finitely generated subgroups $\Gamma_1, \cdots, \Gamma_b$.

\begin{definition}
Let  $(\Gamma, \mathcal{S})=(\Gamma, \{\Gamma_1, \cdots, \Gamma_b\})$ be a group pair. An auxiliary resolution for $(\Gamma, \mathcal{S})$ consists of the following data:
\begin{itemize}
\item A projective resolution $R_*\to \R$ of the trivial $\R\Gamma$-module $\R$ 
\item  Projective resolutions $A_* ^i\to \R$ of the trivial $\R\Gamma_i$-module $\R$  for each $i=1,2,\cdots, b$. These resolutions have the property that $A_*:= \oplus_{i=1} ^ b \R \Gamma \otimes_{\Gamma_i} A_* ^i$ is a direct summand of the projective resolution $R_*\to \R$. 
\end{itemize}
\end{definition}

Let $(R_*,A^i _*)$ be an auxiliary resolution for a group pair $(\Gamma, \mathcal{S})$. Since $A_*$ is a direct summand of $R_*$, we have a chain complex of  $\R\Gamma$-modules $R_*/A_* \to \R$.   For any $\R\Gamma$-module $M$, we apply the functor $\Hom_\Gamma(-,M)$  to $R_*/A_*$ to get the chain complex  $(\Hom_\Gamma (R_*/ A_* , M),\delta)$. We define \emph{the relative group cohomology} $H^* (\Gamma, \mathcal{S};M)$ with coefficients in $M$  to be the cohomology of this chain complex $(\Hom_\Gamma(R_*/ A_* , M),\delta)$.

More intrinsic definition for the relative group cohomology can be found in Bieri-Eckmann \cite{bieri1978}, where the relative group cohomology is defined by $H^k(\Gamma, \mathcal{S};M)=\operatorname{Ext}^{k-1}_{\R\Gamma}(\Delta_{\Gamma/\mathcal{S}},M)$ for some $\Gamma$-module $\Delta_{\Gamma/\mathcal{S}}$. In particular, relative cohomology is functorial.

Consider the following long exact sequence associated to the exact sequence $0\to A_* \to R_* \to R_*/A_* \to 0$:
\begin{equation}\label{relativeseq}
\cdots \to \bigoplus_{i=1} ^b  H^0(\Gamma_i;M)\to  H^1(\Gamma, \mathcal{S};M)\to H^1 (\Gamma;M) \to \bigoplus_{i=1} ^b H^1(\Gamma_i; M)\to \cdots.
\end{equation}
\begin{definition}
The 1st \emph{parabolic cohomology} $H^1 _{\operatorname{par}} (\Gamma, \mathcal{S};M)$ of the group pair $(\Gamma, \mathcal{S})$ with coefficients in $M$ is defined by the image of the map $H^1(\Gamma, \mathcal{S};M)\to H^1(\Gamma; M)$ given in (\ref{relativeseq}).
\end{definition}

We mostly use the $\Gamma$-module $\mathfrak{g}_{\rho}$ for some representation $\rho:\Gamma \to G$ as our coefficients. 

Let $Z^q(\Gamma, \mathcal{S};M):=\ker (\Hom_\Gamma (R_q/A_q,M) \to \Hom_\Gamma(R_{q+1}/A_{q+1},M))$ be the space of relative $q$-cocycles and $Z^q(\Gamma;M)=\ker (\Hom_\Gamma (R_q,M) \to \Hom_\Gamma(R_{q+1},M))$ be the space of usual $q$-cocycle. We use the notations $B^q(\Gamma,\mathcal{S};M)$ and $B^q(\Gamma;M)$ to denote the relative $q$-coboundary and usual $q$-coboundary  respectively.

Recall that, in the cochain level, the restriction map $Z^1(\Gamma,\mathcal{S};M) \to Z^1(\Gamma;M)$ is induced from the projection $R_1 \to R_1/A_1$ which induces the restriction map $\Hom_\Gamma(R_1/A_1,M) \to \Hom_\Gamma (R_1,M)$. We can then compute the 1st parabolic cohomology as \[ H^1_{\operatorname{par}}(\Gamma, \mathcal{S};M)= \frac{\operatorname{Image}(Z^1(\Gamma,\mathcal{S};M) \to Z^1(\Gamma;M))}{B^1(\Gamma;M)}. \]

We discuss a functorial property of the parabolic group cohomology. Let $(\Gamma, \mathcal{S})$ be a group pair. Let $\mathcal{S}'$ be a subset of $\mathcal{S}$. We can construct auxiliary resolutions $(R_*,A_*)$ for $(\Gamma, \mathcal{S})$ and  $(R_*,A'_*)$ for $(\Gamma, \mathcal{S}')$ such that $A_*'$ is a direct summand of $A_*$. Hence we have a chain map $R_*/A_*' \to R_*/A_*$ which induces the map $H^1(\Gamma, \mathcal{S};M) \to H^1(\Gamma, \mathcal{S}';M)$.  Note that this map is not necessarily injective. However, we have  injectivity in the parabolic group cohomology. 

\begin{lemma}\label{funtorialinj}
Let $(\Gamma, \mathcal{S})$ be a group pair. Let $\mathcal{S}'$ be a subset of $\mathcal{S}$ Then the induced map 
\[
H^1_{\operatorname{par}}(\Gamma, \mathcal{S};M) \to H^1_{\operatorname{par}} (\Gamma, \mathcal{S}';M)
\]
is injective. 
\end{lemma}
\begin{proof}
Observe that the restriction maps  $Z^1(\Gamma, \mathcal{S};M)\to Z^1(\Gamma;M)$ and $Z^1(\Gamma, \mathcal{S}';M)\to Z^1(\Gamma;M)$ are injective. Then  the   injection $Z^1(\Gamma, \mathcal{S};M)\to Z^1(\Gamma, \mathcal{S}';M)$   gives rise to the injection 
\[
\frac{Z^1(\Gamma, \mathcal{S};M)}{B^1(\Gamma;M)}\to \frac{Z^1(\Gamma, \mathcal{S}';M)}{B^1(\Gamma;M)}. 
\]
The left hand side is, by definition, $H^1_{\operatorname{par}}(\Gamma, \mathcal{S};M)$ and the right hand side equals $H^1_{\operatorname{par}}(\Gamma, \mathcal{S}';M)$.  This yields the lemma. 
\end{proof}

\subsection{Models for the group cohomology}
The purpose of this section is to construct four different models for the group cohomology, each of which has its own benefits. 

Throughout this section, $\mathcal{O}$ denotes a compact connected oriented 2-orbifold of negative Euler characteristic with genus $g$, $b$ boundary components and $c$ cone points. We choose a base point $p_0$ in the interior of $\mathcal{O}$ and fix a presentation 
\begin{multline}\label{present}
\Gamma:= \pi_1 ^{\text{orb}}(\mathcal{O},p_0) \\= \langle x_1, y_1, \cdots, x_g, y_g, z_1, \cdots, z_b, s_1, \cdots, s_{c}\,|\,\mathbf{r}=  \mathbf{r}_1 = \cdots =\mathbf{r}_{c}=1\rangle
\end{multline}
where
\[
\mathbf{r}=\prod_{i=1} ^g [x_i,y_i] \prod_{j=1} ^b z_j \prod_{k=1} ^{c} s_k
\]
and where $\mathbf{r}_i= s_i ^{r_i}$, $i=1,2,\cdots, c$. 

\subsubsection{The bar resolution}
The bar resolution is the standard resolution of the trivial $\R\Gamma$-module $\R$ for any group $\Gamma$.  For the sake of completeness, we briefly remind the readers of its construction. 

We first define the $n$-chain $\mathbf{B}_n(\Gamma)$, $n>0$, to be the free $\R\Gamma$-module based on the elements of the form $\llbracket g_1|g_2|\cdots|g_n\rrbracket $ where $g_1, \cdots, g_n\in \Gamma\setminus\{1\}$.  Exceptionally we define $\mathbf{B}_0 (\Gamma):= \R\Gamma[p_0]$ or just simply $\mathbf{B}_0(\Gamma)=\R\Gamma$. The differential is given by 
\begin{multline*}
\partial (\llbracket g_1|g_2|\cdots|g_n\rrbracket )= g_1\llbracket g_2|\cdots|g_n\rrbracket +\sum_{i=1} ^{n-1} (-1)^{i} \llbracket g_1|\cdots|g_{i}g_{i+1}|\cdots|g_n\rrbracket \\
  + (-1)^n\llbracket g_1|\cdots|g_{n-1}\rrbracket 
\end{multline*}
for $n>1$ and by
\[
\partial (\llbracket g\rrbracket  ) =g \llbracket p_0\rrbracket   - \llbracket p_0\rrbracket .
\]
We also have the augmentation map $\epsilon:\mathbf{B}_0(\Gamma) \to \R$ given by \[\epsilon\left(\sum_{\gamma\in \Gamma} n_\gamma \gamma\right) = \sum_{\gamma\in \Gamma} n_\gamma.\]

It is well-known and can be shown by adopting the proof of Lemma \ref{barresolexact} that 
\begin{equation}\label{barresolution}
\cdots \to \mathbf{B}_2(\Gamma) \to \mathbf{B}_1(\Gamma) \to \mathbf{B}_0(\Gamma) \overset{\epsilon}{\to} \R \to 0
\end{equation} 
is a free resolution.

\subsubsection{The groupoid bar resolution}

A groupoid is a small category whose morphisms are all invertible. For a morphism $f$ in a groupoid,  we denote by $o(f)$ and $t(f)$ its domain and range respectively. 

Given a compact oriented 2-orbifold $\mathcal{O}$ of negative Euler characteristic, we choose a point $p_i$ at each boundary component. Define $\widetilde{\Gamma} = \pi_1^{\operatorname{orb}}(\mathcal{O},\{p_0,p_1, \cdots, p_{b}\})$ to be the fundamental groupoid of $\mathcal{O}$ with base $p_0, \cdots , p_{b}$. We fix, for each $i$, a path $w_i$ from $p_0$ to $p_i$.  Let $\zeta_i$ be the oriented boundary loop based at $p_i$ so that the complement of $x_1,y_1,\cdots, x_g, y_g,s_1, \cdots,s_c, w_1, \cdots ,w_b,\zeta_1,\cdots, \zeta_b$ in the underlying space $X_{\mathcal{O}}$ is a union of open disks. We then have the following presentation of $\widetilde{\Gamma}$
\begin{multline*}
\widetilde{\Gamma}=\langle x_1,y_1, \cdots,x_g, y_g,w_1, \cdots, w_b, \zeta_1, \cdots, \zeta_b, s_1, \cdots, s_c\,|\,\\
 \prod_{i=1} ^g [x_i,y_i]\prod_{i=1} ^b w_i \zeta_i w_i^{-1} \prod_{i=1}^c s_i= s_1^{r_1}=\cdots =s_c ^{r_c}=1\rangle.
\end{multline*}

There is a natural groupoid morphism $\operatorname{ext}:\Gamma \to \widetilde{\Gamma}$ given by
\[
x_i \mapsto x_i,\: y_i\mapsto y_i,\: z_i \mapsto w_i  \zeta_i w_i^{-1},\:  s_i\mapsto s_i.
\]
On the other hand, we also have the groupoid morphism $\operatorname{ret}: \widetilde{\Gamma} \to \Gamma$ defined by
\[
x_i \mapsto x_i,\: y_i \mapsto y_i,\: w_i \mapsto 1,\: \zeta_i \mapsto z_i,\: s_i \mapsto s_i.
\]

Now we define the $n$-chain $\widetilde{\mathbf{B}}_n(\Gamma)$, $n>0$, to be the free $\R\Gamma$-module over the elements of the form 
\[
\llbracket g_1| g_2| \cdots| g_n\rrbracket  
\]
where $g_i \in \widetilde{\Gamma}\setminus\{1\}$ such that $t(g_i)=o(g_{i+1})$. The 0-chain $\widetilde{\mathbf{B}}_0(\Gamma)$ is given by the free $\R\Gamma$-module on $\{\llbracket p_0\rrbracket  ,\cdots,\llbracket p_b\rrbracket  \}$. The differential $\partial:\widetilde{\mathbf{B}}_n(\Gamma) \to \widetilde{\mathbf{B}}_{n-1}(\Gamma)$ is defined by
\begin{multline*}
\partial (\llbracket g_1|g_2|\cdots|g_n\rrbracket  )= \operatorname{ret}(g_1)\llbracket g_2|\cdots|g_n\rrbracket  \\
+\sum_{i=1} ^{n-1} (-1)^{i} \llbracket g_1|\cdots|g_{i}g_{i+1}|\cdots|g_n\rrbracket   + (-1)^n\llbracket g_1|\cdots|g_{n-1}\rrbracket  
\end{multline*}
for $n>1$ and 
\[
\partial (\llbracket g\rrbracket  ) = \operatorname{ret}(g) \llbracket t(g)\rrbracket   - \llbracket o(g)\rrbracket  .
\]
\begin{lemma}\label{barresolexact}
Define the augmentation $\widetilde{\epsilon}:\widetilde{\mathbf{B}}_0(\Gamma)\to \R$  by $\widetilde{\epsilon} (\sum a_i \llbracket p_i \rrbracket) = \sum \epsilon (a_i)$ where $\epsilon: \R\Gamma \to \R$ is the usual augmentation map given by $\epsilon(\sum n_i \gamma_i)=\sum n_i$.   Then the complex $\widetilde{\mathbf{B}}_\ast (\Gamma)\overset{\widetilde{\epsilon}}{\to} \R \to 0$ is exact. 
\end{lemma}
\begin{proof}
It is straightforward to check that $\partial \partial =0$. To prove  exactness, we use the homotopy operator $h^n : \widetilde{\mathbf{B}}_n (\Gamma) \to \widetilde{\mathbf{B}}_{n+1} (\Gamma)$ defined by
\[
h ^n(g\llbracket g_1|g_2|\cdots|g_n\rrbracket  )= \begin{cases}
 \llbracket g^{g_1}|g_1|g_2|\cdots|g_n\rrbracket   & \text{ if }g\ne 1\\
 0 & \text{ if }g=1
 \end{cases}
\]
where $g^{g_1}= \operatorname{ext}(g) w_i$ for the generator $w_i$ of $\widetilde{\Gamma}$ that joins $p_0$ and the origin $o(g_1)$ of $g_1$. Observe that $\operatorname{ret}(g^{g_1})= g$ and $(g\operatorname{ret}(g_1))^{g_2} = g^{g_1} g_1$ for $g\in \Gamma$, and $g_1,g_2\in \widetilde{\Gamma}$ with $t(g_1)=o(g_2)$. Using this observation, one can check that $h \partial + \partial h =\operatorname{Id}$ which proves the exactness of $\widetilde{\mathbf{B}}_*(\Gamma)$. 
\end{proof}

Therefore we have another free resolution $\widetilde{\mathbf{B}}_i(\Gamma)$ 
\begin{equation}\label{groupoidbarresolution}
\cdots \to \widetilde{\mathbf{B}}_i(\Gamma) \to\widetilde{\mathbf{B}}_{i-1}(\Gamma) \to \cdots \to\widetilde{\mathbf{B}}_1(\Gamma) \to \widetilde{\mathbf{B}}_0 (\Gamma) \to \R \to 0.
\end{equation}

\subsubsection{The geometric resolution I}
We consider another model which is based on the canonical cellular decomposition of $\mathcal{O}$. This resolution carries  clearer geometric picture of the  group $\pi_1 ^{\text{orb}}(\mathcal{O})$.   Although this type of resolution may be classical and well-known, its construction is somehow scattered all over the old literatures, for instance, \cite{lyndon1950}, \cite{fox1954}, \cite{huebschmann1979}, and \cite{majumdar1970}. Here, we give a specific construction that fits for our situation.

Our free resolution consists of the free $\R\Gamma$-modules
\begin{align*}
\mathbf{G}_k &= \R\Gamma [\mathbf{r}_1, \cdots, \mathbf{r}_c] , \text{ for }k\ge 3\\
\mathbf{G}_2 &= \R\Gamma[ \mathbf{r}, \mathbf{r}_1, \cdots, \mathbf{r}_{c}],\\
\mathbf{G}_1 &= \R\Gamma[x_1,y_1, \cdots, x_g,y_g, z_1, \cdots, z_b,s_1, \cdots, s_{c}], \text{ and} \\
\mathbf{G}_0 &= \R\Gamma[p_0].
\end{align*}
To  distinguish the generators of each module from group elements, we enclose the generators of modules by the double bracket $\llbracket\cdot \rrbracket$. 

Now we introduce the Fox derivatives \cite{fox1953}, the essential ingredients to define the differential. 
\begin{definition}
Let $F$ be the free group on the generators $v_1, \cdots, v_n$. The \emph{Fox derivatives} $\frac{\partial }{\partial v_i}:\R F \to \R F$, $i=1,2, \cdots, n$ are  $\R$-linear operators with the following properties
\begin{itemize}
\item For every $j$, 
\[
\frac{\partial v_j}{\partial v_i} = \begin{cases} 1 & \text{ if } i=j \\ 0 & \text{ if }i\ne j\end{cases}.
\]
\item For every $x,y\in \R F$, 
\[
\frac{\partial xy}{\partial v_i} = \frac{\partial x}{\partial v_i}\epsilon(y)+x \frac{\partial y}{\partial v_i}
\]
where $\epsilon:\R F \to \R$ is the augmentation map.
\item (The mean value property) For every $x\in \R F$, 
\[
x = \epsilon(x)1 +\sum_{i=1} ^n  \frac{\partial x}{\partial v_i}(v_i-1).
\]
\end{itemize}
\end{definition}

Let $\Gamma$ be given as the presentation (\ref{present}), and let $\mathcal{G}$ be the set of generators. We know that $\R\Gamma = \R F /N$ where $F$ is the free group on $\mathcal{G}$ and $N$ is the ideal in $\R F$ generated by the elements of the form $(x-1)$ for $x$ in the normal closure of $\mathbf{r}, \mathbf{r}_1, \cdots, \mathbf{r}_c$. Let $\pi:\R F \to \R \Gamma$ be the projection.

Using this Fox derivatives, the differential is given by 
\begin{align*}
\partial (\llbracket \mathbf{r}_i\rrbracket  )& := \pi(s_i-1)\llbracket \mathbf{r}_i\rrbracket  ,\text{ for }\llbracket\mathbf{r}_i\rrbracket \in \mathbf{G}_3, \mathbf{G}_5, \mathbf{G}_7,\cdots\\
\partial(\llbracket \mathbf{r}_i\rrbracket  ) &:=\pi(1+s_i+\cdots+s_i^{r_i-1})\llbracket \mathbf{r}_i\rrbracket,\text{ for }\llbracket\mathbf{r}_i\rrbracket \in \mathbf{G}_4, \mathbf{G}_6, \mathbf{G}_8,\cdots  \\
\partial(\llbracket \mathbf{r}_i\rrbracket  ) &:= \pi \frac{\partial \mathbf{r}_i}{\partial s_i}\llbracket s_i\rrbracket  = \pi(1+s_i+\cdots+s_i^{r_i-1})\llbracket s_i\rrbracket,\text{ for }\llbracket\mathbf{r}_i\rrbracket \in \mathbf{G}_2,  \\
\partial (\llbracket \mathbf{r}\rrbracket  ) &:= \sum_{i=1} ^g \left( \pi \frac{\partial \mathbf{r}}{\partial x_i } \llbracket x_i\rrbracket   + \pi \frac{\partial \mathbf{r}}{\partial y_i}\llbracket y_i\rrbracket  \right) + \sum_{i=1} ^b \pi \frac{\partial \mathbf{r}}{\partial z_i}\llbracket z_i\rrbracket   + \sum_{i=1} ^{c} \pi \frac{\partial \mathbf{r}}{\partial s_i}\llbracket s_i\rrbracket  \\
\partial(\llbracket x_i\rrbracket  ) &:=  \pi(x_i-1)\llbracket p_0\rrbracket  , \quad \partial(\llbracket y_i\rrbracket  ) :=  \pi(y_i-1)\llbracket p_0\rrbracket   \\
\partial (\llbracket z_i \rrbracket  )&:=  \pi(z_i-1)\llbracket p_0\rrbracket  ,\quad \partial (\llbracket s_i\rrbracket  ) :=  \pi (s_i-1)\llbracket p_0\rrbracket  .
\end{align*}
Together with this differential we get the chain complex of free $\R\Gamma$-modules
\begin{equation}\label{georesol1}
\cdots \to \mathbf{G}_k \to \cdots \to \mathbf{G}_2 \to \mathbf{G}_1 \to \mathbf{G}_0 \overset{\epsilon}{\to} \R \to 0. 
\end{equation}
\begin{lemma}\label{georesolproof}
The chain complex $(\mathbf{G}_*,\partial)$ is exact. 
\end{lemma}
\begin{proof}
Using the mean value property of the Fox derivatives, we can show that $\partial\partial =0$.  

Let $\mathcal{G}$ be the generating set of the presentation (\ref{present}) and let $\mathcal{G}'= \mathcal{G}\cup \mathcal{G}^{-1}$.  Denote by $\|\cdot\|$ the word length with respect to the presentation (\ref{present}). Observe that $\mathbf{G}_1= \R \Gamma [\mathcal{G}]$. From now on, we identify $\mathbf{G}_0 = \R \Gamma [p_0]$ with $\R\Gamma$. 

Any element $z\in \R\Gamma$ can be uniquely written as $z= \sum_{i} n_i g_i$ where $n_i\in \R\setminus\{0\}$ and $g_i \in \Gamma$. Since such an expression is unique, the quantity $\ell(z) := \sum_i \| g_i\|$ is well-defined.

Suppose that $\partial (x) =\epsilon(x)=0$ for $x\in \mathbf{G}_0=\R\Gamma$. We show that there is $y\in \mathbf{G}_1$ such that $x=\partial (y)$ by induction on $\ell(x)$. 

When $\ell(x)= 1$, $x$ must be of the form $x= n v-n$ where $n \in \R\setminus\{0\}$ and $v\in \mathcal{G}\cap \mathcal{G}^{-1}$. If $v\in \mathcal{G}$, we have $x = \partial (n\llbracket v \rrbracket)$. Otherwise, as $v^{-1} \in \mathcal{G}$, we have $x = \partial (-nv \llbracket v^{-1} \rrbracket)$.

Now, assume $\ell(x)>1$. Write $x= \sum_{i\ge 1} n_i g_i$ where $n_i\in \R\setminus\{0\}$ and $g_i \in \Gamma$. Observe that at least one, say $g_1$, is not the identity element. Let $v_{\|g_1\|} v_{\|g_1\|-1} \cdots  v_1$ be the reduced word representing $g_1$ and let $g'=v_{\|g_1\|} v_{\|g_1\|-1} \cdots  v_2$. Then we can write 
\[
x=\begin{cases}
\pi(n_1 g' (v_1-1) +n_1 g' +\sum_{i>1} n_i g_i) & \text{ if }v_1 \in \mathcal{G}\\
\pi(-n_1g' v_1 (v_1 ^{-1} -1)+n_1 g'+\sum_{i>1} n_i g_i) & \text{ if }v_1 \in \mathcal{G}^{-1}
\end{cases}.
\]
Let $x' = \pi( n_1 g' +\sum_{i>1} n_i g_i)$. Since $\ell(x')<\ell(x)$ and since $\epsilon(x')=0$, by the induction hypothesis, we can find $y'\in \mathbf{G}_1$ such that $\partial (y' ) = x'$. Then \[ x=\begin{cases}\partial (n_1 \pi(g') \llbracket v_1 \rrbracket +y') & \text{ if } v_1 \in \mathcal{G}\\ \partial (-n_1 \pi( g' v_1) \llbracket v_1^{-1} \rrbracket +y') & \text{ if }v_1 \in \mathcal{G}^{-1}\end{cases},\]and this completes the induction. 

Now we prove the exactness at $\mathbf{G}_1$. For this let $x\in \mathbf{G}_1$ be such that $\partial (x) =0$. Lift the coefficients of $x$ to $\R F$ and write $x=\sum_{i\ge 0 } \pi\left( d_i\right) \llbracket v_i \rrbracket$ for $d_i\in \R F$.  Being $\partial (x)=0$  means that $\sum_{i\ge 1} d_i  (v_i -1)=\sum_{j \ge 1} n_j (m_j -1) u_j$ for  $n_j,u_j \in \R F \setminus\{0\}$ and $m_j$ in the normal closure of $\mathbf{r}, \mathbf{r}_1, \cdots, \mathbf{r}_c$. By applying the Fox derivative $\frac{\partial}{\partial v_i}$ on both sides we get 
\[
d_i = \sum_{j\ge 1} \left(n_j \frac{\partial m_j }{\partial v_i}\epsilon(u_j) + n_j (m_j-1)\frac{\partial u_j}{\partial v_i}\right).
\]
Note that the terms $n_j (m_j-1)\frac{\partial u_j}{\partial v_i}$ vanish under the projection  $\pi:\R F\to \R\Gamma$. Hence, we may assume that \[ d_i =\sum_{j\ge 1} n_j \frac{\partial m_j }{\partial v_i}\epsilon(u_j). \]

Since $m_j$ is in the normal closure of  the words $\mathbf{r}, \mathbf{r}_1, \cdots, \mathbf{r}_c$, $m_j$ is a product of words of the forms $g \mathbf{r}^{\pm 1} g^{-1}$ and $g \mathbf{r}_i ^{\pm 1} g^{-1}$. Observe that 
\[
\frac{\partial  }{\partial v_i}(g \mathbf{r}^{\pm 1} g^{-1}) = \pm g \frac{\partial \mathbf{r}}{\partial v_i}.
\]
This observation, in particular, implies that $\frac{ \partial m_j}{\partial v_i}$ can be written as 
\[
\frac{ \partial m_j}{\partial v_i}= f_j \frac{\partial \mathbf{r}}{\partial v_i} + \sum_{k=1} ^c h_{j,k} \frac{\partial \mathbf{r}_k}{\partial v_i}
\]
for the elements $f_j$ and $h_{j,k}$ of $\R F$ which depend only on $m_j$. This shows that $x$ is the image of
\[
\sum_{j\ge 1} n_j  \epsilon(u_j)\pi(f_j) \llbracket \mathbf{r}\rrbracket + \sum_{k=1} ^ c \sum_{j\ge 1} n_j \epsilon(u_j) \pi(h_{j,k}) \llbracket \mathbf{r}_k \rrbracket.
\]
This shows the exactness at $\mathbf{G}_1$. 

The exactness at $\mathbf{G}_k$, $k>2$, follows immediately from the following lemma:
\begin{lemma}[Also proven in \cite{majumdar1970}] \label{zerodivisor}
Let $x\in \R F$ be such that $\pi(x(s_i -1))=0$ then $\pi(x)=\pi( x' (s_i ^{r_i-1} + \cdots+s_i+1))$ for some $x'\in \R F$. 

If $\pi(x (s_i ^{r_i-1} + \cdots+s_i+1)) =0$ then $\pi(x)=\pi(x'(s_i-1))$ for some $x'\in \R F$. 
\end{lemma}
\begin{proof}
Suppose that $\pi(x(s_i-1) )=0$. Then $xs_i +y = x$ for some $y\in N$.  Inductively, we have that  \[ x=x s_i+y=xs_i ^2+ys_i+y = \cdots = xs_{i} ^{r_i-1}+y(1+s_i+s_i^2+\cdots+s_i^{r_i-2}).\] Therefore, $r_i \pi(x) = \pi(x+xs_i + \cdots +xs_i ^{r_i-1})$. This yields 
\[
\pi(x)= \pi\left(\frac{x}{r_i} \cdot  (s_i ^{r_i-1} + \cdots+s_i+1)\right).\]

For the second assertion, assume that $\pi(x  (s_i ^{r_i-1} + \cdots+s_i+1) )=0$. This means that $ x  (s_i ^{r_i-1} + \cdots+s_i+1) = y$ for some $y\in N$.  Observe that, since $\epsilon(x(s_i ^{r_i-1}+\cdots + s_i +1))=\epsilon (x) r_i = 0$, we have $\epsilon (x) =0$.  Apply the Fox derivatives to get $\frac{\partial y } {\partial v} = r_i  \frac{\partial x}{\partial v}$ for all $v\in \mathcal{G}\setminus\{s_i\}$. Hence, by the mean value property,
\[
x= \epsilon (x) + \sum_{v\in \mathcal{G}}\frac{\partial x}{\partial v} (v-1) = \frac{\partial x}{\partial s_i} (s_i-1)+ \sum_{v\in \mathcal{G}\setminus\{s_i\}} \frac{1}{r_i} \frac{\partial y}{\partial v}(v-1).
\]
Again by the mean value property, we have
\[
 \sum_{v\in \mathcal{G}\setminus\{s_i\}} \frac{1}{r_i} \frac{\partial y}{\partial v}(v-1) = \frac{1}{r_i}\left( y-\epsilon (y) - \frac{\partial y}{\partial s_i}(s_i-1) \right)=\frac{1}{r_i}\left( y- \frac{\partial y}{\partial s_i}(s_i-1) \right).
\]
Thus
\[
x= \left( \frac{\partial x}{\partial s_i} -\frac{1}{r_i} \frac{\partial y}{\partial s_i} \right) (s_i-1) +\frac{y}{r_i}.    
\]
Modulo $N$, we have that \[
\pi(x)= \pi\left(\left( \frac{\partial x}{\partial s_i} -\frac{1}{r_i} \frac{\partial y}{\partial s_i} \right) (s_i-1)\right).\]
\end{proof}

Finally, we show the exactness at $\mathbf{G}_2$. Let $x= \pi(d_0) \llbracket\mathbf{r}\rrbracket + \sum_{i=1} ^c \pi(d_i) \llbracket \mathbf{r}_i \rrbracket$, $d_i \in\R F$, be such that $\partial (x) =0$. 

Suppose that $b>0$. Since the coefficient of $\partial(x)$ of $\llbracket z_i \rrbracket$ is 
\[
\pi(d_0)\pi\left(\frac{\partial \mathbf{r}}{\partial z_i}\right)=\pi(d_0)\pi\left(\prod_{k=1} ^g [x_k,y_k]\prod_{j<i} z_j\right),
\]
we know that $\pi(d_0)=0$.  By Lemma  \ref{zerodivisor}, we have $\pi(d_i)=\pi(d_i' (s_i-1))$ for some $d_i'\in \R F$ as we wanted.  

If $g>0$ we look at the coefficient of $\partial(x)$ of $\llbracket x_1\rrbracket$, which is 
\[
\pi(d_0) \pi\left(\frac{\partial \mathbf{r}}{\partial x_1}\right)=\pi(d_0)\pi (1-x_1y_1 x_1^{-1}).
\]
Since $\pi(y_1)$ is not a torsion, $\pi(d_0)\pi(1-x_1y_1x_1^{-1})\ne 0$ unless $\pi(d_0)=0$. To conclude we argue as in the previous case. 

Therefore, it remains the case when $b=g=0$, in which case we must have $c\ge 3$ by the hyperbolicity of $\Gamma$.  Note that it suffices to show  $\pi(d_0)=0$. Note also that $\partial (x) = 0$ means
\begin{multline*}
 \sum_{i=1} ^c \pi\left( d_0 \frac{\partial \mathbf{r}}{\partial s_i} +  d_i \frac{\partial \mathbf{r}_i}{\partial s_i} \right) \llbracket s_i \rrbracket  \\  = \sum_{i=1} ^c   \pi\left( d_0 (s_1 s_2 \cdots s_{i-1}) +d_i (1+s_i + \cdots s_i ^{r_i-1})\right) \llbracket s_i \rrbracket=0. 
\end{multline*}
Since $\pi(\mathbf{r})=\pi( s_1 s_2 \cdots s_c)=1$,  we get
\begin{align*}
\pi(d_0) & = -\pi(d_j (1+ s_j + \cdots + s_j ^{r_j-1}) (s_1 s_2 \cdots s_{j-1})^{-1})\\
& = -\pi(d_j (1+s_j + \cdots + s_j ^{r_j -1})s_j s_{j+1} \cdots s_c) \\
&= -\pi(d_j (1+s_j + \cdots + s_j ^{r_j -1}) s_{j+1}s_{j+2} \cdots s_c)\\
&= -\pi(d_j)\pi\left( \frac{\partial \mathbf{r}_j}{\partial s_j}\right) \pi(s_{j+1}s_{j+2} \cdots s_c) 
\end{align*}
from the coefficient of $\llbracket s_{j} \rrbracket$. On the other hand, by looking at the coefficient of $\llbracket s_{j+1} \rrbracket$, we have
\begin{align*}
\pi(d_0) & = -\pi(d_{j+1} (1+s_{j+1} + \cdots +s_{j+1} ^{r_{j+1} -1}) (s_1 s_2  \cdots s_j )^{-1}) \\
&= -\pi(d_{j+1} (1+s_{j+1} + \cdots +s_{j+1} ^{r_{j+1} -1})s_{j+1} s_{j+2} \cdots s_c)\\
&=-\pi(d_{j+1})\pi\left( \frac{\partial \mathbf{r}_{j+1}}{\partial s_{j+1}}\right) \pi(s_{j+1}s_{j+2} \cdots s_c). 
\end{align*} From the above two equations, we know that \[\pi(d_j)\pi\left(\frac{\partial \mathbf{r}_{j}}{\partial s_{j}}\right)=\pi(d_{j+1})\pi\left(\frac{\partial \mathbf{r}_{j+1}}{\partial s_{j+1}}\right) \] for all $j=1,2,\cdots,c-1$. Note also that \begin{align*}-\pi(d_0)&=\pi(d_1) \pi \left(\frac{\partial \mathbf{r}_1}{\partial s_1}\right) \pi (s_2 \cdots s_c) \\ &= \pi(d_1)\pi((1+s_1+\cdots +s_1 ^{r_1-1})s_1^{-1})\\ &=\pi(d_1)\pi(1+s_1+\cdots + s_1 ^{r_1-1})\\ &=\pi(d_1)\pi \left(\frac{\partial \mathbf{r}_1}{\partial s_1}\right). \end{align*}Altogether, we have
\[
-\pi(d_0)=\pi(d_1)\pi\left( \frac{\partial \mathbf{r}_1}{\partial s_1}\right)=\pi(d_2)\pi\left( \frac{\partial \mathbf{r}_2}{\partial s_2}\right) =\cdots = \pi(d_{c})\pi\left( \frac{\partial \mathbf{r}_{c}}{\partial s_c}\right). 
\]
Then, $\pi(d_0(s_i-1))=0$ for all $i$, because 
\begin{align*}\pi(d_0)\pi(s_i-1) &= - \pi (d_i)\pi\left( \frac{\partial \mathbf{r}_i}{\partial s_i}\right)\pi(s_i-1) \\ & = - \pi(d_i)\pi((1+s_i +\cdots +s_i^{r_i-1})(s_i-1))\\ &=\pi(d_i)\pi(1-s_i^{r_i})\\ &=0.
\end{align*}
In other words, $\pi(d_0 s_i) = \pi(d_0)$ for all $i=1,2,\cdots, c$. Since any element $z\in \Gamma$ can be written as a product of generators $z=\pi(s_{i_1}\cdots s_{i_l})$, we can conclude that $\pi(d_0 (z-1))=0$ for all $z\in \Gamma$. Since $c\ge 3$, $\Gamma$ is infinite and there is a non-torsion  element $z_0\in \Gamma$. Therefore, $\pi(d_0 (z_0-1))=0$ for some non-torsion element $z_0$. This leads us to $\pi(d_0)=0$ as we wanted. 
 \end{proof}

\subsubsection{The geometric resolution II} We present the last projective resolution which is a slightly refined version of the previous one. This model is particularly useful for  computing the parabolic cohomology.  

As we did in the construction of the groupoid bar resolution, we  additionally pick a point $p_i$ on each $\zeta_i$ and a path  $w_i$ from $p_0$ to $p_i$ in such a way that the complement of
\[
x_1,y_1, \cdots, x_g,y_g, s_1, \cdots, s_{c}, w_1, \cdots, w_b, \zeta_1, \cdots, \zeta_b
\]
in the underlying topological space $X_\mathcal{O}$ of $\mathcal{O}$ is a union of open disks.

As before, let
\begin{align*}
\widetilde{\mathbf{G}}_k &= \R\Gamma [\mathbf{r}_1, \cdots, \mathbf{r}_{c}] , \text{ for }k\ge 3\\
\widetilde{\mathbf{G}}_2 &= \R\Gamma[ \mathbf{r}, \mathbf{r}_1, \cdots, \mathbf{r}_{c}],\\
\widetilde{\mathbf{G}}_1 &= \R\Gamma[x_1,y_1, \cdots, x_g,y_g, \zeta_1, \cdots, \zeta_b, w_1, \cdots, w_b,s_1, \cdots, s_{c}], \text{ and} \\
\widetilde{\mathbf{G}}_0 &= \R\Gamma[p_0,p_1, \cdots, p_b]. 
\end{align*}
Again to avoid any potential confusion, we place the generators of each module inside the double bracket $\llbracket\cdot \rrbracket$. We define the  differential by
\begin{align*}
\partial (\llbracket \mathbf{r}_i\rrbracket  )& := (s_i-1)\llbracket \mathbf{r}_i\rrbracket, \text{ for }\llbracket \mathbf{r}_i\rrbracket \in \widetilde{\mathbf{G}}_3, \widetilde{\mathbf{G}}_5, \widetilde{\mathbf{G}}_7,\cdots \\
\partial (\llbracket \mathbf{r}_i\rrbracket  )& := (1+s_i+\cdots+s_i^{r_i-1})\llbracket \mathbf{r}_i\rrbracket, \text{ for }\llbracket \mathbf{r}_i\rrbracket \in \widetilde{\mathbf{G}}_4, \widetilde{\mathbf{G}}_6, \widetilde{\mathbf{G}}_8,\cdots \\
\partial(\llbracket \mathbf{r}_i\rrbracket  ) &:= \frac{\partial \mathbf{r}_i}{\partial s_i}\llbracket s_i\rrbracket  =(1+s_i+\cdots+s_i^{r_i-1})\llbracket s_i\rrbracket, \text{ for }\llbracket \mathbf{r}_i \rrbracket \in \widetilde{\mathbf{G}}_2 \\
\partial (\llbracket \mathbf{r}\rrbracket  ) &:= \sum_{i=1} ^g \left( \frac{\partial \mathbf{r}}{\partial x_i } \llbracket x_i\rrbracket   + \frac{\partial \mathbf{r}}{\partial y_i}\llbracket y_i\rrbracket  \right) + \sum_{i=1} ^b \frac{\partial \mathbf{r}}{\partial z_i}\llbracket \zeta_i\rrbracket   + \sum_{i=1} ^{c} \frac{\partial \mathbf{r}}{\partial s_i}\llbracket s_i\rrbracket  +\sum_{i=1} ^b \frac{\partial \mathbf{r}}{\partial z_i} (1-z_i)\llbracket w_i\rrbracket  \\
\partial(\llbracket x_i\rrbracket  ) &:= (x_i-1)\llbracket p_0\rrbracket  , \quad \partial(\llbracket y_i\rrbracket  ) := (y_i-1)\llbracket p_0\rrbracket   ,\quad \partial(\llbracket w_i\rrbracket  ) = \llbracket p_i\rrbracket  -\llbracket p_0\rrbracket  \\
\partial (\llbracket \zeta_i\rrbracket  )&:= (z_i-1)\llbracket p_0\rrbracket  ,\quad \partial (\llbracket s_i\rrbracket  ) := (s_i-1)\llbracket p_0\rrbracket  .
\end{align*}

By adopting the proof of Lemma \ref{georesolproof}, one can show that 
\begin{equation}\label{georesol2}
\cdots  \to \widetilde{\mathbf{G}}_3 \to \widetilde{\mathbf{G}}_2 \to \widetilde{\mathbf{G}}_1 \to \widetilde{\mathbf{G}}_0 \overset{\widetilde{\epsilon}}{\to} \R\to 0 
\end{equation}
gives also a free resolution. 

\subsection{Consequences}
We list some  consequences of the above construction of various resolutions. 

Recall that a group $\Gamma$ is \emph{of type $FP$} over $\R$ if there is a finite length resolution
\[
0 \to R_n\to \cdots \to R_1 \to R_0 \to \R \to 0
\]
such that each $R_i$ is a projective and finitely generated $\R\Gamma$-module. 
\begin{proposition}\label{FP}
Let $\mathcal{O}$ be a compact orientable 2-orbifold of negative Euler characteristic. Then $\Gamma=\pi_1 ^{\operatorname{orb}}(\mathcal{O})$ is of type $FP$ over $\R$.
\end{proposition}
\begin{proof}
We begin with the geometric resolution (\ref{georesol1}). Let $Q$ be the submodule generated by $\llbracket \mathbf{r}_1\rrbracket  , \cdots, \llbracket \mathbf{r}_{c}\rrbracket  $ in $\mathbf{G}_2$. Consider 
\[
\mathbf{G}'_2  := \R\Gamma[\mathbf{r}],\quad  \mathbf{G}'_1 := \mathbf{G}_1 / \partial (Q), \quad  \mathbf{G}'_0 = \mathbf{G}_0, \quad  \text{ and }\quad \mathbf{G}'_j =0 \text{ for }j\ge 3.
\]
In light of Lemma \ref{zerodivisor}, we know that
\[
\mathbf{G}'_1\cong \R\Gamma[x_1, y_1, \cdots, x_g,y_g, z_1, \cdots, z_b]\oplus\bigoplus_{i=1} ^{c} (\R\Gamma \otimes_{\R Q_i} K_i)
\]
where $Q_i$ is a finite cyclic subgroup of $\Gamma$ generated by $s_i$ and $K_i$ is the kernel of the augmentation map $\varepsilon: \R Q_i \to \R$. Since $Q_i$ is finite, we have a section $\sigma:\R \to \R Q_i$ of $\varepsilon$ defined by $\sigma(x) = \frac{1}{\# Q_i} \sum_{\gamma\in Q_i} x\cdot \gamma$. This makes $K_i$  a direct summand of a free $\R Q_i$-module $\R Q_i$. Therefore, $K_i$ is a projective $\R Q_i$-module. It follows that $\R \Gamma \otimes_{\R Q_i} K_i$ is a projective $\R \Gamma$-module. Therefore  $\mathbf{G}'_1$ is also projective. The new complex
\begin{equation}\label{resolution'}
\cdots \to 0 \to \cdots \to 0 \to \mathbf{G}'_2 \to \mathbf{G}'_1 \to \mathbf{G}'_0 \to \R 
\end{equation}
is clearly exact and hence a  projective resolution of length 2. 
\end{proof}

We can also obtain the following description of the parabolic cohomology. While the result itself might be well-known, we are not able to find a good reference for its proof. For the sake of completeness, we present the proof here as well. 

\begin{proposition}\label{paraboliccocycle}
Let $\mathcal{S}=\{\langle z_1 \rangle, \cdots ,\langle z_b\rangle \}$. In terms of the groupoid bar resolution, the parabolic cohomology is represented by the parabolic cocycles
\[
\widetilde{Z}^1_{\operatorname{par}} (\Gamma, \mathcal{S}; \mathfrak{g}) := \{u\in \widetilde{Z}^1(\Gamma;\mathfrak{g})\,|\, u(\llbracket \zeta_i\rrbracket  )=0 \text{, } i=1,2,\cdots, b\}
\]
where $\widetilde{Z}^1(\Gamma;\mathfrak{g})=\ker (\Hom_{\Gamma}(\widetilde{\mathbf{B}}_1(\Gamma), \mathfrak{g})\to \Hom_{\Gamma}(\widetilde{\mathbf{B}}_2(\Gamma), \mathfrak{g}))$ is the group of usual cocycles. 

In terms of the bar resolution,  the parabolic cohomology is represented by
\[
Z^1_{\operatorname{par}} (\Gamma, \mathcal{S}; \mathfrak{g}) := \{u\in Z^1(\Gamma;\mathfrak{g})\,|\, u(\llbracket z_i\rrbracket  )=z_i\cdot X_i -X_i \text{ for some }X_i \in \mathfrak{g} \text{, } i=1,2,\cdots, b\}
\]
where $Z^1(\Gamma;\mathfrak{g})=\ker (\Hom_{\Gamma}(\mathbf{B}_1(\Gamma), \mathfrak{g})\to \Hom_{\Gamma}(\mathbf{B}_2(\Gamma), \mathfrak{g}))$. 
\end{proposition}

\begin{proof} The proof goes as follow. We first characterize the parabolic cohomology in terms of the resolution (\ref{groupoidbarresolution}) that captures the peripheral structure well. Then we translate (\ref{groupoidbarresolution}) into the bar resolution (\ref{barresolution}) and track the parabolic cocycles. 

For each $i$, we define $A^i _j$, $j>0$, to be the free $\R[\langle z_i \rangle]$-module on the set $\{\llbracket x_1|x_2|\cdots|x_j\rrbracket\,|\,x_1, \cdots, x_j \in \langle \zeta_i \rangle\setminus \{1\}\}$.  We let $A^i_0 = \R[\langle z_i \rangle][p_i]$. The differential is defined in the same way as that of the bar resolution. Then it is clear that $A_j:=\bigoplus_{i=1} ^b \R \Gamma \otimes A^i _j$ is a direct summand of $\widetilde{\mathbf{B}}_j(\Gamma)$ and $(\widetilde{\mathbf{B}}_*(\Gamma), A_*)$ is an auxiliary resolution for the pair $(\Gamma,\mathcal{S})$.

By taking $\Hom_\Gamma(-,\mathfrak{g})$ on the complex $\widetilde{\mathbf{B}}_*(\Gamma)/ A_*$, we get the chain complex $C^*(\Gamma, \mathcal{S}; \mathfrak{g}):=\Hom_\Gamma(\widetilde{\mathbf{B}}_*(\Gamma)/ A_*, \mathfrak{g})$ for $H^*(\Gamma, \mathcal{S};\mathfrak{g})$. The map $j:H^1(\Gamma, \mathcal{S};\mathfrak{g}) \to H^1(\Gamma; \mathfrak{g})$ in (\ref{relativeseq})  sends the space of 1-cocycles $\ker(C^1(\Gamma, \mathcal{S};\mathfrak{g}) \to C^2(\Gamma, \mathcal{S};\mathfrak{g}))$ to $\widetilde{Z}^1_{\operatorname{par}} (\Gamma, \mathcal{S}; \mathfrak{g})$. This proves the first part of the lemma. 

For the second part, recall that both $\mathbf{B}_*(\Gamma)$ and $\widetilde{\mathbf{B}}_*(\Gamma)$ are projective resolutions, and $\operatorname{ext}:\mathbf{B}_*(\Gamma)\to \widetilde{\mathbf{B}}_*(\Gamma)$ is a chain map that commutes with the augmentations. Under this assumption, $\operatorname{ext}$ is a homotopy equivalence. (See Theorem 7.5 of Brown \cite{brown1982}) We use this homotopy equivalence $\operatorname{ext}$ to translate $\widetilde{Z}^1_{\operatorname{par}}(\Gamma,\mathcal{S};\mathfrak{g})$ into $Z^1_{\operatorname{par}}(\Gamma,\mathcal{S};\mathfrak{g})$. Observe that, by using the cocycle condition,
\[
u(\operatorname{ext}(\llbracket z_i \rrbracket))=  u (\llbracket w_i \zeta_i w_i ^{-1}\rrbracket)=u(\llbracket w_i \rrbracket) - z_i \cdot u(\llbracket w_i \rrbracket).
\]
Therefore, $\operatorname{ext}$ sends $\widetilde{Z}^1_{\operatorname{par}}(\Gamma,\mathcal{S};\mathfrak{g})$ to $Z^1_{\operatorname{par}}(\Gamma,\mathcal{S};\mathfrak{g})$. Conversely, let $u\in Z^1_{\operatorname{par}}(\Gamma,\mathcal{S};\mathfrak{g})$ and, for each $i$, let $X_i\in \mathfrak{g}$ be such that $u(\llbracket z_i \rrbracket ) = z_i\cdot X_i - X_i$. We define the 1-cocycle $u'\in \widetilde{Z}^1 (\Gamma; \mathfrak{g})$ by $u'(\llbracket x \rrbracket ) = u(\llbracket \operatorname{ret}(x)\rrbracket)$. Let $X\in \Hom_\Gamma (\widetilde{\mathbf{B}}_0(\Gamma),\mathfrak{g})$ be the 0-cocycle given $X(\llbracket p_0\rrbracket) = 0$ and $X(\llbracket p_i \rrbracket ) = X_i$ for $i=1,2,\cdots, b$. We define $\widetilde{u} = u' - \delta X$. We know that $\widetilde{u}\in\widetilde{Z}^1_{\operatorname{par}} (\Gamma, \mathcal{S}; \mathfrak{g})$ because
\begin{align*}
\widetilde{u}(\llbracket \zeta_i\rrbracket)  &= u(\llbracket \operatorname{ret}(\zeta_i) \rrbracket ) - (\operatorname{ret}(\zeta_i)\cdot X(\llbracket p_i \rrbracket ) - X(\llbracket p_i \rrbracket))\\
&=z_i \cdot X_i - X_i - (z_i \cdot X_i - X_i) \\
&= 0.
\end{align*}
We also compute, for $\llbracket x \rrbracket \in \mathbf{B}_1(\Gamma)$, 
\begin{align*}
\operatorname{ext}(\widetilde{u})(\llbracket x\rrbracket)&=\widetilde{u}(\llbracket\operatorname{ext}(x)\rrbracket)\\
& = u'(\llbracket \operatorname{ext}(x)\rrbracket)- \delta X(\llbracket\operatorname{ext}(x)\rrbracket)\\
&=u(\llbracket \operatorname{ret}(\operatorname{ext}(x))\rrbracket)-(x\cdot X(\llbracket p_0 \rrbracket)-X(\llbracket p_0 \rrbracket))\\
&=u(\llbracket x \rrbracket)
\end{align*}
to show that $\operatorname{ext}(\widetilde{Z}^1_{\operatorname{par}}(\Gamma,\mathcal{S};\mathfrak{g}))=Z^1_{\operatorname{par}}(\Gamma,\mathcal{S};\mathfrak{g})$.
\end{proof}

\section{The symplectic structure on orbifold Hitchin components}\label{symp}
In this section, we construct  the Atiyah-Bott-Goldman type symplectic form on the relative character variety $\Rep_n ^{\mathscr{B}}(\Gamma)$ of the orbifold group $\Gamma= \pi_1 ^{\text{orb}}({\mathcal{O}})$. 

In section \ref{poincaredualityconstruction}, we apply the Poincar\'{e} duality to produce a nondegenerate anti-symmetric bilinear pairing $\omega_{PD}$ on $T_{[\rho]}\Rep_n ^{\mathscr{B}}(\Gamma)$. Lemma \ref{p2} and the fact that $\Gamma$ is of type FP over $\R$, which is proven in the previous section, guarantee that $\Gamma$ has the  Poincar\'{e} duality map. In Lemma \ref{fundamentalclass}, we give an explicit form of the fundamental cycle of $\Gamma$ which leads to the concrete formula for the pairing $\omega_{PD}$ in Theorem \ref{explicit}. Then we embed $\Rep_n ^{\mathscr{B}}(\Gamma)$ into $\Rep^{\mathscr{B}'}_n (\pi_1(X'_{\mathcal{O}}))$ as a open submanifold where $X'_{\mathcal{O}}$ is  some compact surface obtained by removing cone points of $\mathcal{O}$. We show that $\omega_{PD}$ is the pull-back of the (well-known) Atiyah-Bott-Goldman symplectic form on $\Rep^{\mathscr{B}'}_n (\pi_1(X'_{\mathcal{O}}))$. 

In section \ref{deRham}, we construct another 2-form $\omega_H$ defined on a neighborhood of $\Hom^{\mathscr{B}}(\Gamma, G)$ by modifying Huebschmann \cite{huebschmann1995} and Guruprasad-Huebschmann-Jeffrey-Weinstein \cite{guruprasad1997}. From its construction we know that the 2-form $\omega_H$ is equivariantly closed in a small neighborhood of $\Hom^{\mathscr{B}}(\Gamma, G)$. We  show  in Theorem \ref{comp} that, after taking the symplectic reduction, $\omega_H = - \omega_{PD}$ on $\Rep_n ^{\mathscr{B}}(\Gamma)$.   

Consequently, we know that $\omega:=\omega_{PD}=-\omega_H$ is a closed non-degenerate 2-form on $\Rep_n ^{\mathscr{B}}(\Gamma)$ and this form $\omega$  will be our Atiyah-Bott-Goldman symplectic form. 

Throughout this section $G$ denotes the Lie group $\PSL_n(\R)$ and $\mathfrak{g}$ its Lie algebra. We also adopt the following notation.
\begin{notation}\label{notation}
Let $\mathcal{O}$ be a compact oriented 2-orbifold with boundary of negative Euler characteristic. Let us choose $z_1, \cdots, z_b$ the primitive peripheral elements of $\pi_1^{\operatorname{orb}}(\mathcal{O})$. Then we denote by $\per{\pi_1 ^{\operatorname{orb}}(\mathcal{O})}$ the set  $\{\langle z_1\rangle, \cdots, \langle z_b\rangle \}$. Although $\per{\pi_1 ^{\operatorname{orb}}(\mathcal{O})}$ is defined only  up to conjugation, this ambiguity does not make any difference in our discussion. 
\end{notation}

\subsection{Construction by the Poincar\'{e} duality}\label{poincaredualityconstruction}
As in the previous section, denote by $\mathcal{O}$ a compact orientable 2-orbifold of negative Euler characteristic and $\Gamma$ the orbifold fundamental group $\pi_1 ^{\operatorname{orb}}(\mathcal{O})$. We continue to use the presentation (\ref{present}) for $\Gamma$. 

Let $F$ be the free group on $\{x_1, y_1,\cdots, x_g, y_g, s_1, \cdots, s_c, z_1, \cdots, z_b\}$ and let
\[
F^{\#}:= \langle x_1, y_1, \cdots, x_g, y_g, s_1, \cdots, s_{c},z_1, \cdots, z_b | \mathbf{r}_1= \cdots = \mathbf{r}_{c} =1 \rangle
\]
where $\mathbf{r}_i=s_i ^{r_i}$.  Let $Q_i = \langle s_i\rangle$ be a cyclic subgroup  of $F^{\#}$ generated by $s_i$.  As $H^2(Q_i, \mathfrak{g}_\rho)=0$ for any $\rho\in \Hom(Q_i,G)$, $\Hom(Q_i,G)$ is a smooth embedded subspace of $\Hom(\Z ,G)=G$ and  
\[
\Hom(F^{\#},G) = G^{2g+b} \times \Hom(Q_1,G)\times \cdots\times \Hom(Q_c, G)
\]
is a smooth embedded submanifold of $G^{2g+b+c}$. 

The tangent space of $G$ at $k\in G$ can be identified with $\mathfrak{g}$ as follows. Recall that the Lie algebra $\mathfrak{g}$ is the tangent space of $G$ at the identity element. Write $R_k: G\to G$ the right translation defined by $R_k (h) = hk$. This map induces the tangent map $\mathfrak{g} \to T_k G$. Similarly for  \[\rho=(k_1, \cdots, k_{2g+b+c}) \in \Hom(F^{\#},G) \subset G^{2g+b+c},\]  write $R_{\rho*}: \mathfrak{g}^{2g+b+c} \to T_\rho \Hom(F^{\#},G)$ the tangent map of the coordinate-wise right translation  $R_{\rho} : G^{2g+b+c}\to G^{2g+b+c}$  defined by \[R_\rho (h_1, \cdots, h_{2g+b+c}) = (h_1 k_1, \cdots, h_{2g+b+c} k_{2g+b+c}).\] 

We have the conjugation action of $G$ on $\Hom^{\mathscr{B}}(F^{\#}, G)$. Each $X\in \mathfrak{g}$ induces the fundamental vector field $\overline{X}$ defined by the formula $\overline{X} _\rho (f) = \frac {\der }{\der t} f(\exp (-tX) \cdot \rho )$, where $f\in C^\infty (\Hom (F^{\#}, G))$, and where $\rho \in \Hom (F^{\#}, G)$. This gives rise to the linear map $A: X \mapsto \overline{X}_\rho$ from $\mathfrak{g}$ to $T_\rho \Hom ^{\mathscr{B}}(F^{\#},G)$.

Finally let $E: \Hom(F^{\#}, G) \to G$ be the map defined by  $E(\rho) = \rho(\mathbf{r})$.  Then we know that $\Hom(\Gamma,G) = E^{-1}(1)$.

\begin{lemma}\label{tangent1}Let $\mathbf{G}_*' \to \R \to 0$ be the resolution (\ref{resolution'}).  The cochain complex $C^* (\Gamma ;\mathfrak{g})=\Hom_{\Gamma}(\mathbf{G}_* ' ,\mathfrak{g})$  that computes $H^*(\Gamma;\mathfrak{g})$  fits into the commutative diagram
\[
\xymatrix{C^0 (\Gamma; \mathfrak{g})  \ar[rr] \ar[d]^{=}&& C^1  (\Gamma; \mathfrak{g})  \ar[rr] \ar[d]^{R_{\rho*}}&& C^2 (\Gamma; \mathfrak{g})\ar[d]^{R_{E(\rho)*}} \\ 
T_e G \ar[rr]_{A}  && T_{\rho} \Hom (F^{\#},G)  \ar[rr]_{ E_*} && T_{E(\rho)} G 
}
\]
where $R_\rho$ and $R_{E(\rho)}$ are the right translation maps. 
\end{lemma}
\begin{proof}
For each $j=1,2,\cdots, c$, let $\mathfrak{t}_j = \{X\in \mathfrak{g}\,|\, X+ s_j\cdot X + \cdots + s_j ^{r_j-1}\cdot X = 0\}$. Observe that
\[
C^1(\Gamma; \mathfrak{g}) = \mathfrak{g}^{2g+b}\times \mathfrak{t}_1 \times \cdots \times \mathfrak{t}_c.
\]
Since $\Hom(F^{\#}, G) = G^{2g+b}\times \Hom(Q_1,G) \times \cdots \times \Hom(Q_c, G)$, it is enough to check that $\mathfrak{t}_j$ can be identified with $T_\rho \Hom(Q_j, G)$ under the right translation. For this, we show that if $\rho_t$ is an analytic curve in $\Hom(Q_j, G)$ with $\rho_0 = \rho$ then $X_i:=\left(\frac{d}{dt} \rho_t (s_i) |_{t=0} \right)\rho(s_i)^{-1}$ is in $\mathfrak{t}_i$. Indeed, since $\rho_t (s_i)^{r_i} = 1$, we have that
\begin{align*}
0&=X_i\rho(s_i) ^{r_i}  + \rho(s_i) X_i \rho(s_i)^{r_i -1}  +  \cdots + \rho(s_i) ^{r_i-1}X_i \rho(s_i)\\
& = X_i + s_i \cdot X_i + \cdots + s_i ^{r_i-1} \cdot X_i 
\end{align*}
as desired. 

Commutativity of the diagram can be seen by straightforward computations. 
\end{proof}

The following corollary is known for closed surface groups. Here we record the same result for orbifold groups. 

\begin{corollary}\label{smoothpoint}
If $H^2(\Gamma; \mathfrak{g}_\rho)=0$ then $\rho$ has a neighborhood that is an analytic submanifold of $\Hom(F^{\#}, G)$. 
\end{corollary}
\begin{proof}
From Lemma \ref{tangent1}, the differential $E_*$ at $\rho$ is surjective if and only if $H^2(\Gamma; \mathfrak{g}_\rho)=0$. The result follows from the inverse function theorem. 
\end{proof}

The following lemma is a relative version of Lemma \ref{tangent1}. 
\begin{lemma}\label{tangent}There is a chain complex $C^* _{\operatorname{par}}(\Gamma, \per{\Gamma};\mathfrak{g})$ (will be explicitly constructed in the proof (\ref{paraboliccomplex2})) that computes $H^1_{\operatorname{par}}(\Gamma, \per{\Gamma};\mathfrak{g})$ and that fits into the commutative diagram
\[
\xymatrix{C^0 _{\operatorname{par}} (\Gamma, \per{\Gamma}; \mathfrak{g})  \ar[rr] \ar[d]^{=}&& C^1 _{\operatorname{par}} (\Gamma, \per{\Gamma}; \mathfrak{g})  \ar[rr] \ar[d]^{R_{\rho*}}&& C^2 _{\operatorname{par}} (\Gamma, \per{\Gamma}; \mathfrak{g})\ar[d]^{R_{E(\rho)*}} \\ 
T_e G \ar[rr]_{A}  && T_{\rho} \Hom ^{\mathscr{B}}(F^{\#},G)  \ar[rr]_{ E_*} && T_{E(\rho)} G 
}
\]
where $R_\rho$ and $R_{E(\rho)}$ are the right translation maps. 

In particular, the right translation identifies $H_{\operatorname{par}}^1(\Gamma, \per{\Gamma};\mathfrak{g}_\rho)$ with the (Zariski) tangent space $T_{[\rho]}\Rep^{\mathscr{B}}_ n(\Gamma)$ at $[\rho]$.
\end{lemma}

\begin{proof}  We begin with a modified version of the resolution (\ref{georesol2}). Let $Q$ be the submodule of $\widetilde{\mathbf{G}}_2$ generated by $\llbracket \mathbf{r}_1\rrbracket  , \cdots, \llbracket \mathbf{r}_{c}\rrbracket  $. Define
\begin{align*}
\widetilde{\mathbf{G}}_i' & = 0 , \quad i\ge 3\\
\widetilde{\mathbf{G}}_2'&= \R\Gamma [ \mathbf{r}]\\
\widetilde{\mathbf{G}}_1' &= \widetilde{\mathbf{G}}_1/ \partial (Q)\\
\widetilde{\mathbf{G}}_0'&= \widetilde{\mathbf{G}}_0.
\end{align*}

Using the same argument of the proof of Proposition \ref{FP}, we know that $\widetilde{\mathbf{G}}_* ' \to \R$ is a projective resolution of $\R$. 

For each $i$,  we define the $\R Q_i$-module $A_* ^i$ by
\[
A^i _0 = \R Q_i [p_i] , \quad A^i _1 = \R Q_i [ \zeta_i ], \quad \text{and} \quad  A^i _j =0 \text{ for }j\ge 2.
\]
Then $A_*:= \bigoplus_{i=1} ^ b \R \Gamma \otimes A_* ^i$ is a direct summand of the projective resolution $\widetilde{\mathbf{G}}' _*\to \R$. Therefore we have the auxiliary projective resolution $A_* ^i\to \R$ of $\R$.

Observe that
\begin{align*}
\widetilde{\mathbf{G}}'_0/A_0 &\cong \R\Gamma[p_0]\\
\widetilde{\mathbf{G}}'_1/A_1 &\cong \R\Gamma[x_1,y_1, \cdots,x_g,y_g, s_1, \cdots, s_{c},w_1, \cdots, w_b]/\partial(Q)\\
\widetilde{\mathbf{G}}'_i/A_i&\cong \widetilde{\mathbf{G}}'_i, \quad i\ge 2.
\end{align*}
We apply $\Hom_\Gamma(-,\mathfrak{g})$ to the above resolution $\widetilde{\mathbf{G}}'_*/A_*$. It follows that chain complex that computes the relative cohomology $H^1(\Gamma,\per{\Gamma};\mathfrak{g})$ is 
\begin{align*}
\widetilde{C}^0(\Gamma, \per{\Gamma};\mathfrak{g})&= \mathfrak{g}\\
\widetilde{C}^1(\Gamma, \per{\Gamma};\mathfrak{g})&=\mathfrak{g}^{2g}\times \mathfrak{t}_1\times\cdots \times \mathfrak{t}_{c} \times \mathfrak{g}^b \\
\widetilde{C}^2(\Gamma, \per{\Gamma};\mathfrak{g}) &= \mathfrak{g}
\end{align*}
where $\mathfrak{t}_j = \{X\in \mathfrak{g}\,|\, X+ s_j\cdot X + \cdots + s_j ^{r_j-1}\cdot X = 0\}$. 

The exact sequence of complexes $0\to A_* \to \widetilde{\mathbf{G}}'_* \to \widetilde{\mathbf{G}}'_*/A_*\to 0$ leads us to the following commutative diagram
\[
\xymatrix{
\widetilde{C}^0(\Gamma, \per{\Gamma};\mathfrak{g}) \ar[rr]\ar[d]^{\operatorname{inclusion}} && \widetilde{C}^1(\Gamma,\per{\Gamma}; \mathfrak{g}) \ar[rr] \ar[d]^{\operatorname{inclusion}} && \widetilde{C}^2(\Gamma, \per{\Gamma};\mathfrak{g})\ar[d]^{=}\\
\widetilde{C}^0(\Gamma; \mathfrak{g}) \ar[rr] && \widetilde{C}^1(\Gamma; \mathfrak{g}) \ar[rr] && \widetilde{C}^2(\Gamma; \mathfrak{g})
}
\]
where $\widetilde{C}^i(\Gamma;\mathfrak{g}) = \Hom_\Gamma (\widetilde{\mathbf{G}}'_i , \mathfrak{g})$. By the definition of the parabolic cohomology, we have that
\begin{equation}\label{paracomputation}
H^1_{\operatorname{par}}(\Gamma,\per{\Gamma};\mathfrak{g}) = \frac{\ker (\widetilde{C}^1(\Gamma, \per{\Gamma};\mathfrak{g}) \to \widetilde{C}^2(\Gamma, \per{\Gamma};\mathfrak{g}))}{\operatorname{im}(\widetilde{C}^0(\Gamma;\mathfrak{g}) \to \widetilde{C}^1(\Gamma; \mathfrak{g}))\cap \widetilde{C}^1 (\Gamma,\per{\Gamma};\mathfrak{g})}
\end{equation}
where the intersection $\operatorname{im}(\widetilde{C}^0(\Gamma;\mathfrak{g}) \to \widetilde{C}^1(\Gamma; \mathfrak{g}))\cap \widetilde{C}^1 (\Gamma,\per{\Gamma};\mathfrak{g})$ in the denominator takes place in $\widetilde{C}^1(\Gamma;\mathfrak{g})$ after identifying $\widetilde{C}^1 (\Gamma,\per{\Gamma};\mathfrak{g})$ with its image under the inclusion.

Note that we can also identify $\widetilde{C}^0 (\Gamma;\mathfrak{g})$ with $\mathfrak{g}\times \mathfrak{g}^b$. Take the subspace
\[
\mathfrak{g}\times \prod_{i=1} ^b  \ker( \Ad_{\rho(z_i)} -1)
\]
of $\widetilde{C}^0 (\Gamma;\mathfrak{g})=\mathfrak{g}\times \mathfrak{g}^b$. Let  $X=(X_0,X_1,\cdots,X_{b})$ be in this subspace, where $X_0$ is in the first $\mathfrak{g}$ factor and $X_i\in \ker(\operatorname{Ad}_{\rho(z_i)}-1)$. Then, we know that $\delta X \in \widetilde{C}^1(\Gamma,\per{\Gamma}; \mathfrak{g})$ as $\delta X (\llbracket \zeta_i \rrbracket) = \operatorname{Ad}_{\rho(z_i)} ( X_i ) - X_i = 0$. Moreover, every element in the denominator of (\ref{paracomputation}) can be realized in this way. Hence, $H^1 _{\text{par}}(\Gamma,\per{\Gamma};\mathfrak{g})$ can be computed from the following subcomplex
\begin{equation}\label{paraboliccomplex}
\mathfrak{g}\times \prod_{i=1} ^b  \ker( \Ad_{\rho(z_i)} -1) \to \widetilde{C}^1(\Gamma, \per{\Gamma};\mathfrak{g}) \to \widetilde{C}^2(\Gamma,\per{\Gamma};\mathfrak{g})
\end{equation}
where $ \widetilde{C}^1(\Gamma, \per{\Gamma};\mathfrak{g})$ and $\widetilde{C}^2(\Gamma, \per{\Gamma};\mathfrak{g})$ are regarded as subspaces of $ \widetilde{C}^1(\Gamma;\mathfrak{g})$ and $ \widetilde{C}^2(\Gamma;\mathfrak{g})$ respectively. 

We translate the above expression to the bar complex. For this purpose we write down the explicit chain homotopy equivalence between $\Hom_\Gamma (\widetilde{\mathbf{G}}'_*, \mathfrak{g})$ and $\Hom_\Gamma(\mathbf{G}'_* , \mathfrak{g})$. The chain homotopy equivalence $ \Hom_\Gamma(\widetilde{\mathbf{G}}'_*, \mathfrak{g})\to \Hom_\Gamma(\mathbf{G}'_* , \mathfrak{g})$ is induced from the chain map  $f_*:\mathbf{G}'_* \to \widetilde{\mathbf{G}}'_*$ given by
\begin{gather*}
f_i =\operatorname{Id},\quad i\ge 2\\
f_1: \llbracket x_i\rrbracket   \mapsto \llbracket x_i\rrbracket  ,\quad  \llbracket y_i\rrbracket   \mapsto \llbracket y_i\rrbracket   ,\quad \llbracket z_i\rrbracket   \mapsto \llbracket w_i\rrbracket   +\llbracket \zeta_i\rrbracket   - z_i \llbracket w_i\rrbracket   , \quad \llbracket s_i\rrbracket   \mapsto \llbracket s_i\rrbracket  \\
f_0: \llbracket p_0\rrbracket   \mapsto \llbracket p_0\rrbracket  
\end{gather*}
It follows that the subcomplex (\ref{paraboliccomplex}) can be translated into 
\begin{equation}\label{paraboliccomplex2}
C^*_{\operatorname{par}}(\Gamma,\per{\Gamma};\mathfrak{g}): \mathfrak{g} \to \mathfrak{g}^{2g} \times \mathfrak{t}_1\times \cdots \times \mathfrak{t}_{c}\times  \mathfrak{h}_1 \times \cdots \times \mathfrak{h}_b\to \mathfrak{g}
\end{equation}
where $\mathfrak{h}_i = \operatorname{im}(1- \Ad_{\rho(z_i)})$.

In light of Lemma \ref{tangent1}, we know that the tangent space of $\Hom(F^{\#}, G)$ can be identified with the subspace $\mathfrak{g}^{2g} \times \mathfrak{t}_1\times \cdots \times \mathfrak{t}_{c}\times  \mathfrak{g}_1 \times \cdots \times \mathfrak{g}_b$ of $C^1_{\operatorname{par}}(\Gamma,\per{\Gamma};\mathfrak{g})$ by the right translation $R_\rho$. Here $\mathfrak{g}_i = \mathfrak{g}$ for all $i=1,2,\cdots,b$. We now show that each $\mathfrak{h_i}\subset \mathfrak{g}_i$ is the tangent space to the orbit  of the conjugation action $B_i=\{g  \rho(z_i)g^{-1}\,|\,g\in G\}$  through $\rho(z_i)$. For if $g_t$ is an analytic curve in $B_i$ passing through $g_0=\rho(z_i) \in B_i$, then  we can find an analytic curve $h_t$ in $G$ with $h_0 = 1$  such that $g_t = h_t g_0 h_t ^{-1}$. Taking the derivative at $t=0$ yields $(\frac{ d g_t}{dt} |_{t=0} )g_0 ^{-1} = X_i - g_0 X_i g_0 ^{-1}$ where $X_i = \frac{d h_t}{d t}|_{t=0} \in \mathfrak{g}_i$. It follows that $T_{\rho(z_i)} B_i \subset \mathfrak{h}_i$. The reverse inclusion is achieved by setting $h_t = \exp t X_i$. We thus have the identification $R_{\rho*} : C^1_{\operatorname{par}}(\Gamma, \per{\Gamma}; \mathfrak{g}) \to T_\rho \Hom ^{\mathscr{B}}(F^{\#},G)$. 

Finally, it is a straightforward computation to show that the diagram in the statement is commutative. \end{proof}

Before we move on, we introduce the following definition which can be seen as a relative version of Poincare dual groups.

\begin{definition}[see also Bieri-Eckmann \cite{bieri1978}]
A group pair $(\Gamma,\mathcal{S})$ is called $PD^2 _{\R}$-pair if there is a distinguished  class, called a relative fundamental class, $[\Gamma,\mathcal{S}]\in H_2(\Gamma,\mathcal{S}; \R)$ such that
\[
\cap [\Gamma,\mathcal{S}]: H^* (\Gamma, \mathcal{S}; \R) \to H_{2-*} (\Gamma; \R)
\]
and
\[
\cap [\Gamma,\mathcal{S}] : H^* (\Gamma; \R) \to H_{2-*} (\Gamma, \mathcal{S}; \R)
\]
are isomorphisms.\end{definition}

\begin{lemma}\label{p2}
Let $\mathcal{O}$ be a compact orientable 2-orbifold of negative Euler characteristic.  Let $\Gamma= \pi_1 ^{\operatorname{orb}}(\mathcal{O})$.  Then $H^k (\Gamma; \R \Gamma )=0$ for $k\ne 1$ and $H^1 (\Gamma; \R\Gamma)\cong \Delta$ where $\Delta$ is the kernel of the augmentation $\bigoplus_{\langle z \rangle \in \per{\Gamma}}  \R \Gamma\otimes_{\langle z  \rangle }\R   \to \R$.
\end{lemma}
\begin{proof}
Let $C_i\subset \partial \mathcal{O}$ be the boundary component corresponding to the generator $z_i\in \Gamma$. Take a finite index torsion-free normal subgroup $N$ in $\Gamma$. $N$ is necessarily the fundamental group of a compact surface with boundary, say $S$ and we have a finite covering map $p:S\to \mathcal{O}$ with the deck group  $D:= \Gamma/N=\{\gamma_1, \cdots, \gamma_d\}$. Since there is no singular points on the boundary of $\mathcal{O}$, we know that $D$ acts simply transitively on the set of connected components of $p^{-1}(C_i)$. Hence $N$ admits a presentation
 \begin{multline*}
 N=\langle x_1',y_1', \cdots, x_{g'}', y_{g'}',\\
 \gamma_1 z_1 \gamma_1 ^{-1}, \cdots,\gamma_1 z_b\gamma_1^{-1}, \gamma_2 z_1 \gamma_2 ^{-1}, \cdots,\gamma_2 z_b\gamma_2^{-1},\\
 \cdots,\gamma_d z_1 \gamma_d ^{-1}, \cdots,\gamma_d z_b\gamma_d ^{-1}\,|\,\prod_i [x_i',y_i'] \prod_{i,j} \gamma_j z_i \gamma_j ^{-1}=1\rangle. 
 \end{multline*}
Let $\per{N}=\{\langle \gamma_j z_i \gamma_j ^{-1} \rangle \}_{i=1,2,\cdots,b} ^{ j=1,2,\cdots,d}$.  Then, by  Theorem 6.3 of Bieri-Eckmann \cite{bieri1978},  we know that $(N,\per{N})$ is a $PD^2 _{\Z}$-pair. Let $\Delta_N$ be the kernel of the augmentation map $\bigoplus_{\langle z \rangle\in\per{N}} \R N \otimes_{\langle z \rangle } \R \to \R$. 

Since $N$ is the fundamental group of the compact surface $S$ with boundary, \cite[Theorem 6.3]{bieri1978} shows that $H^q(N;\R\Gamma)=0$ for $q\ne  1$ and $H^1(N;\R N) = \Delta_N$. We also have the Lyndon-Hochschild-Serre spectral sequence whose second page is 
\[
E_2 ^{p,q} = H^p(\Gamma/N; H^q(N;\R\Gamma)). 
\] 
Recall that this spectral sequence converges to $H^{p+q}(\Gamma; \R\Gamma)$. Consequently, we know that 
\[
H^*(\Gamma;\R\Gamma)=H^0(\Gamma/N; H^*(N;\R\Gamma))=H^*(N; \R\Gamma)^{D},
\]where $H^*(N; \R\Gamma)^{D}$ is the space of fixed points under the action of the Deck transformation group $D$.

Observe that as an $N$-module, $\R\Gamma$ can be decomposed as $\bigoplus_{\gamma\in D}\R [N \gamma]$ where $N\gamma=\{n \gamma\,|\, n\in N\}$. This gives rise to the decomposition 
\[
H^*(N; \R \Gamma) = \bigoplus_{\gamma\in D}H^*(N;\R N) \gamma.
\]
Hence, we have that $H^q(\Gamma;\R\Gamma) =0$ for all $q\ne 1$ and $H^1(\Gamma;\R\Gamma)=(\bigoplus_{\gamma\in D} \Delta_N \gamma)^D$. Because the action of $D=\Gamma/N$ on  $\bigoplus_{\gamma\in D}   \Delta_N \gamma$  is the left multiplication, the set of the fixed points is isomorphic (as abelian groups) to $\Delta_N$. There is a further isomorphism $\Delta_N \to \Delta$ given by sending $x\otimes r \in \R N \otimes_{\langle \gamma_i z_j \gamma_i ^{-1}\rangle } \R$ to $x \gamma_i \otimes   r\in  \R \Gamma \otimes_{\langle z_j\rangle } \R$. This proves that $H^1(\Gamma;\R\Gamma) \cong \Delta$. 
\end{proof}

\begin{proposition}\label{product}
Let $\mathcal{O}$ be a compact orientable 2-orbifold of negative Euler characteristic and $\Gamma= \pi_1 ^{\operatorname{orb}}(\mathcal{O})$.  Then $(\Gamma, \per{\Gamma})$ is a $PD^2 _{\R}$-pair. 
\end{proposition}
\begin{proof}
Due to Proposition \ref{FP}, Lemma \ref{p2} and Theorem 6.2 of \cite{bieri1978}, $(\Gamma,\per{\Gamma})$ is a $PD^2_\R$-pair. 
\end{proof}

Recall that we also have the cup product
\[
\cup: H^1(\Gamma, \per{\Gamma};\mathfrak{g})\otimes H^1(\Gamma;\mathfrak{g}) \to H^2 (\Gamma, \per{\Gamma};\mathfrak{g}\otimes\mathfrak{g})\overset{\Tr}{\to} H^2(\Gamma, \per{\Gamma};\R),
\]
where $\Tr:\mathfrak{g} \otimes \mathfrak{g} \to \R$ denotes the trace pairing, $X\otimes Y \mapsto \Tr (XY)$, acting on the coefficients. Composing this cup product with the cap product $\cap [\mathcal{O},\partial\mathcal{O}]$, we get the pairing 
\[
\omega_{PD}: H^1(\Gamma, \per{\Gamma};\mathfrak{g})\otimes H^1(\Gamma;\mathfrak{g}) \overset{\cup}{\to}H^2(\Gamma,\per{\Gamma};\R) \overset{\cap[\mathcal{O},\partial\mathcal{O}]}{\to} H_0(\Gamma;\R)=\R.
\]

We now present the explicit form of $\omega_{PD}$. The groupoid bar resolution is the best choice for this purpose. 
\begin{notation}
We adopt the following convention: $\llbracket a\pm b | c\rrbracket = \llbracket a  | c \rrbracket \pm \llbracket b | c \rrbracket$ for all $a,b,c\in F \setminus \{1\}$. For instance,
\[
\left\llbracket\left. \frac{\partial [x,y]}{\partial y} \right| y\right\rrbracket = \llbracket\left. x-xyx^{-1}y^{-1} \right| y\rrbracket = \llbracket\left. x\right| y\rrbracket -\llbracket xyx^{-1}y^{-1}|y \rrbracket. 
\]
\end{notation}

\begin{lemma}\label{fundamentalclass}
Let
\begin{equation}\label{prefundamentalcycle}
\mathfrak{c}= \sum_{i=1} ^g\left( \left\llbracket \left. \frac{\partial \mathbf{r}}{\partial x_i }\right| x_i \right\rrbracket +   \left\llbracket\left. \frac{\partial \mathbf{r}}{\partial y_i }\right| y_i \right\rrbracket \right) +  \sum_{i=1} ^b \left\llbracket\left. \frac{\partial \mathbf{r}}{\partial z_i }\right| z_i \right\rrbracket +\sum_{i=1} ^{c} \left\llbracket\left. \frac{\partial \mathbf{r}}{\partial s_i} \right| s_i \right\rrbracket -\sum_{i=1} ^{c} \frac{1}{r_i}\left\llbracket\left. \frac{\partial \mathbf{r}_i}{\partial s_i}\right|s_i\right\rrbracket
\end{equation}
be an element of $ \mathbf{B}_2(\Gamma)\otimes \R$. Here, we understand $\frac{\partial \mathbf{r}}{\partial x_i}$ and any other terms involving the Fox differential as elements of $\R\Gamma$.
Then
\[
\widetilde{\mathfrak{c}} := \operatorname{ext} ( \mathfrak{c}) - \sum_{i=1} ^b ( \llbracket w_i ^{-1} | w_i \zeta_i \rrbracket - \llbracket w_i \zeta_i | w_i^{-1} \rrbracket) \in (\widetilde{\mathbf{B}}_2(\Gamma)/A_2)\otimes \R
\]
represents the relative fundamental class $[\mathcal{O},\partial \mathcal{O}]\in H_2(\Gamma, \per{\Gamma};\R)$.
\end{lemma}

\begin{proof}
We use the auxiliary resolution $(\widetilde{\mathbf{G}}_*, A_*)$ and claim that the element
\begin{equation}\label{relfun}
\llbracket \mathbf{r}\rrbracket-\sum_{i=1} ^{c} \frac{1}{r_i}\llbracket \mathbf{r}_i \rrbracket \in  (\widetilde{\mathbf{G}}_2 /A_2)\otimes \R
\end{equation}
represents the relative fundamental class.  Since we know that $(\Gamma,\per{\Gamma})$ is a $PD_{\R}^2$-pair, it is enough to show that (\ref{relfun}) represents a non-trivial element of $H_2(\Gamma,\per{\Gamma};\R)$. For this, we compute
\begin{align*}
\partial\left(\llbracket \mathbf{r}\rrbracket-\sum_{i=1} ^c \frac{1}{r_i}\llbracket \mathbf{r}_i \rrbracket \right) &= \sum_{i=1} ^ b \llbracket \zeta_i \rrbracket + \sum_{i=1}^c \llbracket s_i \rrbracket -\sum_{i=1} ^c \frac{1}{r_i} \partial ( \llbracket \mathbf{r}_i \rrbracket )\\
&=\sum_{i=1} ^ b \llbracket \zeta_i \rrbracket +\sum_{i=1} ^c \llbracket s_i \rrbracket - \sum_{i=1} ^c \frac{1}{r_i} \epsilon (\pi (1+s_i + \cdots + s_i ^{r_i-1}))\llbracket s_i \rrbracket\\
&= \sum_{i=1} ^b \llbracket \zeta_i \rrbracket.
\end{align*}
Hence, $\partial\left(\llbracket \mathbf{r}\rrbracket-\sum_{i=1} ^c \frac{1}{r_i}\llbracket \mathbf{r}_i \rrbracket \right)\in A_1$. Here, we use the fact that $a\otimes b = 1 \otimes \epsilon (a) b$ in $\R\Gamma \otimes \R$. Since, $\partial(\widetilde{\mathbf{G}}_3)$ does not contain the $\llbracket \mathbf{r} \rrbracket$ component, we know that (\ref{relfun}) is nonzero in $H_2(\Gamma,\per{\Gamma};\R)$.

Now we explicitly write down the chain homotopy equivalence $f_*$ from $\widetilde{\mathbf{G}}_*$ to $\widetilde{\mathbf{B}}_*(\Gamma)$. The chain map $f_*$ sends the generator $\llbracket p_i\rrbracket \in \widetilde{\mathbf{G}}_0$ to $\llbracket p_i\rrbracket\in\widetilde{\mathbf{B}}_0(\Gamma)$ and acts similarly in degree 1. In degree 2, we define $f_*$ by
\begin{align*}
\llbracket \mathbf{r}\rrbracket\mapsto&  \sum_{i=1} ^g \operatorname{ext} \left( \left\llbracket \left. \frac{\partial \mathbf{r}}{\partial x_i }\right | x_i \right\rrbracket +   \left\llbracket\left. \frac{\partial \mathbf{r}}{\partial y_i }\right| y_i \right\rrbracket\right) \\ 
&\qquad+  \sum_{i=1} ^b\left( \left\llbracket\left. \frac{\partial \mathbf{r}}{\partial z_i } w_i \right| \zeta_i \right\rrbracket - \left\llbracket \left. \frac{\partial \mathbf{r}}{\partial z_i} w_i \right| w_i^{-1} \right\rrbracket+\left\llbracket\left. \frac{\partial \mathbf{r}}{\partial z_i } w_i \zeta_i\right|w_i^{-1} \right\rrbracket \right)+\sum_{i=1} ^{c} \left\llbracket\left.\frac{\partial \mathbf{r}}{\partial s_i} \right| s_i \right\rrbracket\\
\left\llbracket \mathbf{r}_i\right\rrbracket \mapsto& \operatorname{ext}\left(\left\llbracket\left. \frac{\partial \mathbf{r}_i}{\partial s_i} \right| s_i\right\rrbracket\right).
\end{align*}
Since both $\widetilde{\mathbf{G}}_*$ and $\widetilde{\mathbf{B}}_*$ are projective resolutions, the chain map $f_*$ inductively extends to all higher degrees. Our chain map $f_*$ takes the element (\ref{relfun}) to 
\begin{multline*}
\mathfrak{c}_0:=\sum_{i=1} ^g\left( \left\llbracket\left. \frac{\partial \mathbf{r}}{\partial x_i }\right| x_i \right\rrbracket +   \left\llbracket\left. \frac{\partial \mathbf{r}}{\partial y_i }\right| y_i \right\rrbracket \right) \\ 
+  \sum_{i=1} ^b\left( \left\llbracket\left. \frac{\partial \mathbf{r}}{\partial z_i } w_i \right| \zeta_i \right\rrbracket - \left\llbracket \left. \frac{\partial \mathbf{r}}{\partial z_i} w_i \right| w_i^{-1} \right\rrbracket+\left\llbracket\left. \frac{\partial \mathbf{r}}{\partial z_i } w_i \zeta_i\right|w_i^{-1} \right\rrbracket\right)+\sum_{i=1} ^{c} \left\llbracket\left.\frac{\partial \mathbf{r}}{\partial s_i} \right| s_i \right\rrbracket\\ 
-\sum_{i=1} ^{c} \frac{1}{r_i} \left\llbracket\left. \frac{\partial \mathbf{r}_i}{\partial s_i} \right| s_i\right\rrbracket.
\end{multline*}
By direct computation, we know that $\mathfrak{c}_0$ and $\widetilde{\mathfrak{c}}$ are homologues. Indeed $\mathfrak{c}_0 - \widetilde{\mathfrak{c}}$ is the boundary of the 3-chain
\[ 
\sum_{i=1} ^b \left(\left\llbracket\left.\left. \frac{\partial \mathbf{r} }{\partial z_i}w_i \right| w_i^{-1}\right| w_i \zeta_i \right\rrbracket - \left\llbracket \left. \left. \frac{\partial \mathbf{r}}{\partial z_i } \right| w_i \zeta_i \right| w_i ^{-1}\right\rrbracket \right)\in \widetilde{\mathbf{B}}_3\otimes \R.
\]
Hence the lemma follows. 
\end{proof}

Recall that we have the auxiliary resolution $(\widetilde{\mathbf{B}}_*(\Gamma), A_*)$ for the group pair $(\Gamma, \per{\Gamma})$. For the description of the peripheral part $A_*$, see the proof of Proposition \ref{paraboliccocycle}.

\begin{lemma}\label{diagonalapprox}
Let $\mathcal{O}$ be a compact orientable 2-orbifold of negative Euler characteristic, and $\Gamma= \pi_1 ^{\operatorname{orb}}(\mathcal{O})$. Let $(\widetilde{\mathbf{B}}_*(\Gamma), A_*)$ be the auxiliary resolution for a group pair $(\pi_1 ^{\operatorname{orb}}(\mathcal{O}), \per{\pi_1 ^{\operatorname{orb}}(\mathcal{O})})$. Let  $u\in \Hom_\Gamma (\widetilde{\mathbf{B}}_1(\Gamma)/A_1, \mathfrak{g})$ and $v\in  \Hom_\Gamma (\widetilde{\mathbf{B}}_1(\Gamma), \mathfrak{g})$. Then for any $\llbracket x | y\rrbracket\in  \widetilde{ \mathbf{B}}_2(\Gamma)\otimes \R$, we have
\[
\langle u\cup v,\llbracket x_1|x_2\rrbracket \rangle = \Tr( u(\llbracket x_1\rrbracket) \operatorname{ret}(x_1)\cdot v(\llbracket x_2\rrbracket )).
\]
\end{lemma}
\begin{proof}
This is the Alexander-Whitney diagonal approximation theorem. 
\end{proof}

\begin{lemma}\label{relll}
Suppose that $u$ is either in $Z^1(\Gamma, \per{\Gamma};\mathfrak{g})\subset \Hom_\Gamma(\widetilde{\mathbf{B}}_1(\Gamma)/A_1, \mathfrak{g})$ or in $Z^1_{\operatorname{par}}(\Gamma, \per{\Gamma};\mathfrak{g})\subset \Hom_{\Gamma} (\mathbf{B}_1(\Gamma),\mathfrak{g})$. Then for each $s_i$, there exists $T_i\in  \mathfrak{g}$ such that $u(\llbracket s_i\rrbracket) = s_i\cdot  T_i-T_i$. Moreover if $T_i'\in \mathfrak{g}$ are other choices with the same property, $u(\llbracket s_i \rrbracket) = s_i \cdot T_i ' - T_i'$, then we have $\Tr(T_i v(\llbracket s_i\rrbracket)) = \Tr(T_i ' v(\llbracket s_i\rrbracket))$ for any $v\in Z^1 _{\operatorname{par}}(\Gamma,\per{\Gamma};\mathfrak{g})$. 
\end{lemma}
\begin{proof}
The proof is along the same lines as that of Lemma \ref{zerodivisor}. More precisely, let
\[
T_i = -\frac{1}{r_i} \left( u(\llbracket s_i\rrbracket) + u(\llbracket s_i ^2\rrbracket) + \cdots + u(\llbracket s_i ^{r_i-1}\rrbracket)\right).
\]
Then using the cocycle condition, we have that 
\begin{align*}
r_i (s_i \cdot T_i - T_i )& =  -s_i \cdot u(\llbracket s_i\rrbracket ) -s_i \cdot u(\llbracket s_i ^2\rrbracket )- \cdots - s_i\cdot u(\llbracket s_i ^{r_i-1}\rrbracket ) \\ 
&\qquad  + u(\llbracket s_i\rrbracket) + (u(\llbracket s_i\rrbracket  ) + s_i \cdot u(\llbracket s_i\rrbracket ) ) + (u(\llbracket s_i\rrbracket ) + s_i \cdot u(\llbracket s_i^2\rrbracket )) + \cdots\\
&\qquad \quad  \cdots + (u(\llbracket s_i\rrbracket ) + s_i \cdot u(\llbracket s_i ^{r_i-2}\rrbracket ))\\
&= (r_i -1) u(\llbracket s_i\rrbracket ) - s_i \cdot u(\llbracket s_i ^{r_i-1}\rrbracket ) \\
&= (r_i -1) u(\llbracket s_i\rrbracket ) + u(\llbracket s_i\rrbracket ) -u(\llbracket s_i ^{r_i}\rrbracket )\\
&= r_i u(\llbracket s_i\rrbracket ).
\end{align*}
Therefore the result follows. 

To prove the second statement, we observe that $s_i \cdot(T_i - T_i')=T_i-T_i'$ for each $i$. Let $D_i=T_i - T_i'$. Since $v\in Z^1(\Gamma, \per{\Gamma};\mathfrak{g})$ we can also  choose $V_i\in \mathfrak{g}$  so that $v(\llbracket s_i \rrbracket) = s_i\cdot V_i - V_i$ for each $i$.  Then we have 
\begin{align*}
\Tr(T_i  v(\llbracket s_i\rrbracket)) - \Tr(T_i ' v(\llbracket s_i\rrbracket)) & = \Tr(D_i  v(\llbracket s_i\rrbracket)) \\
&= \Tr (D_i (s_i\cdot V_i -V_i)) \\
&= \Tr(D_i s_i \cdot V_i) - \Tr(D_i V_i) \\
&= \Tr (D_i V_i ) -\Tr(D_i V_i) \\
&=0
\end{align*}
as we desired. 
\end{proof}

The following proposition and its proof is a slight variation of that of Guruprasad-Huebschmann-Jeffrey-Weinstein \cite{guruprasad1997}. 

To avoid cumbersome notation, we sometimes write $u(x)$ rather than $u(\llbracket x \rrbracket)$.  We also use the map $\overline{(\cdot)}:\mathbf{B}_1(\Gamma) \to \mathbf{B}_1 (\Gamma)$ defined by $\overline{\sum n_i\llbracket g_i\rrbracket }= \sum n_i \llbracket g_i ^{-1}\rrbracket$.

\begin{proposition}\label{explicitpre}
Let $\mathcal{O}$ be a compact oriented 2-orbifold of negative Euler characteristic and $\Gamma=\pi_1 ^{\operatorname{orb}}(\mathcal{O})$. Let $[\rho]\in \Rep_n ^{\mathscr{B}} (\Gamma)$. Let $u\in Z^1(\Gamma, \per{\Gamma}; \mathfrak{g}_\rho)\subset \Hom_\Gamma (\widetilde{\mathbf{B}}_1(\Gamma)/A_1 , \mathfrak{g}_\rho)$ and $v\in \widetilde{Z}^1_{\operatorname{par}}(\Gamma,\per{\Gamma};\mathfrak{g}_\rho)\subset \Hom_\Gamma(\widetilde{\mathbf{B}}_1(\Gamma), \mathfrak{g}_\rho)$. Let $T_i$ be the elements in $\mathfrak{g}$ defined in Lemma \ref{relll}. Then
\begin{align*}
\omega_{PD}([u],[v])&=-\sum_{i=1} ^ g \Tr \left( u' \left( \overline{\frac{\partial \mathbf{r}}{\partial x_i}}\right) v'(x_i) \right)-\sum_{i=1} ^g \Tr \left( u' \left(\overline{ \frac{\partial \mathbf{r}}{\partial y_i}}\right) v'(y_i) \right)\\
&\quad\quad\quad -\sum_{i=1} ^b \Tr \left( u' \left( \overline{\frac{\partial \mathbf{r}}{\partial z_i}}\right) v'(z_i) \right)-\sum_{i=1} ^{c} \Tr \left( u' \left( \overline{\frac{\partial \mathbf{r}}{\partial s_i}}\right) v'(s_i) \right)\\
&\qquad\qquad\qquad\qquad-\sum_{i=1} ^{c} \Tr  (T_i  v'(s_i)) -\sum_{i=1} ^b  \Tr (X_i v'(z_i)).
\end{align*}
where $u' = u \circ \operatorname{ext}$, $v'= v\circ \operatorname{ext}$,  and $X_i =- u(\llbracket w_i\rrbracket )$. 
\end{proposition}

\begin{proof}
By the definition of $\omega_{PD}$ and by Lemma \ref{fundamentalclass}, we have that 
\begin{align*}
\omega_{PD} ([u],[v]) & = \langle u \cup v, \widetilde{\mathfrak{c}}\rangle\\
&= \langle u\cup v,\mathfrak{c}'\rangle -\sum_{i=1} ^b \langle u \cup v , \llbracket w_i^{-1}|w_i\zeta_i\rrbracket-\llbracket w_i\zeta_i|w_i ^{-1}\rrbracket \rangle.
\end{align*}
where $\mathfrak{c}' = \operatorname{ext}( \mathfrak{c})$. Then by making use of Lemma \ref{diagonalapprox} we can evaluate
\begin{align*}
\langle u \cup v, \llbracket w_i^{-1} | w_i \zeta_i\rrbracket \rangle &= \Tr (u(\llbracket w_i^{-1}\rrbracket) \operatorname{ret}(w_i^{-1})\cdot v(\llbracket w_i \zeta_i \rrbracket))\\
&=\Tr (X_i v(\llbracket w_i \zeta_i \rrbracket))\\
&=\Tr(X_i v(\llbracket w_i \rrbracket)).
\end{align*}
For the last equality, we use the cocycle condition to get $v(\llbracket w_i \zeta_i \rrbracket)= v(\llbracket w_i\rrbracket)$. 
Likewise, we have that
\begin{align*}
\langle u \cup v, \llbracket w_i\zeta_i | w_i ^{-1}\rrbracket \rangle &=  - \Tr (X_i( \operatorname{ret}(w_i\zeta_i)\cdot v(\llbracket w_i ^{-1}\rrbracket)))\\ 
&= \Tr  (X_i (z_i \cdot  v(\llbracket w_i\rrbracket))).
\end{align*}
Observe that $v'(\llbracket z_i \rrbracket)=v(\llbracket \operatorname{ext} (z_i )\rrbracket) = v(\llbracket w_i \zeta_i w_i^{-1} \rrbracket)=v(\llbracket w_i\rrbracket) - z_i \cdot v(\llbracket w_i \rrbracket)$. Combining all, we get
\begin{align*}
 \sum_{i=1} ^b \langle u \cup v , \llbracket w_i^{-1}|w_i\zeta_i\rrbracket -\llbracket w_i\zeta_i|w_i ^{-1}\rrbracket \rangle &=\sum_{i=1} ^b \left(\Tr(X_i v(\llbracket w_i \rrbracket)) - \Tr (X_i(z_i \cdot v(\llbracket w_i \rrbracket )))\right)\\ &=\sum_{i=1} ^b \Tr(X_i (v(\llbracket w_i \rrbracket ) -z_i \cdot v(\llbracket w_i \rrbracket))) \\
 &= \sum_{i=1} ^b \Tr(X_i v'(\llbracket z_i\rrbracket )). 
\end{align*}

We can also compute $\langle u \cup v , \frac{1}{r_i}\llbracket \frac{\partial \mathbf{r}_i }{\partial s_i} | s_i \rrbracket \rangle$ as follows. First, the inductive application of the cocycle condition, shows that $u(\llbracket s_i ^j \rrbracket )= s_i ^j \cdot T_i - T_i$ for all $j\in \mathbb{Z}$, where $T_i$ is as in Lemma \ref{relll}. Due to Lemma \ref{diagonalapprox}, we know that
\begin{align*}
\left\langle u \cup v, \left.\left\llbracket \frac{\partial \mathbf{r}_i }{\partial s_i }\right| s_i \right\rrbracket \right\rangle &=\sum_{j=1} ^{r_i}\Tr (u(\llbracket s_i ^j \rrbracket ) s_i ^j \cdot v(\llbracket s_i \rrbracket)) \\
&= \sum_{j=1} ^ {r_i}\Tr ((s_i ^j T_i - T_i )s_i ^j\cdot v(\llbracket s_i \rrbracket))\\ 
&=\sum_{j=1} ^ {r_i } \Tr (T_i v(\llbracket s_i \rrbracket ))-\sum_{j=1} ^{r_i}\Tr (s_i ^j \cdot T_i v(\llbracket s_i \rrbracket)) \\ 
&= r_i \Tr (T_i v(\llbracket s_i \rrbracket)) - \Tr((T_i + s_i \cdot T_i + \cdots + s_i ^{r_i-1}\cdot T_i ) v(\llbracket s_i \rrbracket ) ).
\end{align*}
Since $T_i + s_i T_i + \cdots + s_i ^{r_i-1}T_i$ is an invariant element under the action of $s_i$, we argue as in the proof of Lemma \ref{relll} and this yields 
\[
\Tr((T_i + s_i \cdot T_i + \cdots + s_i ^{r_i-1}\cdot T_i ) v(\llbracket s_i \rrbracket )) =0.
\]
Hence, we have
\[
\left\langle u \cup v, \left.\left\llbracket \frac{\partial \mathbf{r}_i }{\partial s_i }\right| s_i \right\rrbracket \right\rangle =  \Tr (T_i v(\llbracket s_i \rrbracket )).
\]

By using Lemma \ref{diagonalapprox} and Lemma \ref{relll} and the identity 
\[
\Tr ( u(\llbracket g\rrbracket ) g\cdot v(\llbracket h \rrbracket)) = -\Tr (u(\llbracket g^{-1}\rrbracket)v(\llbracket h\rrbracket)),
\]
we can expand $\langle u \cup v ,\mathfrak{c}' \rangle=\langle u'\cup v' ,\mathfrak{c}\rangle$
\begin{multline*}
\langle u'\cup v' ,\mathfrak{c} \rangle=-\sum_{i=1} ^ g \Tr \left( u' \left( \overline{\frac{\partial \mathbf{r}}{\partial x_i}}\right) v'(x_i) \right)-\sum_{i=1} ^g \Tr \left( u' \left(\overline{ \frac{\partial \mathbf{r}}{\partial y_i}}\right) v'(y_i) \right)\\
-\sum_{i=1} ^b \Tr \left( u' \left( \overline{\frac{\partial \mathbf{r}}{\partial z_i}}\right) v'(z_i) \right)-\sum_{i=1} ^{c} \Tr \left( u' \left( \overline{\frac{\partial \mathbf{r}}{\partial s_i}}\right) v'(s_i) \right)-\sum_{i=1} ^{c} \Tr ( T_i  v'(s_i)). 
\end{multline*}
This yields the theorem.
\end{proof}

\begin{corollary}\label{welldefinedness}
Let $j:H^1(\Gamma,\per{\Gamma};\mathfrak{g})\to H^1(\Gamma, \mathfrak{g})$ be the map in (\ref{relativeseq}). For $[u] \in \ker j \subset H^1(\Gamma, \per{\Gamma}; \mathfrak{g}) $ and $[v]\in  \operatorname{im} j=H^1_{\operatorname{par}}(\Gamma, \per{\Gamma};\mathfrak{g})\subset H^1(\Gamma;\mathfrak{g})$, we have 
\[
\omega_{PD}([u],[v])=0.
\]
\end{corollary}
\begin{proof} 
Being $[u]\in \ker j$ means that $u=\delta X$ for some $X\in \Hom_\Gamma (\widetilde{\mathbf{B}}_0(\Gamma), \mathfrak{g})$. Define $X'\in \Hom_\Gamma (\widetilde{\mathbf{B}}_0(\Gamma)/A_0 , \mathfrak{g})$ to be $X'(\llbracket p_0 \rrbracket) = X(\llbracket p_0 \rrbracket)$ and $X'(\llbracket p_i\rrbracket)=0$ for $i=1,2,\cdots, b$. Then $[u-\delta X']=[u]$ in $H^1(\Gamma, \partial \Gamma;\mathfrak{g})$. 

Now we apply Proposition \ref{explicitpre} to conclude. Namely, we show that each term in the formula for $\omega_{PD}$ vanishes. 

Since $[u]= [u-\delta X']$ we can replace $u$ with $u-\delta X'$.  Observe that $(u-\delta X' )(\llbracket \gamma \rrbracket )=0$ whenever $o(\gamma)=t(\gamma)=p_0$.  This means that 
\[
(u-\delta X')\circ \operatorname{ext}\left( \overline{\frac{\partial \mathbf{r}}{\partial x_i}}\right)=(u-\delta X')\circ \operatorname{ext}\left( \overline{\frac{\partial \mathbf{r}}{\partial y_i}}\right)=0,
\] 
for $i=1,2,\cdots, g$, 
\[
(u-\delta X')\circ \operatorname{ext}\left( \overline{\frac{\partial \mathbf{r}}{\partial s_i}}\right)=0
\]
for $i=1,2,\cdots, c$, and
\[
(u-\delta X')\circ \operatorname{ext}\left( \overline{\frac{\partial \mathbf{r}}{\partial z_i}}\right)=0
\]
for $i=1,2,\cdots, b$. Moreover since $(u-\delta X')(\llbracket s_i \rrbracket )=0$, we can set $T_i=0$ for all $i$.  This shows that the first five terms in the formula for $\omega_{PD}$ are zero. 

Recall that in the statement of Proposition \ref{explicitpre} we found $X_i$ as 
\[
-X_i=(u-\delta X' )(\llbracket w_i \rrbracket) =\delta X(\llbracket w_i \rrbracket) - \delta X' (\llbracket w_i \rrbracket) = X(\llbracket p_i \rrbracket).
\]
Since $u=\delta X \in \Hom_\Gamma (\widetilde{\mathbf{B}}_1(\Gamma)/A_1, \mathfrak{g})$ we  have 
\[
0=u(\llbracket \zeta_i \rrbracket)=\delta X (\llbracket \zeta_i \rrbracket ) = z_i \cdot X(\llbracket p_i \rrbracket ) - X(\llbracket p_i \rrbracket)=X_i- z_i \cdot X_i.
\]
This implies that, for each $i$,
\[
\Tr ( X_i v'(z_i)) = \Tr(X_i (v(\llbracket w_i \rrbracket) - z_i \cdot v(\llbracket w_i\rrbracket )))= \Tr (( X_i-z_i ^{-1} \cdot X_i )v(\llbracket w_i \rrbracket)) =0. 
\]

It follows that $\omega_{PD}([u],[v])=\langle (u-\delta X' )\cup v ,\widetilde{\mathfrak{c}}\rangle =0$. 
\end{proof}

Recall that $H^1_{\operatorname{par}}(\Gamma, \per{\Gamma}; \mathfrak{g}) = H^1(\Gamma,\per{\Gamma};\mathfrak{g})/ \ker j$.  In particular, due to Corollary \ref{welldefinedness}, $\omega_{PD}$ gives rise to the well-defined paring 
\[
\omega_{PD} : H^1_{\operatorname{par}}(\Gamma, \per{\Gamma};\mathfrak{g})\otimes H^1_{\operatorname{par}}(\Gamma, \per{\Gamma};\mathfrak{g}) \to \R.
\]
We use the same letter $\omega_{PD}$ to denote the induced pairing. By Lemma \ref{tangent}, and the properties of cup and cap products, it is immediate that $\omega_{PD}$ is a 2-form on $\Rep_n ^{\mathscr{B}}(\Gamma)$ via the right translation. This 2-form on $\Rep_n ^{\mathscr{B}}(\Gamma)$ will also be denoted by $\omega_{PD}$.

\begin{theorem}\label{explicit}
Let $\mathcal{O}$ be a compact oriented 2-orbifold of negative Euler characteristic and $\Gamma=\pi_1 ^{\operatorname{orb}}(\mathcal{O})$. Let $[\rho]\in \Rep_n ^{\mathscr{B}} (\Gamma)$. Let $u,v\in Z^1_{\operatorname{par}}(\Gamma, \per{\Gamma};\mathfrak{g})\subset \Hom_{\Gamma}(\mathbf{B}_1(\Gamma),\mathfrak{g})$ be  parabolic cocycles in terms of the bar resolution.  We choose $X_i, T_j\in \mathfrak{g}$ such that $u(z_i)=z_i\cdot X_i-X_i$ and  $u(s_j)=s_j \cdot T_j - T_j$ for each $i=1,2,\cdots,b$ and $j=1,2,\cdots, c$ respectively. Then  $\omega_{PD}$ can be written explicitly as
\begin{align*}
\omega_{PD}([u],[v])&=\langle u\cup v , \mathfrak{c}\rangle -\sum_{i=1} ^b \Tr(X_i v(z_i))\\
&=-\sum_{i=1} ^ g \Tr \left( u \left( \overline{\frac{\partial \mathbf{r}}{\partial x_i}}\right) v(x_i) \right)-\sum_{i=1} ^g \Tr \left( u \left(\overline{ \frac{\partial \mathbf{r}}{\partial y_i}}\right) v(y_i) \right)\\
&\quad\quad\quad -\sum_{i=1} ^b \Tr \left( u \left( \overline{\frac{\partial \mathbf{r}}{\partial z_i}}\right) v(z_i) \right)-\sum_{i=1} ^{c} \Tr \left( u \left( \overline{\frac{\partial \mathbf{r}}{\partial s_i}}\right) v(s_i) \right)\\
&\qquad\qquad\qquad\qquad-\sum_{i=1} ^{c} \Tr  (T_i  v(s_i)) -\sum_{i=1} ^b  \Tr (X_i v(z_i)).
\end{align*}
\end{theorem}

\begin{proof}
Let $X,Y$ be elements of $\widetilde{\mathbf{B}}_0(\Gamma)$ such that 
\[
X(\llbracket p_i\rrbracket )= \begin{cases} X_i & i=1,2,\cdots,b \\ 0 & i=0\end{cases},\qquad 
Y(\llbracket p_i \rrbracket) = \begin{cases} Y_i& i=1,2,\cdots,b \\ 0 & i=0 \end{cases}. 
\]
Let $\widetilde{u}=u\circ \operatorname{ret}-\delta X$ and $\widetilde{v}=v\circ\operatorname{ret}-\delta Y$. Then $\widetilde{u}$ can be seen as a member of $Z^1(\Gamma, \per{\Gamma}; \mathfrak{g})\subset \Hom_\Gamma (\widetilde{\mathbf{B}}_1(\Gamma)/A_1, \mathfrak{g})$. Observe that $\widetilde{u}\circ \operatorname{ext} = u$ and $\widetilde{v}\circ \operatorname{ext} = v$. Now apply Proposition \ref{explicitpre} to get the conclusion. 
\end{proof}

Let $X_{\mathcal{O}}$ be the underlying space of $\mathcal{O}$. Let $\Sigma_{\mathcal{O}}$ be the set of singularities. Recall that we have defined  $X_{\mathcal{O}}':= X_{\mathcal{O}} \setminus \Sigma_{\mathcal{O}}$.  $X_{\mathcal{O}} '$ is a manifold and each puncture of $X_{\mathcal{O}} '$ corresponds to a singularity of $\mathcal{O}$. Hence the fundamental group is given by
\[
\pi_1(X_{\mathcal{O}}')=\langle x_1, y_1, \cdots, x_g, y_g, z_1, \cdots, z_b, s_1, \cdots, s_c\,|\,\prod_{i=1} ^g [x_i,y_i] \prod_{i=1} ^b z_i \prod_{i=1} ^c s_i\rangle. 
\]

\begin{lemma}\label{smooth}Let $\mathcal{O}$ be a compact oriented 2-orbifold of negative Euler characteristic and $\Gamma=\pi_1 ^{\operatorname{orb}}(\mathcal{O})$.  We have a smooth embedding
\[
\mathcal{I}:\Rep^{\mathscr{B}}_n(\Gamma) \to \Rep_n ^{\mathscr{B}'}(\pi_1(X' _{\mathcal{O}}))
\]
for some choice $\mathscr{B}'$ of conjugacy classes of boundary holonomies. 
\end{lemma}
\begin{proof}

The map $\mathcal{I}$ is defined by precomposing $\rho\in\Rep^{\mathscr{B}}_n(\pi_1^{\operatorname{orb}}(\mathcal{O}))$ with the projection map $\pi_1(X' _{\mathcal{O}})\to \pi_1 ^{\text{orb}}(\mathcal{O})$.

We know that  $\Rep_n(\pi_1(X'_{\mathcal{O}}))$ is foliated by leaves of the form $\bigcup_{\mathscr{Z}}\Rep^{\mathscr{Z}} _n(\pi_1(X'_{\mathcal{O}}))$.  Lemma \ref{relll} tells us that  the derivative $\der \mathcal{I}$ of $\mathcal{I}$ at each point of $\Rep^{\mathscr{B}}_n(\Gamma)$ maps isomorphically the tangent space $H^1_{\operatorname{par}}(\Gamma,\per{\Gamma};\mathfrak{g}_\rho)$ of $\Rep^{\mathscr{B}}_n(\Gamma)$ to the tangent space $H^1_{\text{par}}(\pi_1(X'_{\mathcal{O}}), \per{\pi_1(X'_{\mathcal{O}})};\mathfrak{g}_{\mathcal{I}(\rho)})$ of a leaf.  Therefore, the image of $\mathcal{I}$ must be contained in a single leaf, say  $\Rep_n ^{\mathscr{B}'}(\pi_1(X' _{\mathcal{O}}))$. 
\end{proof}

\begin{remark}
The image $\mathcal{I}(\mathcal{C}^{\mathscr{B}}(\mathcal{O}))$ is not in the Hitchin component of $X'_{\mathcal{O}}$. In fact, it is not even an Anosov representation because every element of  $\mathcal{I}(\mathcal{C}^{\mathscr{B}}(\mathcal{O}))$ is not faithful. 
\end{remark}

 \begin{theorem}\label{pdpairing} 
Let $\mathcal{O}$ be a compact oriented 2-orbifold of negative Euler characteristic and $\Gamma=\pi_1 ^{\operatorname{orb}}(\mathcal{O})$. Then the pairing $\omega_{PD}$ is a symplectic form on   $\Rep_n ^{\mathscr{B}}(\Gamma)$ via the right translation. 
 \end{theorem}
 \begin{proof}
By works of \cite{guruprasad1997} and \cite{kim1999},  each leaf $\Rep_n ^{\mathscr{B}'}(\pi_1(X'_{\mathcal{O}}))$ of $\Rep_n(\pi_1(X'_{\mathcal{O}}))$ carries the symplectic form $\omega_K ^{X'_{\mathcal{O}}}$.  By comparing the explicit formula for $\omega_{PD}$ (Theorem \ref{explicit}) and that of $\omega_K ^{X'_{\mathcal{O}}}$  \cite[Theorem 5.6]{kim1999}, we know that $\mathcal{I}^* \omega_K ^{X'_{\mathcal{O}}}= \omega_{PD}$. This proves that $\omega_{PD}$ is closed. Moreover we know that $\der\mathcal{I}(H^1_{\operatorname{par}}(\Gamma, \per{\Gamma};\mathfrak{g})) = H^1_{\operatorname{par}}(\pi_1(X_{\mathcal{O}}'), \per{\pi_1(X_{\mathcal{O}}')};\mathfrak{g})$. Since $\omega_K ^{X'_{\mathcal{O}}}$ is nondegenerate, so is $\omega_{PD}$. 
 \end{proof}
 
 In particular we have the following result for $n=3$ case.

 \begin{proposition}\label{symplecticembedding}
Let $\mathcal{O}$ be a compact oriented 2-orbifold of negative Euler characteristic. The image of the embedding $\mathcal{I}: \mathcal{C}^{\mathscr{B}}(\mathcal{O}) \to \Rep_3 ^{\mathscr{B}'} (\pi_1(X'_{\mathcal{O}}))$ is a symplectic submanifold. 
\end{proposition}

 \begin{definition}
We will call $\omega_{PD}$ on  $\mathcal{X}_n ^{\mathscr{B}}(\pi_1 ^{\text{orb}}(\mathcal{O}))$ or on $\mathcal{C}^{\mathscr{B}}(\mathcal{O})$ the Atiyah-Bott-Goldman symplectic form and denote it simply by $\omega^{\mathcal{O}}$. 
\end{definition}

\subsection{Construction by equivariant de Rham bicomplex}\label{deRham}
In this section, we give another construction of the Atiyah-Bott-Goldman symplectic form. The approach of this section, originally carried out in \cite{huebschmann1995,guruprasad1997,weinstein1995}, produced the \emph{equivariantly} closed 2-form on a neighborhood of $\Rep_n^{\mathscr{B}}(S)$ which restricts to the Atiyah-Bott-Goldman symplectic form on $\Rep_n^{\mathscr{B}}(S)$ for a compact surface of negative Euler characteristic. We inherit their method to yield the same result for compact orientable orbifolds.

In fact, this is much more than what we actually want. We only need to construct the Atiyah-Bott-Goldman type symplectic form on $\mathcal{C}^{\mathscr{B}}(\mathcal{O})$, which is already done in Theorem \ref{pdpairing}. Nonetheless, we include this construction as it is interesting in itself.

Let $G$ be a Lie group acting on a smooth manifold $M$. Recall that the equivariant de Rham complex $\mathcal{A}^{p,q}_G (M)$ is defined as follows: $\mathcal{A}^{2j,p}_G (M)$ is the set of $G$-equivariant homogeneous polynomials of degree $j$ on the Lie algebra $\mathfrak{g}$ of $G$ with values in  $p$-forms on $M$. We have two differentials $\der$ and $\delta_G$ defined by
\begin{align*}
&\der : \mathcal{A}^{2j, p } _G (M) \to \mathcal{A} ^{2j, p+1} _G (M)\, \quad \text{ the usual de Rham differential}\\
&\delta_G : \mathcal{A} ^{2j, p} _G (M) \to \mathcal{A}^{2j+2, p-1} _G (M) \, \quad \delta_G ( \alpha) (X) = -\iota_{\overline{X}} (\alpha(X))
\end{align*} 
where $\overline{X}$ is the fundamental vector field generated by $X$. 

An element $\alpha\in \mathcal{A}^{2j,p}_G(M)$ is said to be equivariantly closed if $(\der + \delta_G) (\alpha)=0$. If $M$ is a symplectic manifold with a symplectic form $\omega$ and the $G$-action is symplectomorphic, then $\omega$ becomes a $\der$-closed element of $\mathcal{A}^{0,2}_G (M)$. In general $\omega$ may not be equivariantly closed. An element (if exists) $\mu\in \mathcal{A}^{2,0}_G(M)$ making $\omega+\mu$  equivariantly closed is called a \emph{moment map}. 

We consider $M=G^q$ where $G$ acts on $M$ as conjugations on each component. There is another differential $\delta: \mathcal{A}^{2j, p} _G (G^q) \to \mathcal{A}^{2j, p}_G (G^{q+1})$ defined by 
\[
\delta f = (-1)^{q+1} \sum_{i=0} ^{q} (-1)^i (\delta_i f)
\]
where $\delta _i f$ is induced by $(x_0, \cdots, x_q) \mapsto (x_1, \cdots, x_q)$ for $i=0$, $(x_0, \cdots, x_q) \mapsto (x_0, \cdots, x_{q-1})$ for $i=q$ and $(x_0, \cdots, x_q) \mapsto (x_0, \cdots, x_{i-1} x_i, \cdots, x_q)$ for $1\le i \le q$.  

Let $\omega$ be the Cartan-Maurer form on $G$. That is, $\omega_g(X)$ is the unique left invariant vector field such that $\omega_ g (X)=X_g$. Let $\iota_i:G\times G \to G$ be the projection onto $i$th factor, $i=1,2$. Define
\[
\Omega := \frac{1}{2} \iota_1 ^* \omega \cdot \overline{\iota_2 ^* \omega}
\]
where $\cdot$ is the trace form. $\Omega$ is an element of $\mathcal{A}^{0,2}_G(G^2)$. 
\begin{remark}
Our convention of operations of Lie algebra valued forms is
\begin{align*}
(\omega \cdot \omega)(X,Y) &= \omega (X) \cdot \omega(Y) - \omega(Y) \cdot \omega(X)\\
[\omega, \omega] (X,Y) &= [\omega(X),\omega(Y)] - [\omega(Y), \omega(X)].
\end{align*}
This coincides with the one used in Huebschmann \cite{huebschmann1995}, \cite{huebschmann1995b} and Weinstein \cite{weinstein1995}. Under this convention the structure equation takes of the form
\[
\der \omega = -\frac{1}{2} [\omega, \omega].
\]
\end{remark}

We also define
\begin{align*}
\lambda &:= \frac{1}{12}[\omega,\omega]\cdot \omega \in \mathcal{A}^{0,3}_G (G)\\
\theta&:=\frac{1}{2} X \cdot (\omega+ \overline{\omega})\in \mathcal{A}^{2,1}_G (G).
\end{align*}
\begin{lemma}[Weinstein \cite{weinstein1995}] We have the following identities
\begin{align*}
\der \Omega &= \delta \lambda\\
\delta \Omega &= 0\\
\delta_G \Omega &= - \delta \theta\\
\der \lambda &= 0\\
\delta_G \lambda &= \der \theta \\ 
\delta_G \theta &= 0.
\end{align*}
\end{lemma}

Recall that the group $F^{\#}$ was defined by the presentation 
\[
F^{\#}:=\langle x_1, y_1, \cdots, x_g, y_g, s_1, \cdots, s_{c},z_1, \cdots, z_b \,|\, \mathbf{r}_1= \cdots = \mathbf{r}_{c} = 1\rangle.
\]
Let
\[
E: \Hom (F^{\#}, G) \times (F^{\#})^q \to G^q
\]
be the evaluation map $E(\rho, x_1,x_2, \cdots, x_q) = (\rho(x_1),\rho(x_2), \cdots \rho(x_q))$. Here  $G$ acts trivially on $(F^{\#})^q$.  Then  $E$ is $G$-equivariant so that it induces
\[
E^*: \mathcal{A}^{2j, p} _G (G^q) \to \mathcal{A}^{2j, p}_G (\Hom(F^{\#},G)\times (F^{\#})^q). 
\]
Observe that 
\[
 \mathcal{A}^{2j, p}_G (\Hom(F^{\#},G)\times (F^{\#})^q)=\mathcal{A}^{2j,p} _G (\Hom(F^{\#},G))\otimes\mathcal{A}^{2j,0}_G((F^{\#})^q).
 \]
Following Huebschmann \cite{huebschmann1995}, define 
\[
 \omega_{\mathfrak{c}} = \langle E^* \Omega, -\mathfrak{c} \rangle \in \mathcal{A}^{0,2}_G (\Hom(F^{\#},G))
 \]
 where $\mathfrak{c}$ is the one define in Lemma \ref{fundamentalclass}. Here we regard $\mathfrak{c}$ as an element of $ \mathbf{B}_2(F^{\#})\otimes \R\subset \R[F^{\#}\times F^{\#}]$. 
Notice that
\begin{align*}
 \der \omega_{\mathfrak{c}} &=  \langle  E^* \der \Omega, -\mathfrak{c} \rangle \\
 &= \langle E^* \delta \lambda, -\mathfrak{c}\rangle  \\
 &= \langle E^* \lambda, -\partial \mathfrak{c} \rangle\\
 &= \langle E^* \lambda,  \mathbf{r} - \sum z_i  \rangle\\
 &= \mathbf{r}^* \lambda  -\sum z_i ^* \lambda.
\end{align*}
Here we view $\mathbf{r}$ and $z_i$ as the maps from $\Hom(F^{\#},G)$ to $G$ defined by $\rho \mapsto E(\rho, \mathbf{r})$ and $\rho \mapsto E(\rho, z_i)$ respectively. Now we choose any regular neighborhood $O\subset \mathfrak{g}$ for the exponential map at the identity and consider the pull-back diagram
\[
\xymatrix{\mathcal{H}_{\#} ^{\mathscr{B}}\ar[d]^{\eta} \ar[rrr]^--{(\widehat{\mathbf{r}},\widehat{z_1} ,\cdots, \widehat{z_b})} && &O \times C_1\times \cdots \times C_b \ar[d]^{\exp\times \operatorname{Id}} \\
\Hom^{\mathscr{B}}(F^{\#}, G) \ar[rrr]_--{(\mathbf{r},z_1, \cdots, z_b)} &&&G \times C_1\times \cdots \times C_b}.
\]
Since $O\subset \mathfrak{g}$ is contractible, we can construct an adjoint invariant homotopy operator $h: \mathcal{A}^{2j, p}_G (\mathfrak{g} ) \to \mathcal{A}^{2j, p-1}_G (\mathfrak{g})$ such that $\der h + h \der = \operatorname{Id}$.  Using this homotopy operator, we define
\[
\omega= \eta^* \omega_{\mathfrak{c}} - \widehat{\mathbf{r}} ^* h \exp^* \lambda. 
\]
 Then we have
 \begin{equation}\label{dcomputed}
 \der \omega = -\sum \widehat{z_i} ^* \lambda. 
 \end{equation}

Let $\mu: \mathcal{H}_{\#} ^{\mathscr{B}} \to \mathfrak{g}^*$ be given by $\mu = \psi \circ \widehat{\mathbf{r}}$ where $\psi:\mathfrak{g} \to \mathfrak{g}^*$ is the dualization with respect to the trace form $\Tr$.  In view of the proof of Theorem 2 of \cite{huebschmann1995}, we know
\begin{equation}\label{deltacomputed}
 \delta_G \omega =- \der \mu +\sum \widehat {z_i} ^* \theta. 
\end{equation}

Let $C$ be a conjugacy class in $G$. To find a closed 2-form on $\mathcal{H}_{\#} ^{\mathscr{B}}$, we need $\tau \in \mathcal{A}^{0,2}_G (C)$ such that $\der \tau = \lambda$ and $\delta_G \tau =- \theta$. At each $p\in C$, the right translation by $p$ identifies the tangent space $T_p C$ with the subspace $\{ \Ad _p X - X\,|\,X\in \mathfrak{g}\}$ of $\mathfrak{g}= T_e G$. The 2-form $\tau$ on $C$ is then defined by
\[
\tau (\Ad_p X- X , \Ad_p Y - Y) = \frac{1}{2} ( \Tr (X \Ad _p Y) - \Tr (Y \Ad _p X)).
\] 
\begin{lemma}[Guruprasad-Huebschmann-Jeffrey-Weinstein \cite{guruprasad1997}]\label{ghjw2}  We have
\[
\der \tau = \lambda ,\quad \delta_G \tau = -\theta.
\]
\end{lemma}

Now define the 2-form $\omega_H$ on $\mathcal{H}_{\#} ^{\mathscr{B}}$ by 
\[
\omega_H = \eta ^* \langle  E^* \Omega, -\mathfrak{c} \rangle - \widehat{\mathbf{r}} ^* h \exp^* \lambda  + \sum_{i=1} ^b  \widehat{z_i} ^*  \tau_i. 
\]

\begin{proposition} $\omega_H$ is closed  on $\mathcal{H}_{\#} ^{\mathscr{B}}$ and a moment map is given by
\[
\mu= \psi \circ \widehat{\mathbf{r}}.
\]
\end{proposition}
\begin{proof}
We first observe that $\omega_H$ is $\der$-closed. Indeed, from (\ref{dcomputed}) and  Lemma \ref{ghjw2}, we have
\[
\der\omega_H = -\sum_{i=1} ^b  \widehat{z}_i ^* \lambda +\sum_{i=1} ^b  \widehat{z}_i ^* \der \tau_i = 0.
\]

For the second assertion regarding the moment map, we use (\ref{deltacomputed}) and Lemma \ref{ghjw2} to compute
\[
\delta_G \omega_H = -\der \mu +\sum_{i=1} ^b \widehat{z}^* _i \theta +\sum_{i=1} ^b \widehat{z}^*_i \delta_G \tau_i = -\der \mu.
\]
Since $\mu\in \mathcal{A}_G ^{2,0}(\mathcal{H}^{\mathscr{B}}_{\#})$, $\delta_G \mu = 0$ is automatic. It follows that $\omega_H+\mu$ is equivariantly closed meaning that $\mu=\psi\circ \widehat{\mathbf{r}}$ is a moment map for the $G$-action on $\mathcal{H}^{\mathscr{B}}_{\#}$. 
\end{proof}

Let $\rho_0\in \mu^{-1}(0)$. We know that $\eta(\rho_0)\in \Hom (F^{\#}, G)$ induces the representation $\rho:\Gamma \to G$. In view of Lemma \ref{tangent} the tangent space at $\rho_0 \in \mu^{-1}(0)$ can be identified with $Z^1 _{\text{par}}(\Gamma,  \per{\Gamma}; \mathfrak{g}_{\rho})$. The fundamental vector fields $\overline{X}$, $X\in \mathfrak{g}$ span the subspace $B^1(\Gamma, \mathfrak{g})$. 

\begin{theorem}\label{comp}
Let $\rho_0 \in \mu^{-1} (0)$. Assume that the induced representation $\rho:\Gamma \to G$ is in $\Hom_{\operatorname{s}} ^{\mathscr{B}}(\Gamma, G)$.  Then $\omega_{H,\rho_0}|_{Z^1_{\operatorname{par}}(\Gamma, \per{\Gamma}; \mathfrak{g}_\rho)}$ descends to $H^1 _{\operatorname{par}}(\Gamma, \per{\Gamma}; \mathfrak{g}_\rho)$ and defines the closed 2-form, denoted by the same symbol  $\omega_H$, on $\mathcal{X}^ {\mathscr{B}} _n (\Gamma)$.   Moreover the induced closed 2-form $\omega_H$ on  $\mathcal{X}^ {\mathscr{B}} _n (\Gamma)$ coincides with the right translation of $-\omega_{PD}$. 
\end{theorem}
\begin{proof}
By the construction, we have $\delta_G \omega_H=0$ on $\mu^{-1}(0)$. This means that $\omega_H$ vanishes along $B^1(\Gamma; \mathfrak{g})$. Therefore $\omega_H$ descends to $H^1_{\text{par}}(\Gamma,  \per{\Gamma}; \mathfrak{g})$. By $\der$-closedness of $\omega_H$, we know that the induced 2-form is also closed on $\mathcal{X}^ {\mathscr{B}} _n (\Gamma)$. Since the $G$-action on $\Hom_{\operatorname{s}} ^\mathscr{B} (\Gamma, G)\subset\mu^{-1}(0)$ is proper, we can understand this process as the Marsden-Weinstein quotient.

Recall Theorem \ref{explicit}. Since $\omega_{PD}$ is anti-symmetric, we have that
\begin{align}
\omega_{PD}([u],[v]) & = \frac{1}{2} (\langle u\cup v, \mathfrak{c}\rangle -\langle v\cup u, \mathfrak{c}\rangle) - \frac{1}{2} \sum_{i=1} ^b \Tr (X_i v(z_i) - Y_i u(z_i))\nonumber\\
&=\frac{1}{2} (\langle u\cup v, \mathfrak{c}\rangle -\langle v\cup u, \mathfrak{c}\rangle) - \frac{1}{2} \sum_{i=1} ^b \Tr (X_i z_i \cdot Y_i  - Y_i z_i \cdot X_i). \label{comp2}
\end{align}
The last term coincides with $\sum_{i=1} ^b \tau(z_i\cdot X_i-X_i, z_i\cdot Y_i -Y_i)=\sum_{i=1} ^b z_i ^* \tau_i$. We claim that the first term $\frac{1}{2}(\langle u\cup v , \mathfrak{c}\rangle - \langle v\cup u , \mathfrak{c}\rangle)$ equals $\langle E^*\Omega,-\mathfrak{c}\rangle(u,v)$. For this,  it suffices to show that  
\[
\langle  E^*\Omega ,  \llbracket x|y \rrbracket\rangle (u,v)=\frac{1}{2}\left(\langle u\cup v,\llbracket x|y\rrbracket \rangle - \langle v\cup u , \llbracket x|y \rrbracket \rangle\right)
\]
for $\llbracket x|y\rrbracket \in \mathbf{B}_2(F^{\#})\otimes \R$.  

Recall that (Lemma \ref{tangent1}) the right translation identifies  $Z^1(\Gamma;\mathfrak{g})$ with $T_\rho \Hom(\Gamma,G)$. Thus given $u\in Z^1(\Gamma;\mathfrak{g})$ and $x\in \Gamma$, we have that $E_{x*} (u) = R_{\rho(x)*} (u(x))$ where  the map $E_x:\Hom(\Gamma,G) \to G$ is defined by $\rho \mapsto \rho(x)$. Then we compute
\begin{align*}
\langle  E^*\Omega, \llbracket x|y \rrbracket \rangle (u,v)&= \frac{1}{2}( \iota_1 \omega \cdot \overline{\iota_2 ^*\omega})_{(\rho(x),\rho(y))} (E_* u, E_*v)\\
&=\frac{1}{2}\left( \omega_{\rho(x)}(E_{x*}u) \cdot \overline{ \omega_{\rho (y)} (E_{y*}v) } - \omega_{\rho(x)} (E_{x*}v)\cdot \overline{\omega_{\rho(y)} (E_{y*}u)}\right)\\
&=\frac{1}{2}\left(\Tr(\Ad_{\rho(x)^{-1}}(u(x))  v(y))-\Tr(\Ad_{\rho(x)^{-1}}(v(x)) u(y))\right)\\
&=\frac{1}{2}\left(\Tr(u(x) x\cdot v(y))-\Tr(v(x)x\cdot u(y))\right).
\end{align*}
By comparing this to the Alexander-Whitney approximation theorem, Lemma \ref{diagonalapprox}, we conclude that $\langle  E^*\Omega ,  \llbracket x|y \rrbracket\rangle (u,v)=\frac{1}{2}\left(\langle u\cup v,\llbracket x|y\rrbracket \rangle - \langle v\cup u , \llbracket x|y \rrbracket \rangle\right)$ as we wanted. 

Finally observe that, on the space $\Hom^{\mathscr{B}}(\Gamma,G)$, the second term $ \widehat{\mathbf{r}} ^* h \exp^* \lambda $ of $\omega_H$ vanishes. Therefore  $-\omega_H$ equals (\ref{comp2}) on $\mathcal{X}^{\mathscr{B}} _n (\Gamma)$. 
\end{proof}

\section{Global Darboux coordinates}\label{main}
In this section, we prove our main theorem. Recall that we have the splitting operation that decomposes a given convex projective orbifold into simpler pieces of convex projective orbifolds.  We show the decomposition theorem for $\mathcal{C}^{\mathscr{B}}(\mathcal{O})$ which is analogous to a previous result of the authors \cite[Theorem 4.5.7]{choi2019}.  Finally by using our version of the action-angle principle \cite[Theorem 3.4.5]{choi2019}, we obtain the global Darboux coordinates for the symplectic manifold $(\mathcal{C}(\mathcal{O}), \omega^{\mathcal{O}})$. 

Throughout this section, we continue to use $\per{\Gamma}$ to denote a set of subgroups generated by primitive peripheral elements. See Notation \ref{notation} for details. 
\subsection{Local decomposition}\label{localsecompsec}

We will prove that the tangent space can be decomposed into a direct sum of symplectic subspaces under the splitting operations. Indeed, we did show the same result for compact surfaces. What is new here is that  splittings can happen along full 1-suborbifolds. 

Let $\mathcal{O}$ be a compact oriented 2-orbifold of negative Euler characteristic, $\Gamma:=\pi_1 ^{\operatorname{orb}}(\mathcal{O})$. Let $\xi$ be a principal full 1-suborbifold joining two order 2 cone points of $\mathcal{O}$.  Note that $\mathcal{O}'=\mathcal{O}\setminus \xi$ has also negative Euler characteristic. The orbifold fundamental group $\Gamma':= \pi_1^{\operatorname{orb}}(\mathcal{O}')$ of $\mathcal{O}'$ is a subgroup of $\Gamma$ and we denote by  $\iota:\Gamma' \to \Gamma$ the inclusion.  We can find presentations 
\begin{multline*}
\Gamma= \langle x_1,y_1,\cdots, x_g,y_g, z_1, \cdots, z_b, s_1, \cdots, s_c\,|\,\\ 
\prod_{i=1} ^g [x_i,y_i] \prod_{i=1} ^b z_i \prod_{i=1} ^c s_i = s_1 ^{r_1}=\cdots=s_c ^{r_c}=1\rangle
\end{multline*}
and 
\begin{multline*}
\Gamma'=\langle x_1,y_1,\cdots, x_g,y_g, z_1, \cdots, z_{b+1}, s_3, \cdots, s_c\,|\,\\
\prod_{i=1} ^g [x_i,y_i] \prod_{i=1} ^{b+1} z_i \prod_{i=3} ^c s_i = s_3^{r_3}=\cdots=s_c ^{r_c}=1\rangle
\end{multline*}
such that $r_1=r_2=2$ and $\iota (z_{b+1}) = s_1s_2$. Recall that $z_{1},\cdots, z_{b+1}$ correspond to the boundary components of $\mathcal{O}'$. 

Let  $\mathcal{S}= \per{\Gamma} \cup \{\langle s_1s_2\rangle\}$. We consider the group pair $(\Gamma, \mathcal{S})$. In view of  Lemma \ref{funtorialinj}, $H^1_{\operatorname{par}}(\Gamma, \mathcal{S};\mathfrak{g})$ is a subspace of $H^1_{\operatorname{par}}(\Gamma, \per{\Gamma};\mathfrak{g})$. Hence one can compute $\omega^{\mathcal{O}}([u],[v])$  by regarding $[u]$ and  $[v]$ as members of $H^1_{\operatorname{par}}(\Gamma, \per{\Gamma};\mathfrak{g})$.

\begin{proposition}\label{sum}
Keep the notations as above. Let $[\rho]\in \Rep_n ^{\mathscr{B}}(\Gamma)$ be such that $[\rho\circ\iota]\in \Rep_n ^{\mathscr{B}'}(\Gamma')$ for some $\mathscr{B}'$.  Then for any $[u],[v]\in H^1_{\operatorname{par}}(\Gamma, \mathcal{S};\mathfrak{g}_\rho)$ we have
\[
\omega^{\mathcal{O}}([u],[v]) = \omega^{\mathcal{O}' }(\iota^* [u],\iota^* [v]).
\]
where $\iota^*: H^1 _{\operatorname{par}} (\Gamma, \mathcal{S}; \mathfrak{g}) \to H^1_{\operatorname{par}} (\Gamma',\per{\Gamma'}; \mathfrak{g})$ is the map induced from $\iota: \Gamma'\to \Gamma$. 
\end{proposition}
\begin{proof}
Throughout the proof, we will use the explicit formula for $\omega$ given in Theorem \ref{explicit} and conventions therein. In particular, the parabolic cocycles $u,v$ are regarded as members of $\Hom_\Gamma (\mathbf{B}_1,\mathfrak{g})$ with the property that, for each $i=1,2,\cdots, b$, $u(\llbracket z_i \rrbracket) = z_i \cdot X_i -X_i$ for some $X_i\in \mathfrak{g}$.

Recall that the map $\iota:\Gamma' \to \Gamma$ induces the equivariant chain map $\iota_*:\mathbf{B}_*(\Gamma') \to \mathbf{B}_*(\Gamma)$ defined by $\iota_*(\llbracket x | y \rrbracket) = \llbracket \iota(x) | \iota(y) \rrbracket$. Let $\mathbf{r}' := \prod_{i=1} ^g [x_i,y_i] \prod_{i=1} ^ {b+1} z_i \prod_{i=3}^c s_i$,  $\mathbf{r} = \prod_{i=1} ^g [x_i,y_i] \prod_{i=1} ^{b} z_i$, and $\mathbf{r}_i = s_i ^{r_i}$.  For the chain $\mathfrak{c}'\in \mathbf{B}_2(\Gamma')\otimes \R$ given as in (\ref{prefundamentalcycle}) of Lemma \ref{fundamentalclass} with respect to $\Gamma'$, we compute
\begin{align*}
\iota_* \mathfrak{c}' &= \sum_{i=1}^g \left( \left.\left\llbracket \iota\frac{\partial \mathbf{r}'}{\partial x_i} \right|\iota x_i \right\rrbracket+ \left.\left\llbracket \iota\frac{\partial \mathbf{r}'}{\partial y_i} \right| \iota y_i \right\rrbracket\right) +\sum_{i=1}^{b+1}  \left.\left\llbracket \iota\frac{\partial \mathbf{r}'}{\partial z_i} \right|\iota z_i \right\rrbracket \\ 
&\qquad\qquad +\sum_{i=3} ^ c\left.\left\llbracket \iota\frac{\partial \mathbf{r}'}{\partial s_i}\right|\iota s_i \right\rrbracket - \sum_{i=3} ^{c} \frac{1}{r_i} \left.\left\llbracket \iota\frac{\partial \mathbf{r}_i'}{\partial s_i}\right| \iota s_i \right\rrbracket\\
&= \sum_{i=1}^g \left( \left.\left\llbracket \frac{\partial \mathbf{r}}{\partial x_i} \right| x_i \right\rrbracket+ \left.\left\llbracket \frac{\partial \mathbf{r}}{\partial y_i} \right|  y_i \right\rrbracket\right)+ \sum_{i=1}^{b}  \left.\left\llbracket \frac{\partial \mathbf{r}}{\partial z_i} \right| z_i \right\rrbracket +  \left.\left\llbracket \mathbf{r} \right| s_1s_2 \right\rrbracket\\ 
&\qquad\qquad +\sum_{i=3} ^ c\left.\left\llbracket \iota \frac{\partial \mathbf{r}'}{\partial s_i}\right| s_i \right\rrbracket - \sum_{i=3} ^{c} \frac{1}{r_i} \left.\left\llbracket \frac{\partial \mathbf{r}_i}{\partial s_i}\right|  s_i \right\rrbracket\\
&=\mathfrak{c}  +  \left.\left\llbracket \mathbf{r} \right| s_1s_2 \right\rrbracket - \left.\left\llbracket \mathbf{r}\right| s_1 \right\rrbracket-\left.\left\llbracket  \mathbf{r}s_1\right| s_2 \right\rrbracket+ \frac{1}{r_1} \left.\left\llbracket \frac{\partial \mathbf{r}_1}{\partial s_1}\right|  s_1 \right\rrbracket+\frac{1}{r_2} \left.\left\llbracket \frac{\partial \mathbf{r}_2}{\partial s_2}\right|  s_2 \right\rrbracket.
\end{align*}
Since $r_1=r_2=2$, we know that the last two terms become
\[
 \frac{1}{r_1} \left.\left\llbracket \frac{\partial \mathbf{r}_1}{\partial s_1}\right|  s_1 \right\rrbracket+\frac{1}{r_2} \left.\left\llbracket \frac{\partial \mathbf{r}_2}{\partial s_2}\right|  s_2 \right\rrbracket= \frac{1}{2} \llbracket s_1|s_2 \rrbracket + \frac{1}{2}\llbracket s_2|s_2\rrbracket.
\]
This shows that the chain $\mathfrak{c}\in \mathbf{B}_2(\Gamma) \otimes \R$ given in (\ref{prefundamentalcycle}) with respect to $\Gamma$ can be  written as
\[
\mathfrak{c}=\iota_*\mathfrak{c}'-\llbracket \mathbf{r}|s_1s_2\rrbracket+\llbracket \mathbf{r}|s_1\rrbracket+\llbracket \mathbf{r}s_1|s_2\rrbracket-\frac{1}{2}\llbracket s_1|s_1\rrbracket - \frac{1}{2}\llbracket s_2|s_2\rrbracket.
\]

Since $[u],[v]\in H^1_{\operatorname{par}}(\Gamma,\mathcal{S};\mathfrak{g})$, one can choose representatives $u,v$ of $[u], [v]$ so that $u(s_1s_2)=v(s_1s_2)=0$. Then by the cocycle condition $v(s_1s_2) = v(s_1) + s_1 \cdot v(s_2)$, we have
\[
\Tr(u(\mathbf{r}) \mathbf{r}\cdot v(s_1s_2)) = \Tr (u(\mathbf{r})\mathbf{r}\cdot v(s_1))+\Tr(u(\mathbf{r}) (\mathbf{r}s_1)\cdot v(s_2))
\]
and
\[
\Tr(u(\mathbf{r}s_1) (\mathbf{r}s_1)\cdot v(s_2)) = \Tr((u(\mathbf{r})+\mathbf{r}\cdot u(s_1)) (\mathbf{r}s_1) \cdot v(s_2)). 
\]
Therefore we get
\begin{multline*}
\langle u\cup v ,\mathfrak{c}\rangle = \langle u \cup v,\iota_* \mathfrak{c}'\rangle\\+\Tr(u(s_1) s_1 \cdot v(s_2))-\frac{1}{2}\Tr(u(s_1)s_1\cdot v(s_1))-\frac{1}{2}\Tr(u(s_2)s_2\cdot v(s_2)).
\end{multline*}
Since $u(s_1s_2)=0$, we know that $0=u(s_1) + s_1 \cdot u(s_2)$. Because $s_1 ^2=1$, we can deduce $u(s_2) = -s_1\cdot u(s_1) = u(s_1)$. Similarly, we get $v(s_1) = v(s_2)$. Hence it follows that
\begin{align*}
\Tr(u(s_1) s_1 \cdot  v(& s_2)) -\frac{1}{2}\Tr(u(s_1)s_1\cdot v(s_1))-\frac{1}{2}\Tr(u(s_2)s_2\cdot v(s_2))\\
&=-\Tr(u(s_1)  v(s_2))+\frac{1}{2}\Tr(u(s_1) v(s_1))+\frac{1}{2}\Tr(u(s_2) v(s_2))\\
&=-\Tr(u(s_1)  v(s_1))+\frac{1}{2}\Tr(u(s_1) v(s_1))+\frac{1}{2}\Tr(u(s_1) v(s_1))\\
&=0.
\end{align*}
Finally, our choice of cocycles $u,v$ yields that $X_{b+1}=0$. Combining all, we get
\begin{align*}
\omega^{\mathcal{O}} ([u],[v]) &= \langle u \cup v , \iota_*\mathfrak{c}'\rangle +\Tr (X_{b+1} v(z_{b+1})) -\sum_{i=1} ^{b+1} \Tr (X_i v(z_i))\\
&=\omega^{\mathcal{O}'}(\iota^*[u], \iota^*[v]).
\end{align*}
as desired.
\end{proof}

We unify the local decomposition theorem in the full generality. Let $\mathcal{O}$ be a compact oriented 2-orbifold of negative Euler characteristic. Suppose that we are given   a collection $\{\xi_1, \cdots, \xi_m\}$ of pairwise disjoint essential simple closed curves or full 1-suborbifolds such that the completions of the connected components $\mathcal{O}_1, \cdots, \mathcal{O}_l$ of $\mathcal{O}\setminus \bigcup_{i=1} ^m  \xi_i$ have negative Euler characteristic. As before, we denote $\Gamma:= \pi_1 ^{\operatorname{orb}}(\mathcal{O})$ and $\Gamma_i := \pi_1 ^{\operatorname{orb}}(\mathcal{O}_i)$.

Let $\mathcal{S}=\per{\Gamma}\cup\{\langle \xi_1\rangle, \cdots, \langle \xi_m\rangle\}$. Here we should clarify what we mean by  $\langle \xi_i \rangle$ when $\xi_i$ is a full 1-suborbifold. Suppose that the lift $\widetilde{\xi_i}$ of $\xi_i$ in the universal cover is joining the fixed points of the corresponding generators in $\pi_1 ^{\operatorname{orb}}(\mathcal{O})$, say, $s_1$ and $s_2$. Then $\langle \xi_i\rangle$ is, by convention, $\langle s_1 s_2\rangle$.

For each $i$, there is the inclusion $\iota_i: \Gamma_i \to \Gamma$ which induces $\iota^*_i :H^1_{\operatorname{par}}(\Gamma, \mathcal{S} ; \mathfrak{g}) \to H^1_{\operatorname{par}}(\Gamma_i, \per{\Gamma_i} ;\mathfrak{g})$.

\begin{theorem}\label{localdecompmain}
Let $\mathcal{O}$ be a compact oriented 2-orbifold of negative Euler characteristic and $\Gamma=\pi_1^{\operatorname{orb}}(\mathcal{O})$. Let $\{\xi_1, \cdots, \xi_m\}$ be a set of pairwise disjoint essential simple closed curves or full 1-suborbifolds such that the completions of the connected components $\mathcal{O}_1, \cdots, \mathcal{O}_l$ of $\mathcal{O}\setminus \bigcup_{i=1} ^m  \xi_i$ have negative Euler characteristic. Let $[\rho]\in \Rep_n ^{\mathscr{B}}(\Gamma)$ be such that, for each $i$, $[\rho\circ \iota_i ]\in \Rep_n ^{\mathscr{B}_i}(\Gamma_i)$ for some choice of conjugacy classes $\mathscr{B}_i$.   Then  for any $[u],[v]\in H^1_{\operatorname{par}}(\Gamma, \mathcal{S};\mathfrak{g})$ we have
\[
\omega^{\mathcal{O}}([u],[v]) = \sum_{i=1} ^{l} \omega^{\mathcal{O}_i}(\iota_i^* [u], \iota_i ^* [v]). 
\]
\end{theorem}
\begin{remark}
Again, $[u]$ and  $[v]$ should be understood (by Lemma \ref{funtorialinj}) as members of $H^1_{\operatorname{par}}(\Gamma,\per{\Gamma};\mathfrak{g})$  so that the left hand side $\omega^{\mathcal{O}}([u],[v]) $ makes sense. 
\end{remark}
\begin{proof}
Renumbering if necessary, we may assume that $\xi_1, \cdots, \xi_k$ represent the essential simple closed curves and remaining $\xi_{k+1}, \cdots, \xi_m$ are the full 1-suborbifolds. Then, we adopt the proof of  Corollary 4.4.3 of \cite{choi2019}. Although Corollary 4.4.3 of \cite{choi2019} is stated for compact surfaces, the same assertion and the proof work for compact orbifolds. Then, we use Proposition \ref{sum} inductively for each full 1-suborbifold $\xi_i$, $i\ge k+1$. 
\end{proof}

\subsection{Hamiltonian action and global decomposition}
Now we deal with the global decomposition theorem. Throughout this subsection we focus on $n=3$ case only. Unless otherwise stated,  $G$ denotes $\PSL_3(\R)$.

We briefly summarize Goldman's work on Hamiltonian flows. By an \emph{invariant function}, we mean a smooth function $f: U \to \R$ defined on an nonempty conjugation invariant subset $U$ of $G$ such that $f(ghg^{-1})= f(h)$ for all $h\in U$ and $g\in G$.  Given an invariant function $f$ we can associate so call the \emph{Goldman function} $f^{\#}:U \to \mathfrak{g}$ whose defining property is that $\frac{\der}{\der t}|_{t=0} f(h \exp tX)=\Tr (f^{\#}(h)X)$ for all $X\in \mathfrak{g}$.

We fix the invariant functions $\ell$ and $\mathbf{m}$ defined on the set $\mathbf{Hyp}^+$ of hyperbolic elements of $G$. 
\[
\mathbf{\ell} (g) := \log \left|\frac{\lambda_1(g)}{\lambda_3(g)}\right|,\quad \mathbf{m}(g):=\log |\lambda_2(g)|
\]
where $\lambda_i(g)$ is the $i$th largest eigenvalue of $g\in \mathbf{Hyp}^+$. We can observe that $(\ell,\mathbf{m})$ is the set of complete invariants for elements in $\mathbf{Hyp}^+$. Namely, two elements $g,h\in \mathbf{Hyp}^+$ are conjugate if and only if $(\ell(g), \mathbf{m}(g))=(\ell(h), \mathbf{m}(h))$.

Let $\xi_1, \cdots, \xi_{m_0}$ be essential simple closed curves and let $\xi_{m_0+1}, \cdots, \xi_{m_0+m_e}$ be full 1-suborbifolds. Set $m=m_0+m_e$ and $M=2m_0+m_e$. Suppose that $\xi_1, \cdots, \xi_m$ are mutually disjoint, non-isotopic and each component of $\mathcal{O}\setminus \bigcup_{i=1} ^m \xi_i$ has negative Euler characteristic. 

We apply the argument of Choi-Jung-Kim \cite{choi2019} with respect to the closed curves $\xi_1, \cdots, \xi_{m_0}$. This process identifies $\pi_1 ^{\operatorname{orb}}(\mathcal{O})$ with the fundamental group of a graph of groups $(\Gamma, \mathcal{G})$ where $\Gamma=\pi_1 ^{\operatorname{orb}}(O)$.  The graph $\mathcal{G}$ has the vertex set $V(\mathcal{G}):=\{\mathcal{O}_1, \cdots, \mathcal{O}_l\}$ and we join $\mathcal{O}_i$ and $\mathcal{O}_j$, possibly $i=j$, by an edge if and only if they are glued along some $\xi_k$. We choose a base vertex say $\mathcal{O}_1$ and a maximal spanning tree $\mathcal{T}$ of $\mathcal{G}$. Then $V(\mathcal{G})$ admits the natural partial ordering $\le$ with  the minimal element $\mathcal{O}_1$. The vertex group $\Gamma_{\mathcal{O}_i}$ of $\mathcal{O}_i$ is isomorphic to $\pi_1 ^{\operatorname{orb}}(\mathcal{O}_i)$ and the edge group $\Gamma_{\xi_i}$ of $\xi_i$ is the infinite cyclic generated by the abstract generator $e_i$. If $\xi_i$, $i\le m_0$, is a loop in $\mathcal{G}$, then we a need extra generator $e_i ^\perp$. For each edge $\xi_i$, we have two homomorphisms $(\cdot)^\pm:\Gamma_{\xi_i} \to \Gamma_{(\xi_i)^\pm}$ where $(\xi_i)^+$ and $(\xi_i)^-$ are the terminal and the initial vertex of $\xi_i$ respectively. All these generators are subject to the following relations
\begin{equation}\label{rel1}
e_i^+ = e_i ^-
\end{equation}
and
\begin{equation}\label{rel2}
e_i^ \perp e_i ^+ (e_i^\perp)^{-1} = e_i^-. 
\end{equation}
Since $\Gamma_{\mathcal{O}_j}$ is isomorphic to $\pi_1 ^{\operatorname{orb}}(\mathcal{O}_j)$, we can choose a presentation of each $\Gamma_{\mathcal{O}_i}$
\begin{multline*}
\Gamma_{\mathcal{O}_j} = \langle x_{j,1}, y _{j,1}, \cdots, x_{j, g_j}, y_{j, g_j}, z_{j,1}, \cdots, z_{j,b_j}, s_{j,1}, \cdots, s_{j, c_j}\,|\, \\\prod_{i=1} ^{g_j} [x_{j,i},y_{j,i}]\prod_{i=1} ^{b_j} z_{j,i} \prod_{i=1} ^{c_j } s_{j,i}= s_{j,1}^{r_{j,1}}=\cdots s_{j,c_j} ^{r_{j, c_j}}=1\rangle
\end{multline*}
such that
\begin{itemize}
\item the homomorphisms $\Gamma_{\xi_a}\to \Gamma_{\xi_a ^\pm}$ at an edge $\xi_a$ connecting $\mathcal{O}_p< \mathcal{O}_q$ are given by $e_a^+ =  z_{q,i}$ and $e_a ^- = z_{p,j}^{-1}$  for some $i,j$. 
\item if $\xi_i$ is in $\mathcal{O}_k$ then the lift of $\xi_i$ in the universal cover joins the fixed points of $s_{k,i}$ and $s_{k,i+1}$ for some $i$. 
\end{itemize}
To make our notation consistent, we formally define $e_i^+$, $i>m_0$,  to be $z_{k,b_k+1}^{-1}=s_{k,2}s_{k,1}$ if $\xi_i$ is a full 1-suborbifold in $\mathcal{O}_k$ whose lift in the universal cover joins the fixed points of $s_{k,1}$ and $s_{k,2}$.

The invariant functions $\mathbf{\ell}$ and $\mathbf{m}$ induce the smooth functions $\ell_i,\mathbf{m}_i$ on $\mathcal{C}^{\mathscr{B}}(\mathcal{O})$ defined by $\mathbf{\ell}_i([\rho]) := \ell(\rho(\xi_i))$, $i=1,2,\cdots,m$ and  $\mathbf{m}_i([\rho]):=\mathbf{m}(\rho(\xi_i))$ for each $i=1,2,\cdots, m_0$.  Now following Goldman \cite[Theorem 4.3]{goldman1986} and Choi-Jung-Kim \cite[Theorem 4.5.6]{choi2019}, we define the flows $L_t ^i$ and $M_t ^i$ along $\xi_i$ for $i=1,2,\cdots, m_0$. The flow along the principal full 1-suborbifold $\xi_i$, $i>m_0$ is something new. For readers' convenience, we present their definitions.

\emph{-- Case I: $\xi_i$ is a simple closed curve and is in $\mathcal{T}$. } Suppose that $\xi_i$ joins $\mathcal{O}_p$ and $\mathcal{O}_q$, $\mathcal{O}_p<\mathcal{O}_q$. Then for each $t\in \R$, we define $L_t ^i$ by 
\[
L^i _t ([\rho])(x) =\begin{cases}
\rho(x) & \text{if } x\in \Gamma_{\mathcal{O}_r}, \mathcal{O}_r \not\ge \mathcal{O}_q \\
\exp( t \ell^{\#}(\rho(e_i^+))) \rho(x) \exp (-t \ell^{\#} (\rho(e_i^+))) & \text{if } x\in \Gamma_{\mathcal{O}_{r}}, \mathcal{O}_r \ge \mathcal{O}_q\\
\exp( t \ell^{\#}(\rho(e_i^+))) \rho(x) \exp (-t \ell^{\#} (\rho(e_i^+))) & \text{if } x=e_r^\perp \text{ and } \xi_r ^+ , \xi_r ^- \ge \mathcal{O}_q\\
\rho(x) \exp( -t \ell^{\#}(\rho(e_i^+))) & \text{if }x=e_r^{\perp} \text{ and } \xi_r ^+ \ge \mathcal{O}_q>\xi^- _r\\
\exp( t \ell^{\#}(\rho(e_i^+)))  \rho(x) & \text{if }x=e_r ^{\perp}\text{ and } \xi_r ^- \ge \mathcal{O}_q>\xi^+ _r
\end{cases}
\]
and similarly $M^i _t$ with $\ell^{\#}$ replaced by $\mathbf{m}^{\#}$. Note that $L^i _t([\rho])$ respects the relations (\ref{rel1}) and (\ref{rel2}).

\emph{-- Case II: $\xi_i$ is a simple closed curve and is not in $\mathcal{T}$. }This case we have simpler expressions
\[
L^i _t ([\rho])(x) = \begin{cases}
 \rho(x)  \exp( t \ell^{\#}(\rho(e_i^+)))  & \text{if }x=e_i^\perp \\
 \rho(x) & \text{otherwise }
\end{cases}
\]
and again similarly $M^i _t$. 

\emph{-- Case III: $\xi_i$ is  a  full 1-suborbifold. }  In this case, we only have flows of the type $L_t ^i$. Suppose that the lift of $\xi_i$ in the universal cover of $\mathcal{O}_j$ joins fixed points of order two generators $s_{j,1}$ and $s_{j,2}$. Then we understand the word $z_{j,b_j+1}$ as $s_{j,1}s_{j,2}$. For each $t\in\R$, we define
\[
L^i _t ([\rho])(x) = \begin{cases}
 \exp( t \ell^{\#}(\rho(e_j^+)))   \rho(x)  \exp( - t \ell^{\#}(\rho(e_j^+)))  & \text{if }x=s_{j,1},s_{j,2} \\
 \rho(x) & \text{otherwise }. 
\end{cases}
\]

\begin{lemma}
The flows $L^i_t$, $M^i_t$ defined above are Hamiltonian, complete and commute each other. Consequently, they give rise to the Hamiltonian action of the abelian Lie group $\R^M$.  
\end{lemma} 
\begin{proof}
Let $X_{\mathcal{O}}'$ be the surface obtained by removing all singular points from the underlying space $X_{\mathcal{O}}$ of $\mathcal{O}$. We have the natural inclusion $f: X_{\mathcal{O}}' \to X_{\mathcal{O}}$. Let $\xi_i ' := f^{-1}(\xi_i)$ for $i=1,2,\cdots, m_0$. For $i=m_0+1,\cdots,m$, consider a small tubular neighborhood $N_i\subset X_{\mathcal{O}}$ of $\xi_i$ and $N_i'=f^{-1}(\overline{N_i})$. Topologically, $N'_i$ is a disk with two punctures. We define $\xi_i'$ to be the boundary of $N_i'$. Now, we can construct the flows ${L_t ^1}', \cdots, {L_t ^m}'$ and ${M_t ^1}', \cdots, {M_t ^{m_0}}'$ on $\Rep_3(\pi_1(X_{\mathcal{O}}'))$ with respect to $\xi_i'$. By \cite[Theorem 4.5.6]{choi2019}, ${L_t ^i}'$ and ${M_t ^j}' $ are Hamiltonian and commute. 

Recall that the embedding $\mathcal{I}$ defined in Lemma \ref{smooth} is symplectic. Moreover, one can observe that $\mathcal{I} \circ  L_t ^i = {L_t ^i}' \circ \mathcal{I}$ and $\mathcal{I} \circ M_t^j = {M_t ^j}'\circ \mathcal{I}$ for $i=1,2,\cdots, m$ and $j=1,2,\cdots,m_0$. Therefore, the flows  $L^1 _t , \cdots, L^m _t$, and $M^1 _t, \cdots, M^{m_0}_t$ are also Hamiltonian and commute.

The completeness of $L^1 _t , \cdots, L^m _t$, $M^1 _t, \cdots, M^{m_0}_t$ follows from section 3.4.1 and section 3.4.2 of Choi-Goldman \cite{choi2005}. 
\end{proof}

The moment map $\mu:\mathcal{C}^{\mathscr{B}}(\mathcal{O})\to (\R^{M})^*$ of this action is given by the invariants of $\rho(\xi_i)$, that is, 
\[
[\rho] \mapsto (\ell_1([\rho]), \cdots, \ell_{m}([\rho]),\mathbf{m}_{1} ([\rho]),\cdots, \mathbf{m}_{m_0}([\rho])).
\]
Here we identify $\R^{M}$ with its dual by the canonical inner product. We briefly explain why this is the case. Since this is an abelian Lie group action, it suffices to handle the action of a single component. The moment map of the total Lie group action is given by summing up all the moments maps of each component action. Hence, let us present the case of the flow $L_t ^1:\R \times \mathcal{C}^{\mathscr{B}}(\mathcal{O})\to \mathcal{C}^{\mathscr{B}}(\mathcal{O})$ only. It is known \cite[Theorem 4.5 and Theorem 4.7]{goldman1986} that the Hamiltonian function for the flow $L_t ^1$ is the function $\ell_1:\mathcal{C}^{\mathscr{B}}(\mathcal{O})\to \R$. We know that an element $x \in \R$ (regarded as the Lie algebra of $\R$) generates the flow $L_{x t} ^1$. Then, the associated Hamiltonian function will be $[\rho]\mapsto x \ell_1([\rho])$. Thus, if we identify the dual $\R ^*$ with $\R$ itself, the moment map $\mu: \mathcal{C}^{\mathscr{B}}(\mathcal{O})\to \R^*$ for this action can be simply written as $\mu([\rho]) = \ell_1 ([\rho])$.   See section 4.5 of \cite{choi2019} for details. 

\begin{lemma}\label{RMfree}
The $\R^{M}$-action is free. 
\end{lemma}
\begin{proof}
Suppose that $m=1$ (either $m_0=1$ or $m_e=1$). In view of Choi-Goldman \cite[Proposition 3.11]{choi2005}, we know that the $\R^M$-action is free. The lemma follows by induction on $m$. 
\end{proof}

Suppose that we are given a value $y$ in the image of the moment map $\mu$. We observe  two elements $g$ and $h$ in $\mathbf{Hyp}^+$ are conjugate if and only if $(\ell(g),\mathbf{m}(g))=(\ell(h),\mathbf{m}(h))$. Therefore, the value of $\mu$ determines the conjugacy classes of $\rho(e_i^+)$ say $C_i$. That is, $\mu^{-1}(y)$ is the set of classes of representations $[\rho]$ in $\mathcal{C}^{\mathscr{B}}(\mathcal{O})$ such that $\rho(e_i^+)\in C_i$ for each $i=1,2,\cdots, m$. This also defines a set of conjugacy class $\mathscr{B}_i$ for each $i=1,2,\cdots,l$ such that the image of the projection $\mathcal{C}^{\mathscr{B}}(\mathcal{O}) \to \mathcal{C}(\mathcal{O}_i)$ lies in $\mathcal{C}^{\mathscr{B}_i }(\mathcal{O}_i)$.

For each $y$ in the image of the moment map $\mu$, the splitting map $\mathcal{SP}$ induces the map $\mathcal{SP}_y: \mu^{-1}(y) \to \mathcal{C}^{\mathscr{B}_1}(\mathcal{O}_1) \times \cdots \times \mathcal{C}^{\mathscr{B}_l}(\mathcal{O}_l)$. Recall that, in terms of holonomy, $\mathcal{SP}_y$ is given by
\[
[\rho] \mapsto ([\rho\circ \iota_1], [\rho\circ \iota_2],\cdots, [\rho\circ \iota_l ]).
\]
We observe that $\mathcal{SP}_y$ descends to the map
\[
\mathcal{SP}'_y: \mu^{-1}(y) / \R^{M} \to  \mathcal{C}^{\mathscr{B}_1}(\mathcal{O}_1) \times \cdots \times \mathcal{C}^{\mathscr{B}_l}(\mathcal{O}_l). 
\]

\begin{lemma}\label{proper1}
For every $y\in \operatorname{Image}(\mu)$, the $\R^{M}$-action on $\mu^{-1}(y)$ is proper. 
\end{lemma}

This lemma is implicitly shown in Choi-Goldman \cite{choi2005}. To be more precise we include two proofs for this lemma. The first proof is indirect; we prove that $\mu^{-1}(y)$ is a fiber bundle over the image.  The second proof is a variant of the authors' previous paper \cite{choi2019} which is more intrinsic but slightly complicated. 

\begin{proof}[Proof 1]
Let $\mathcal{S}=\per{\Gamma} \cup \{\langle \xi_1\rangle, \cdots, \langle \xi_m \rangle\}$. One can identify $T_{[\rho]}\mu^{-1}(y)$ with $H^1_{\operatorname{par}}(\Gamma, \mathcal{S};\mathfrak{g}_\rho)$. There is a Mayer-Vietoris type short exact sequence
\[
0\to \bigoplus_{i=1} ^m H^0 (\langle \xi_i \rangle;\mathfrak{g})\overset{\delta}{ \to} H^1 _{\operatorname{par}}(\Gamma,\mathcal{S};\mathfrak{g}) \overset{\der \mathcal{SP}_y}{\to} \bigoplus_{i=1} ^l H^1_{\operatorname{par}}(\Gamma_i ,\per{\Gamma_i} ;\mathfrak{g})\to 0
\]
where the image of $\delta$ in $H^1 _{\operatorname{par}}(\Gamma,\mathcal{S};\mathfrak{g}) $ is the tangent directions of the $\R^M$-action.  This shows that $\mathcal{SP}_y'$ is a local diffeomorphism and that $\mathcal{SP}_y: \mu^{-1}(y) \to  \mathcal{C}^{\mathscr{B}_1}(\mathcal{O}_1) \times \cdots \times \mathcal{C}^{\mathscr{B}_l}(\mathcal{O}_l)$ is a submersion. Since $\mathcal{SP}'_y$ is known to be one-to-one and onto we know that $\mathcal{SP}_y '$ is a diffeomorphism. In particular, $\mu^{-1}(y) \to \mu^{-1}(y)/\R^M$ is also a submersion whose fiber coincides with that of $\mathcal{SP}_y$.

Since $\mu^{-1}(y)\to \mu^{-1}(y)/\R^M$ is a submersion, we know that each fiber of $\mu^{-1}(y)\to \mu^{-1}(y)/\R^M$ is an embedded submanifold. Since Lemma \ref{RMfree} shows that $\R^M$ acts simply transitively on each fiber, we know that each fiber is homeomorphic to $\R^M$. By Corollary 31 of Meigniez \cite{meigniez2002}, we know that $\mu^{-1}(y)\to \mu^{-1}(y)/\R^M$ is a fiber bundle. Hence the action is proper. 
\end{proof}

\begin{proof}[Proof 2]
Let 
\[
\mathcal{H}(y)=\{\rho\in \Hom(\pi_1(X'_{\mathcal{O}}),G)\,|\,[\rho]\in \mu^{-1}(y)\}.
\]
A slight modification of the proof of Lemma 4.5.4 of \cite{choi2019} shows that the $\R^{M}$-action on  $\mathcal{H}(y)$ is proper. 

Now we construct an equivariant section of $\mathcal{H}(y)\to  \mu^{-1}(y)$. For each $\rho\in \mathcal{H}(y)$, we have a unique limit map $\xi_\rho : \partial \widetilde{\mathcal{O}} \to \mathbb{RP}^2$. Choose four points  $x_1, \cdots, x_4 \in \partial \widetilde{\mathcal{O}}$, no two of which are in the same orbit. We also choose four generic points $v_1, \cdots, v_4$ in $\mathbb{RP}^2$. We define the section $s:  \mu^{-1}(y) \to \mathcal{H}$ by declaring that $s([\rho])$ is the unique $\rho\in \mathcal{H}(y)$ such that $\xi_\rho (x_i) = v_i$. Then $s$ is clearly the desired equivariant section. 

For a given compact set $K$ of $\mu^{-1}(y)$, we know that $\{\mathbf{t}\in \R^{M}\,|\, \mathbf{t}\cdot s(K)\cap s(K) \ne \emptyset\}$ is compact. Since $s$ is equivariant, we have  $\{\mathbf{t}\in \R^{M}\,|\, \mathbf{t}\cdot s(K)\cap s(K) \ne \emptyset\}= \{\mathbf{t}\in \R^{M}\,|\, \mathbf{t}\cdot K\cap K \ne \emptyset\}$. Therefore, the $\R^{M}$-action on $\mu^{-1}(y)$ is also proper. 
\end{proof}

Since the $\R^{M}$-action on $\mu^{-1}(y)$ is proper, we can take the Marsden-Weinstein quotient $\mu^{-1}(y)/\R^{M}$.  On the other hand, the right hand side admits the symplectic structure given by the orthogonal sum $\omega^{\mathcal{O}_1} \oplus \cdots \oplus \omega^{\mathcal{O}_l}$. 

By appealing to Theorem \ref{localdecompmain} and Proposition \ref{symplecticembedding}, we obtain the following decomposition theorem

\begin{theorem}
Let $\mathcal{O}$ be a compact oriented 2-orbifold of negative Euler characteristic. Let $\xi_1, \cdots, \xi_m$ be pairwise disjoint essential simple closed curves or full 1-suborbifolds such that the completion of each connected component $\mathcal{O}_i$ of $\mathcal{O}\setminus \bigcup_{i=1} ^m  \xi_i$ has also negative Euler characteristic. Then the map
\[
\mathcal{SP}'_y: \mu^{-1}(y) / \R^{M} \to  \mathcal{C}^{\mathscr{B}_1}(\mathcal{O}_1) \times \cdots \times \mathcal{C}^{\mathscr{B}_l}(\mathcal{O}_l)
\]
defined above is a symplectomorphism.
\end{theorem}
\begin{proof}
From the proof of Lemma \ref{proper1}, we can conclude that $\mathcal{SP}_y '$ is a diffeomorphism. 

Since $\mathcal{SP}_y'$ is induced from $\iota_i:\Gamma_i\to \Gamma$, Theorem \ref{localdecompmain} tells us that $\mathcal{SP}_y'$ is a symplectomorphism. 
\end{proof}
\subsection{Construction of global coordinates}
To make upcoming indices neat, we introduce the following notation.

\begin{notation}
Let $\mathcal{O}$ be a closed oriented 2-orbifold of negative Euler characteristic, with genus $g$,  $c$ cone points and $c_b$ order two cone points. Let
\[
I=I(g,c,c_b):=\begin{cases} 
2g-2+c-\lfloor c_b/2 \rfloor & \text{if }c_b\text{ is even}\\
2g-2+c-\lfloor c_b/2 \rfloor-1 & \text{if }c_b\text{ is odd}
\end{cases},
\]
\[
J=J(g,c,c_b):= 3g-3+c-\lfloor c_b/2 \rfloor
\]
and 
\[
K=K(g,c):= 3g-3+c.
\] 
\end{notation}
We prove that $\mathcal{O}$ can be decomposed into so called elementary orbifolds of types (P1),(P2),(P3) and (P4). This is analogues to the pair-of-pants decomposition of a surface but we require more building blocks. 

For the definition of elementary orbifolds, we refer readers to Choi-Goldman \cite{choi2005}. 

\begin{lemma}\label{pantsdecom}
There is a decomposition of $\mathcal{O}$ by essential simple closed curves $\xi_1, \cdots, \xi_J$ and full 1-suborbifolds $\xi_{J+1},\cdots, \xi_K$ such that each connected component of $\mathcal{O}\setminus\bigcup_{i=1} ^{K} \xi_i$ is an elementary orbifold of the type (P1), (P2), (P3), or (P4). 
\end{lemma}
\begin{proof}
This is essentially proven in Theorem 4.3  of Choi-Goldman \cite{choi2005}. Here, we give simpler argument specialized to our situation. 

If $\mathcal{O}$ has order 2 cone points, we pair them by full 1-suborbifolds $\xi_{J+1},\cdots, \xi_{K}$.  By splitting $\mathcal{O}$ along $\xi_{J+1},\cdots, \xi_K$, we get another orbifold $\mathcal{O}'$ of the same genus $g$ with $\lfloor c_b/2 \rfloor$ boundary components, $c-2\lfloor c_b/2 \rfloor$ cone points and at most one order two cone point.  Let $\mathscr{C}=\{\xi_1,\cdots, \xi_J\}$ be a pair-of-pants decomposition of $X' _{\mathcal{O}'}$.  Then it is clear that $\mathscr{C}$ decomposes $\mathcal{O}$ into $2g-2+c-\lfloor c_b/2 \rfloor$ connected components $E_1, \cdots, E_{2g-2+c-\lfloor c_b/2 \rfloor}$ each of which is an elementary orbifold of the type (P1), (P2), (P3), or (P4). 
\end{proof}

If an elementary orbifold $E$ of type (P1), (P2), (P3) or (P4) has an order 2 cone point, then the deformation space $\mathcal{C}^{\mathscr{B}}(E)$ becomes singleton. Call such an elementary orbifold \emph{exceptional}. Every non-exceptional elementary orbifold $E$ of type (P1), (P2), (P3) or (P4) has two dimensional deformation space $\mathcal{C}^{\mathscr{B}}(E)$. Observe that among the pieces $E_i$ of Lemma \ref{pantsdecom}, there appears at most one exceptional elementary orbifold depending on the parity of $c_b$.

There is the smooth map $\mathcal{SP}: \mathcal{C}(\mathcal{O}) \to \triangle$ given by the restriction on each component  where $\triangle$ is the subset of  $\mathcal{C}(E_1) \times \cdots \times \mathcal{C}(E_{2g-2+c-\lfloor c_b/2 \rfloor})$ which satisfies the condition that if $\gamma_i$ and $\gamma_j$ are gluing boundaries, then they have the same invariants. Note that $\mathcal{SP}$ restricts on each leaf $\mu^{-1}(y)$ to $\mathcal{SP}_y$ . 

Our first coordinates are called \emph{length coordinates} 
\[
\mu=(\ell_1,\cdots,\ell_K,\mathbf{m}_1, \cdots,\mathbf{m}_J). 
\]

\begin{lemma}
The principal $\R^M$-fiber bundle map $\mathcal{SP}: \mathcal{C}(\mathcal{O}) \to \triangle$ admits a global section, say $s: \triangle\to \mathcal{C}(\mathcal{O})$ where $M=2J+\lfloor c_b/2 \rfloor=6g-6+2c-\lfloor c_b/2 \rfloor$. 
\end{lemma}
\begin{proof}
We know, from Choi-Goldman \cite{choi2005}, that 
\[
\triangle=(\R_+^2)^I\times \mathcal{L}^{J}\times (\R_+)^{K-J}
\]
where $\mathcal{L}= \{(l,m)\in \R^2\,|\,l>0,\,3m+l>0\}$. Indeed, there are parameters $(s_i,t_i)$ defined in Choi-Goldman \cite{choi2005} that parametrize $\mathcal{C}^{\mathscr{B}_i}(E_i)$ when $E_i$ is non-exceptional. This $(s_i,t_i)$ parameters take all values in $\R_+^2=\{(x,y)\in \R^2\,|\, x,y> 0\}$. They correspond to the first factor $(\R_+^2)^I$. The second factor comes from the range of invariants $(\ell,\mathbf{m})$ associated to simple closed geodesics. The last factor is the range of $\ell$ associated to principal full 1-suborbifolds. 

Since, $\triangle$ is contractible, $\mathcal{SP}$ is the trivial $\R^M$-bundle. 
\end{proof}

\begin{remark}
In Choi-Goldman \cite{choi2005}, the parameter $s$ is defined merely as a parameter for the solution space of the system of linear equations \cite[eqns. (23), (24)]{choi2005} and is not specified explicitly. In this paper we use $\gamma_1$ (eq. (20) in \cite{choi2005}) as a free parameter for this system and declare this variable $s$. 
\end{remark}

The \emph{twisting coordinates} $\theta_1,\cdots, \theta_{K}, \phi_1,\cdots, \phi_{J}$ are defined by integration from the image of $s$ along the Hamiltonian flow.  We finally define \emph{internal coordinates}
\[
\mathbf{s}_i := \log  s_i \circ \mathcal{SP}
\]
and 
\[
\mathbf{t}_i := \log t_i\circ \mathcal{SP}.
\] 

Altogether, we get the global coordinates for $\mathcal{C}(\mathcal{O})$. Observe that there are $16g-16+6c-2c_b$ coordinates, which is equal to the dimension of $\mathcal{C}(\mathcal{O})$. These coordinates, however, may not be Darboux. 

\subsection{Existence of global Darboux coordinates}
We now show that $\mathcal{C}(\mathcal{O})$ admits global Darboux coordinates. We begin with the global coordinates constructed in the previous section.

\begin{theorem}\label{symplectic}
Let $E_i$ be a non-exceptional elementary orbifold of the type (P1), (P2), (P3), or (P4). Let $(\mathbf{s}_i,\mathbf{t}_i)$ be the coordinates for $\mathcal{C}^{\mathscr{B}_i}(E_i)$.  Then
\[
\omega^{E_i} = \der \mathbf{t}_i \wedge \der \mathbf{s}_i.
\]
\end{theorem}
\begin{proof}
We use Mathematica to evaluate the formula for $\omega^{E_i}$ (Theorem \ref{explicit}).  Mathematica files are available at  \cite{jung2020}. 
\end{proof}

\begin{remark}
For (P1), a pair-of-pants case, the same result is already shown by H. Kim \cite[Theorem 5.8]{kim1999}.
\end{remark}
By following the arguments of Choi-Jung-Kim \cite{choi2019} directly, one can obtain the parallel results. We list here some of key facts without proofs. 

\begin{lemma}\label{form}
For each $i$, the Hamiltonian vector field $\mathbb{X}_{\mathbf{s}_i}$ is of the form
\[
\mathbb{X}_{\mathbf{s}_i} = \frac{\partial}{\partial \mathbf{t}_i}+\sum_{j=1} ^{K} a_j \frac{\partial}{\partial \theta_j} +\sum_{j=1} ^{J} b_j \frac{\partial}{\partial \phi_j}
\]
for some smooth functions $a_j$ and $b_j$.  
\end{lemma}

\begin{lemma}
For each $i$, the Hamiltonian vector field $\mathbb{X}_{\mathbf{s}_i}$ is complete and 
\[
\mathbb{X}_{\mathbf{s}_1}, \cdots, \mathbb{X}_{\mathbf{s}_{I}}, \mathbb{X}_{\ell_1},\cdots,\mathbb{X}_{\ell_{K}}, \mathbb{X}_{\mathbf{m}_1}, \cdots, \mathbb{X}_{\mathbf{m}_{J}}
\]
span a Lagrangian subspace of $T_{[\rho]}\mathcal{C}(\mathcal{O})$ at each $[\rho]$. 
\end{lemma}

As a consequence, the above Hamiltonian vector fields induces the Hamiltonian $\R^{I+J+K}$-action on $\mathcal{C}(\mathcal{O})$. Let us define the function $F: \mathcal{C}(\mathcal{O}) \to \R^{I+J+K}$ by
\begin{equation}\label{actioncoord}
F([\rho])= (\mathbf{s}_1([\rho]), \cdots, \mathbf{s}_{I}([\rho]), \ell_1([\rho]), \cdots, \ell_{K}([\rho]), \mathbf{m}_1([\rho]),\cdots, \mathbf{m}_{J}([\rho])).
\end{equation}
Then the $\R^{I+J+K}$-action preserves each fiber of $F$. To make use of Theorem 3.4.5 of \cite{choi2019}, we should have the following result

\begin{lemma}
The $\R^{I+J+K}$-action on $\mathcal{C}(\mathcal{O})$ is free, proper. In addition, the action is transitive on each fiber. Therefore $F$ is a fiber bundle such that each fiber is Lagrangian and diffeomorphic to $\R^{I+J+K}$.
\end{lemma}
\begin{proof}
By Lemma \ref{form}, the $\R^{I+J+K}$-action is free and fiberwise transitive. Moreover, $F$ is a submersion since $F$ is a projection with respect to the Goldman coordinates. Therefore, each fiber of $F$ is homeomorphic to $\R^{I+J+K}$.  Now we apply Corollary 31 of \cite{meigniez2002}. 
\end{proof}

Now we can deduce our main theorem:
\begin{theorem}
Let $\mathcal{O}$ be a closed oriented 2-orbifold of negative Euler characteristic, with genus $g$ and $c$ cone points. Then $\mathcal{C}(\mathcal{O})$ admits a global Darboux coordinates system.
\end{theorem}
\begin{proof}
Apply Theorem 3.4.5 of \cite{choi2019} to the function $F$ defined in (\ref{actioncoord}). Then there are complementary coordinates
\[
(\overline{\mathbf{s}}_1, \cdots, \overline{\mathbf{s}}_{I}, \overline{\ell}_1, \cdots, \overline{\ell}_{K},\overline{\mathbf{m}}_1,\cdots, \overline{\mathbf{m}}_{J})
\]
such that
\[
\omega^{\mathcal{O}}=\sum_{i=1}^{I(g,c,c_b)} \der \mathbf{s}_i \wedge\der \overline{\mathbf{s}}_i + \sum_{i=1} ^{K(g,c)} \der \ell_i \wedge \der \overline{\ell}_i+ \sum_{i=1} ^{J(g,c,c_b)} \der \mathbf{m}_i \wedge \der \overline{\mathbf{m}}_i. 
\]
This completes the proof.
\end{proof}

\bibliographystyle{amsplain}
\bibliography{references_symp}

\providecommand{\bysame}{\leavevmode\hbox to3em{\hrulefill}\thinspace}
\providecommand{\MR}{\relax\ifhmode\unskip\space\fi MR }
\providecommand{\MRhref}[2]{%
  \href{http://www.ams.org/mathscinet-getitem?mr=#1}{#2}
}
\providecommand{\href}[2]{#2}
\begin{thebibliography}{10}

\bibitem{ALR}
Alejandro Adem, Johann Leida, and Yongbin Ruan, \emph{Orbifolds and stringy
  topology}, Cambridge Tracts in Mathematics, vol. 171, Cambridge University
  Press, Cambridge, 2007. \MR{2359514}

\bibitem{alessandrini2018}
Daniele Alessandrini, Gye-Seon Lee, and Florent Schaffhauser, \emph{{H}itchin
  components for orbifolds}, arXiv preprint arXiv:1811.05366 (2018), To appear
  in J. EMS (online first, https://ems.press/journals/jems/articles/4629627).

\bibitem{bieri1978}
Robert Bieri and Beno Eckmann, \emph{Relative homology and {P}oincar\'{e}
  duality for group pairs}, J. Pure Appl. Algebra \textbf{13} (1978), no.~3,
  277--319. \MR{509165}

\bibitem{bridson1999}
Martin~R. Bridson and Andr\'{e} Haefliger, \emph{Metric spaces of non-positive
  curvature}, Grundlehren der Mathematischen Wissenschaften [Fundamental
  Principles of Mathematical Sciences], vol. 319, Springer-Verlag, Berlin,
  1999. \MR{1744486}

\bibitem{brown1982}
Kenneth~S. Brown, \emph{Cohomology of groups}, Graduate Texts in Mathematics,
  vol.~87, Springer-Verlag, New York-Berlin, 1982. \MR{672956}

\bibitem{choi2012}
Suhyoung Choi, \emph{Geometric structures on 2-orbifolds: exploration of
  discrete symmetry}, MSJ Memoirs, vol.~27, Mathematical Society of Japan,
  Tokyo, 2012. \MR{2962023}

\bibitem{choi1993}
Suhyoung Choi and William~M. Goldman, \emph{Convex real projective structures
  on closed surfaces are closed}, Proc. Amer. Math. Soc. \textbf{118} (1993),
  no.~2, 657--661. \MR{1145415}

\bibitem{choi2005}
\bysame, \emph{The deformation spaces of convex {$\Bbb{RP}^2$}-structures on
  2-orbifolds}, Amer. J. Math. \textbf{127} (2005), no.~5, 1019--1102.
  \MR{2170138}

\bibitem{choi2019}
Suhyoung Choi, Hongtaek Jung, and Hong~Chan Kim, \emph{Symplectic coordinates
  on {$\rm PSL_3(\Bbb R)$}-{H}itchin components}, Pure Appl. Math. Q.
  \textbf{16} (2020), no.~5, 1321--1386. \MR{4220999}

\bibitem{fox1953}
Ralph~H. Fox, \emph{Free differential calculus. {I}. {D}erivation in the free
  group ring}, Ann. of Math. (2) \textbf{57} (1953), 547--560. \MR{0053938}

\bibitem{fox1954}
\bysame, \emph{Free differential calculus. {II}. {T}he isomorphism problem of
  groups}, Ann. of Math. (2) \textbf{59} (1954), 196--210. \MR{62125}

\bibitem{goldman1984}
William~M. Goldman, \emph{The symplectic nature of fundamental groups of
  surfaces}, Adv. in Math. \textbf{54} (1984), no.~2, 200--225. \MR{762512}

\bibitem{goldman1986}
\bysame, \emph{Invariant functions on {L}ie groups and {H}amiltonian flows of
  surface group representations}, Invent. Math. \textbf{85} (1986), no.~2,
  263--302. \MR{846929}

\bibitem{goldman1987}
\bysame, \emph{Geometric structures on manifolds and varieties of
  representations}, Geometry of group representations ({B}oulder, {CO}, 1987),
  Contemp. Math., vol.~74, Amer. Math. Soc., Providence, RI, 1988,
  pp.~169--198. \MR{957518}

\bibitem{goldman1990}
\bysame, \emph{Convex real projective structures on compact surfaces}, J.
  Differential Geom. \textbf{31} (1990), no.~3, 791--845. \MR{1053346}

\bibitem{guichard2012}
Olivier Guichard and Anna Wienhard, \emph{Anosov representations: domains of
  discontinuity and applications}, Invent. Math. \textbf{190} (2012), no.~2,
  357--438. \MR{2981818}

\bibitem{guruprasad1997}
K.~Guruprasad, J.~Huebschmann, L.~Jeffrey, and A.~Weinstein, \emph{Group
  systems, groupoids, and moduli spaces of parabolic bundles}, Duke Math. J.
  \textbf{89} (1997), no.~2, 377--412. \MR{1460627}

\bibitem{hitchin1992}
N.~J. Hitchin, \emph{Lie groups and {T}eichm\"{u}ller space}, Topology
  \textbf{31} (1992), no.~3, 449--473. \MR{1174252}

\bibitem{huebschmann1979}
Johannes Huebschmann, \emph{Cohomology theory of aspherical groups and of small
  cancellation groups}, J. Pure Appl. Algebra \textbf{14} (1979), no.~2,
  137--143. \MR{524183}

\bibitem{huebschmann1995}
\bysame, \emph{Symplectic and {P}oisson structures of certain moduli spaces.
  {I}}, Duke Math. J. \textbf{80} (1995), no.~3, 737--756. \MR{1370113}

\bibitem{huebschmann1995b}
\bysame, \emph{Symplectic and {P}oisson structures of certain moduli spaces.
  {II}. {P}rojective representations of cocompact planar discrete groups}, Duke
  Math. J. \textbf{80} (1995), no.~3, 757--770. \MR{1370114}

\bibitem{jung2020}
Hongtaek Jung, \emph{Personal website,
  \href{https://sites.google.com/view/htjung/home}{https://sites.google.com/view/htjung/home}}.

\bibitem{karshon1992}
Yael Karshon, \emph{An algebraic proof for the symplectic structure of moduli
  space}, Proc. Amer. Math. Soc. \textbf{116} (1992), no.~3, 591--605.
  \MR{1112494}

\bibitem{kim1999}
Hong~Chan Kim, \emph{The symplectic global coordinates on the moduli space of
  real projective structures}, J. Differential Geom. \textbf{53} (1999), no.~2,
  359--401. \MR{1802726}

\bibitem{kim2001}
Inkang Kim, \emph{Rigidity and deformation spaces of strictly convex real
  projective structures on compact manifolds}, J. Differential Geom.
  \textbf{58} (2001), no.~2, 189--218. \MR{1913941}

\bibitem{labourie2006}
Fran\c{c}ois Labourie, \emph{Anosov flows, surface groups and curves in
  projective space}, Invent. Math. \textbf{165} (2006), no.~1, 51--114.
  \MR{2221137}

\bibitem{lyndon1950}
Roger~C. Lyndon, \emph{Cohomology theory of groups with a single defining
  relation}, Ann. of Math. (2) \textbf{52} (1950), 650--665. \MR{47046}

\bibitem{majumdar1970}
Subrata Majumdar, \emph{A free resolution for a class of groups}, J. London
  Math. Soc. (2) \textbf{2} (1970), 615--619. \MR{276310}

\bibitem{meigniez2002}
Ga\"{e}l Meigniez, \emph{Submersions, fibrations and bundles}, Trans. Amer.
  Math. Soc. \textbf{354} (2002), no.~9, 3771--3787. \MR{1911521}

\bibitem{porti2020}
Joan Porti, \emph{Dimension of representation and character varieties for two
  and three-orbifolds}, arXiv preprint, arXiv:2009.03124 (2020).

\bibitem{sikora2012}
Adam~S. Sikora, \emph{Character varieties}, Trans. Amer. Math. Soc.
  \textbf{364} (2012), no.~10, 5173--5208. \MR{2931326}

\bibitem{swz2017}
Zhe Sun, Anna Wienhard, and Tengren Zhang, \emph{Flows on the {${\rm
  PGL}(V)$}-{H}itchin component}, Geom. Funct. Anal. \textbf{30} (2020), no.~2,
  588--692. \MR{4108617}

\bibitem{sz2017}
Zhe Sun and Tengren Zhang, \emph{The {G}oldman symplectic form on the {$PSL
  (V)$}-{H}itchin component}, arXiv preprint arXiv:1709.03589 (2017).

\bibitem{thurston1979}
William~P Thurston, \emph{The geometry and topology of three-manifolds},
  Princeton University Princeton, NJ, 1979.

\bibitem{weibel1994}
Charles~A. Weibel, \emph{An introduction to homological algebra}, Cambridge
  Studies in Advanced Mathematics, vol.~38, Cambridge University Press,
  Cambridge, 1994. \MR{1269324}

\bibitem{weil1964}
Andr\'{e} Weil, \emph{Remarks on the cohomology of groups}, Ann. of Math. (2)
  \textbf{80} (1964), 149--157. \MR{0169956}

\bibitem{weinstein1995}
A.~Weinstein, \emph{The symplectic structure on moduli space}, The {F}loer
  memorial volume, Progr. Math., vol. 133, Birkh\"{a}user, Basel, 1995,
  pp.~627--635. \MR{1362845}

\bibitem{wolpert1985}
Scott Wolpert, \emph{On the {W}eil-{P}etersson geometry of the moduli space of
  curves}, Amer. J. Math. \textbf{107} (1985), no.~4, 969--997. \MR{796909}

\end{thebibliography}
\end{document}